\documentclass[10pt,a4paper]{amsart}
\usepackage[british]{babel}
\usepackage{amscd}
\usepackage{amssymb}
\usepackage{enumerate}
\usepackage{xypic}
\usepackage{verbatim} 
\usepackage{stmaryrd}

\usepackage{hyperref}

\usepackage{nameref}

\makeatletter
\let\orgdescriptionlabel\descriptionlabel
\renewcommand*{\descriptionlabel}[1]{%
  \let\orglabel\label
  \let\label\@gobble
  \phantomsection
  \edef\@currentlabel{#1}%
  \let\label\orglabel
  \orgdescriptionlabel{#1}%
}
\makeatother

\newtheorem{theorem}{Theorem}[section]
\newtheorem{lemma}[theorem]{Lemma}
\newtheorem{coro}[theorem]{Corollary}
\newtheorem{prop}[theorem]{Proposition}

\theoremstyle{definition}
\newtheorem{defi}[theorem]{Definition}
\newtheorem{exem}[theorem]{Example}

\newtheorem{nota}[theorem]{Notation}

\theoremstyle{remark}
\newtheorem{rema}[theorem]{Remark}

\numberwithin{equation}{section}
\newcommand{\p}{\mathbb{P}}
\newcommand{\C}{\mathcal{C}}
\newcommand{\E}{\mathcal{E}}
\newcommand{\K}{\mathcal{K}}
\newcommand{\XX}{\mathcal{X}}
\newcommand{\HH}{\mathcal{H}}
\newcommand{\CC}{\mathbb{K}}
\newcommand{\N}{\mathcal{N}}
\newcommand{\GG}{\mathbb{G}}
\newcommand{\TTT}{\mathcal{T}}
\newcommand{\GR}{\mathbb{G}}
\newcommand{\TT}{\mathbb{T}}
\newcommand{\OO}{\mathcal{O}}
\newcommand{\II}{\mathcal{I}}
\newcommand{\FF}{\mathcal{F}}

\newcommand{\UU}{\mathcal{U}}
\newcommand{\QQ}{\mathcal{Q}}
\newcommand{\w}{\wedge}
\newcommand{\bw}{\bigwedge}
\newcommand{\gen}[1]{\langle#1\rangle}
\DeclareMathOperator{\Proj}{Proj}
\DeclareMathOperator{\Hom}{Hom}
\DeclareMathOperator{\rk}{rank}
\DeclareMathOperator{\GL}{GL}

\DeclareMathOperator{\Sing}{Sing}
\DeclareMathOperator{\SL}{SL}
\DeclareMathOperator{\pf}{Pfaff}
\DeclareMathOperator{\Sym}{Sym}
\DeclareMathOperator{\codim}{codim}
\DeclareMathOperator{\im}{Im}
\DeclareMathOperator{\multdeg}{multdeg}
\DeclareMathOperator{\G2}{G}
\DeclareMathOperator{\Spin}{Spin}
\DeclareMathOperator{\Sp}{Sp}
\DeclareMathOperator{\Hilb}{Hilb}
\DeclareMathOperator{\Pic}{Pic}
\def\EExt{\mathcal{E}xt}

\begin{document}

\sloppy

\title{Fano congruences of index $3$ and alternating $3$-forms}


\author{Pietro De Poi}
\address{Dipartimento di Scienze Matematiche, Informatiche e Fisiche \\
Universit\`a degli Studi di Udine  \\
Via delle Scienze 206 \\
Localit\`a Rizzi \\
33100 Udine \\
Italy}
\email{pietro.depoi@uniud.it}

\author{Daniele Faenzi}
\address{Universit\'e de Bourgogne\\
Institut de Math\'ematiques de Bourgogne\\
UMR CNRS 5584\\
UFR Sciences et Techniques -- B\^atiment Mirande -- Bureau 310\\
9 Avenue Alain Savary \\ 
BP 47870 21078 Dijon Cedex \\ 
France}
\email{daniele.faenzi@u-bourgogne.fr}

\author{Emilia Mezzetti}
\address{Dipartimento di Matematica e Geoscienze\\
Sezione di Matematica e Informatica\\
Universit\`a degli Studi di Trieste\\
Via Valerio 12/1\\
34127 Trieste\\ 
Italy}
\email{mezzette@units.it}

\author{Kristian Ranestad}
\address{Department of Mathematics\\
University of Oslo\\
P.O. Box 1053 Blindern\\
NO-0316 Oslo\\
Norway}
\email{ranestad@uio.no}
\thanks{P.D.P. \& E.M. are members of INdAM - GNSAGA and are supported by PRIN
``Geometria delle variet\`a algebriche''; E.M. is also supported by
FRA, Fondi di Ricerca di Ateneo, Universit\`a di Trieste.
D.F. partially supported by French National Research Agency Grant "BirPol",
ANR-11-JS01-004-01. K.R. partially supported by RCN project no 239015 ``Special Geometries''.}



\begin{abstract} We study congruences of lines $X_\omega$ defined by a
  sufficiently 
  general choice of an alternating 3-form $\omega$ in $n+1$ dimensions,
  as Fano manifolds of 
  index $3$ and dimension $n-1$. These congruences include the
  $\G2_2$-variety for $n=6$ and the variety
  of reductions of projected $\p^2 \times \p^2$ for $n=7$.
  
  We compute the degree
  of $X_\omega$ as the $n$-th Fine number and study the Hilbert scheme
  of these congruences proving that the choice of $\omega$ bijectively
  corresponds to $X_\omega$ except when $n=5$.
  The fundamental locus of the congruence is also studied together
  with its singular locus: these varieties include the Coble cubic for $n=8$ and
  the Peskine variety for $n=9$.

  The residual congruence $Y$ of $X_\omega$ with respect to a general
  linear congruence containing $X_\omega$ is analysed in terms of the
  quadrics containing the linear span of $X_\omega$. We prove that $Y$
  is Cohen-Macaulay but non-Gorenstein in codimension $4$.
  We also examine the fundamental locus $G$ of $Y$ of which we
  determine the singularities and the
  irreducible components.
\end{abstract}

\maketitle

\tableofcontents


\section{Introduction}\label{sect:intro}
Let $V$ be an $(n+1)$-dimensional $\CC$-vector space, and let $\GR:=G(2,V)\subset\p(\bw^2V)$ 
be the Grassmannian of $2$-dimensional $\CC$-vector subspaces of $V$, or, equivalently,  the Grassmannian of lines in $\p^n=\p(V)$.
 A \emph{congruence of lines} in the projective space $\p(V)$ is an $(n-1)$-dimensional closed subvariety of $\GG$.  
A \emph{linear congruence} is a congruence formed by the 
proper intersection of $\GG$ with a linear space of codimension $n-1$ 
in $\p(\bw^2V)$.  In this paper we consider congruences of lines that are proper components of linear congruences.  
A general $3$-form $\omega\in \bw^3 V^*$  
defines a linear subspace $\Lambda_\omega:=\{[L]\in \p(\bw^2V)\mid \omega(L)=0\}$ of codimension $n+1$ in $\p(\bw^2V)$.  
The intersection $X_\omega:=\GG\cap \Lambda_{\omega}$ is a congruence of lines, the  first main object of this paper.
A general $2$-dimensional subspace $\gen{x,y}\subset V^*$ defines a codimension $n-1$ linear subspace $\Lambda^{xy}_{\omega}\subset\p(\bw^2V)$, 
that contains $\Lambda_\omega$.  The intersection  
\begin{equation*}
\Lambda^{xy}_{\omega}\cap \GG=X_\omega\cup Y_{\omega,x\w y}
\end{equation*}
is a reducible linear congruence, with one component $X_\omega$ and the other component a congruence $Y_{\omega,x\w y}$.  
The congruence $Y_{\omega,x\w y}$---called the \emph{residual congruence}, and sometimes denoted by $Y$ for simplicity---is 
the second main object of this paper.

 The \emph{order} of a congruence is the number of lines in the congruence that pass through a general point in $\p(V)$. 
More generally, the \emph{multidegree} of a congruence is the sequence of the coefficients in the expression 
of the congruence as a linear combination of Schubert cycles in the Chow ring of $\GG$: the \emph{$i$-th multidegree} 
$i=0,\dotsc, \left[\frac{n-1}{2}\right],$ is the number of lines contained in a general $\p^{n-i}\subset\p^n$ 
that intersect a general $\p^i$ contained in $\p^{n-i}$, with $i<n-i$.    
A linear congruence has order one.
Since $X_\omega$ and  $Y_{\omega,x\w y}$ are components of a linear congruence, one of them has order one, the other has order zero.  
The order is the first component of the multidegree referred to above (i.e. the $0$-th multidegree), 
so $X_\omega$ has order one when $n$ is even, while $Y_{\omega,x\w y}$ has order one when $n$ is odd. 
 
The \emph{fundamental locus} of a congruence 
is the locus of points in $\p(V)$ through which there are infinitely many lines  
of the congruence.    

Congruences of lines of order one appear naturally in several interesting problems in geometry, and thus motivated this study.  These include the classification of varieties with one apparent double point (\cite{CMR04}, \cite{CR11}), 
the degree of irrationality of general hypersurfaces (\cite{BCDeP14}, \cite{BDeP15}) 
and in hyperbolic conservation laws, so called Temple systems of partial differential equations (\cite{AF01}, \cite{DePM05})).  
For a survey of order one congruences of lines, see \cite{DePM07}.

 A well-known fact is that while the general linear congruence in $\GG$ is a Fano variety of index $2$, the congruence $X_\omega$ is a Fano variety of index $3$ (Theorem \ref{thm:qq2}).
 Varieties $X_{\omega}$ for small $n$  have been studied by many authors, 
both by the construction from a $3$-form as above, and by other constructions:  When $n=3,4,5,6$  and $\omega$ is general, then $X_\omega$ is a plane, 
a quadric threefold, the Segre product $\p^2\times \p^2$ 
and the closed orbit of the Lie group $\G2_2$ in its adjoint representation, respectively.   
For the next values of $n$ there are more recent studies 
by Peskine in \cite{P2} in the cases $n=7,9$, by Gruson and Sam in \cite{GS15} in the case $n=8$ 
and by Han in \cite{Han} in the case $n=9$.

In this paper we present and prove general properties, some well-known and some new, of the congruences $X_\omega$ and $Y_{\omega,x\w y}$, for sufficiently general $\omega$, $x$ and $y$.   After presenting equations \ref{eq:importantbis}, and locally free resolutions of their ideals (Theorems \ref{thm:qq2} and \ref{linked}), we give
the multidegree of these congruences in the Chow ring of the Grassmannian (Propositions \ref{prop:tria} and \ref{prop:tric}).

Both $X_\omega$ and $Y_{\omega,x\w y}$ are improper intersections of the Grassmannian with a linear space: 
$Y_{\omega,x\w y}$ is contained in a codimension $n$ linear space which we shall denote by $\Lambda_{\omega,x\w y}$ 
(i.e. $\Lambda_{\omega,x\w y}$ is a 
hyperplane in $\Lambda^{xy}_{\omega}$), see Definition \eqref{span4}, Remark \ref{rmk:l} and Proposition \ref{spanY}. 
 We give different characterisations 
of their linear spans.  In particular we identify the quadrics in the ideal of $\GR$---called also \emph{Pl\" ucker quadrics}---that contain 
the linear span of $X_\omega$ (Proposition \ref{Qn+1}).  
These quadrics necessarily have a large linear singular locus.  
We show (Theorem \ref{2n-quadric}) that this singular locus is the linear span of a congruence $X'_{\omega'}$ of dimension one less.

We study the fundamental loci of $X_\omega$ and of its residual congruences, giving equations and numerical invariants of the fundamental locus $F_{\omega}$ of $X_\omega$ (Proposition \ref{prop:fonda}) and numerical invariants for the fundamental locus $G_{\omega,x\w y}$ of $Y_{\omega,x\w y}$ (Theorem \ref{Gomega}).

Peskine showed recently that if an irreducible congruence of lines is Cohen-Macaulay and  has order one, then the lines of the congruence are the $k$-secant lines to its fundamental locus for some $k$ \cite[Theorem 3.2]{P2}.   The integer $k$ is called the {\em secant index} of the congruence.

The congruences $X_\omega$ and $Y_{\omega,x\w y}$ are both Cohen-Macaulay (Theorem \ref{thm:qq1} and Proposition \ref{Y is CM}, \eqref{y:3}).  
In fact $X_\omega$ is smooth, while  $Y_{\omega,x\w y}$ is singular and not even Gorenstein when $n>5$ (Proposition \ref{singular}). 
Moreover, we prove that both  $X_\omega$ and $Y_{\omega,x\w y}$ are 
arithmetically Cohen-Macaulay (Corollary \ref{cor:acm} and Proposition \ref{Yred}, \eqref{r:1}). Since $X_\omega$ is also subcanonical, it is 
arithmetically Gorenstein also (Corollary \ref{cor:acm}). 

When $n$ is odd, then $X_\omega$ has order one, the fundamental locus $F_{\omega}\subset\p(V)$ has codimension $3$ and is singular in codimension $10$.  The secant index of $X_\omega$ is $(n-1)/2$.  The congruence $Y_{\omega,x\w y}$ has order zero and the fundamental locus $G_{\omega,x\w y}$ is a hypersurface of degree $(n-1)/2$ that contains  $F_{\omega}$. 

When $n$ is even, then $X_\omega$ has order zero, and the fundamental locus $F_{\omega}$ is a hypersurface of degree $(n-2)/2$, while 
$Y_{\omega,x\w y}$ has order one and the fundamental locus $G_{\omega,x\w y}=\Pi\cup G_0$ is the union of the codimension $2$ linear space $\Pi=\{x=y=0\}\subset \p(V)$ and a codimension $3$ subvariety $G_0$ contained in $F_\omega$ (Theorem {\ref{Gomega}).   The secant index of $Y_{\omega,x\w y}$ is $n/2$, and each line in the congruence is $(n-2)/2$-secant to $G_0$ and intersects $\Pi$ (Theorem \ref{pippo2}).

The quadrics in the ideal of $\GG$ naturally correspond to elements of $\bw^4V^*$, and the quadrics in this ideal that contain the linear span $\Lambda_{\omega}$ are naturally characterised via this correspondence. Moreover, for these quadrics, their singular locus is the linear span of a 
congruence of the same type and dimension one less (Theorem \ref{2n-quadric}).

Therefore we introduce and present basic results on this correspondence in Section \ref{sect:1}.  The congruence $X_\omega$ is defined and introduced with basic properties in Section \ref{sect:2}.  Section \ref{sect:fund} is devoted to the fundamental locus of $F_\omega$ of $X_\omega$, while properties of the Hilbert scheme of $X_\omega$ are studied in Section \ref{hilbert}.
In Section \ref{sect:quadrics} we identify the quadrics in the ideal of  $\GG$ that contain the linear span of $X_\omega$.   A general linear space of maximal dimension in such a quadric contains a congruence $X_\omega'$ or a congruence $Y_{\omega',x\w y}$ for some  $3$-form $\omega'$ and some linear forms $x,y\in V^*$.  The final section \ref{residual} is devoted to the  congruence $Y_{\omega,x\w y}$ and its fundamental locus.  
Some of the main properties and invariants of the congruences $X_\omega$ 
and $Y$ and their fundamental loci are collected in Tables \ref{t:1} and \ref{t:2} at the end of the paper.
\bigskip

{\sc Acknowledgements.} We wish to thank Ada Boralevi, Fr\'ed\'eric Han and Christian Peskine for encouragements and many interesting discussions on this topic.
We would like to thank the referee for useful remarks.

\subsection{Notation}\label{notation}
We shall 
work over an algebraically closed field $\CC$ of characteristic zero. 

Throughout the paper, $V$ will be an $(n+1)$-dimensional $\CC$-vector space, and $\GR:=G(2,V)$ will 
be the Grassmannian of $2$-dimensional vector subspaces of $V$.

First, we fix a basis $(e_0,\dotsc,e_n)$ for 
$V$.  
Let $(x_0,\dotsc,x_n)$ be the dual basis in $V^*$.  
Then, the $n$-dimensional projective space defined by the lines of $V$ is 
$\p(V)=\Proj(\CC[x_0\dotsc,x_n])=\Proj(\Sym(V^*))$.
We shall denote---as it is the custom---by $\TTT_{\p(V)}$ and $\Omega_{\p(V)}^1$ the tangent and cotangent bundles of $\p(V)$, respectively. Moreover, 
as usual, $\Omega_{\p(V)}^k:=\bw^k \Omega_{\p(V)}^1$,
$\TTT_{\p(V)}(h):=\TTT_{\p(V)}\otimes \OO_{\p(V)}(h)$, etc. 

We will adopt the following convention on parenthesis:  $S^m\E(t)$
means the twist by $\OO(t)$ of the $m$-th symmetric product of $\E$,
and similarly for exterior powers and any other Schur functor of $\E$.

We consider the Grassmannian $\GR$ 
with its Pl\"ucker embedding: 
$\GR\subset\p(\bw^2V)$, and we fix Pl\"ucker coordinates $p_{i,j}$ on it: for 
example the point defined by 
$p_{i,j}=0$ if $(i,j)\neq(0,1)$, 
i.e. the point $[1,0,\dotsc,0]\in \GR$, corresponds to the 
subspace generated by $e_0$ and $e_1$.

Since there is no standard notation about 
universal and quotient bundles on Grassmannians, we fix it in the following.
On $\GR$ we denote the \emph{universal subbundle} of rank $2$ by $\UU$  and the \emph{quotient bundle} of rank $n-1$ by $\QQ$. 
They fit in the following exact sequence
\begin{equation}\label{seq:uq}
0\to \UU     \to     V\otimes \OO_{\GR}\to \QQ\to 0,
\end{equation}
where the universal subbundle $\UU$ has as its fibre over $\ell\in \GR$ the $2$-dimensional vector subspace $L$ of $V$ 
which corresponds to the point $\ell$, while the quotient bundle $\QQ$ has as its fibre over $\ell\in \GR$ the quotient space $V/L$.

A \emph{congruence of lines} in $\p(V)$ is a family of lines of dimension $n-1$. In other words, 
it is a closed subvariety---not necessarily irreducible---of dimension $n-1$ of $\GR$. 
A \emph{linear congruence} is a congruence corresponding to the proper intersection of $\GR$ with a linear subspace of $\p(\bw^2V)$ of codimension $n-1$. 
Note that this definition of linear congruence is more restrictive than the one given by Peskine  \cite{P2}, who considers linear but possibly improper sections of $\GR$ of pure dimension $n-1$.

The \emph{order} of a congruence $\Gamma$ is  the number of lines of $\Gamma$ passing through a general point $P\in \p(V)$. 
In other words it is the 
degree of the intersection of $\Gamma$ with the Schubert variety $\Sigma_P\subset \GR$  of the lines passing through $P$.

Finally, for $x\in V^*\setminus\{0\}$ we denote by $V_x:=\{x=0\}\subset V$, the hyperplane of equation $x=0$. 
Moreover, for two linearly independent forms  $x,y\in V^*$, we denote by $V_{x\w y}:=\{x=y=0\}\subset V$ the corresponding codimension $2$ subspace 
$V_{x\w y}=V_x\cap V_y$. As in the introduction, its projectivisation is $\Pi:=\p(V_{x\w y})\subset \p(V)$. 

Let $\omega_x\in\bw^3V_x^*$  be the natural
restriction of $\omega$ to $V_x$. If we choose a vector $e\in V$ such that $x(e)=1$, 
then we have an identification $V_x^*\simeq e^\perp \subset V^*$, which induces an inclusion $\bw^3V_x^*\subset \bw^3V^*$. 
Therefore we get 
\begin{equation}\label{decomp}
\bw^3 V^*=\bw^3V_x^*\oplus \bw^{2} V_x^*\w \gen{x}
\end{equation}
which induces a unique decomposition $\omega=\omega_x+\beta_x\w x$.

We shall often decompose $\omega$ in this way, without specifying the choice of the vector $e$.


\section{$3$-forms, $4$-forms and linear spaces in quadrics defining the Grassmannian of lines }\label{sect:1}

In this section we discuss quadrics in $\p(\bw^2V)$ defined by $4$-forms on $V$, and in particular by decomposable $4$-forms $\eta\in \bw^4V^*$ that have a linear factor,
i.e. $\eta=x\w \omega$ for some $x\in V^*$ and $\omega\in \bw^3 V^*$.  We shall study the rank of these quadrics.  In the next sections we shall consider the intersection of maximal linear subspaces of such quadrics with the  Grassmannian $\GR$ of lines in $\p(V)$.  In particular, interesting congruences of lines appear in such intersections.

We start by recalling the  well-known isomorphism between the space of $4$-forms on $V$ 
and the space of quadratic forms in the ideal of $\GR$.

Recall that the ideal of the Grassmannian $\GR\subset\p(\bw^2V)$ 
is generated by the quadrics of rank $6$ given by the Pl\"ucker relations. 
One way to obtain them is the following, see \cite{Mu}: let us take a general element $[L]\in\p(\bw^2V)$
\begin{equation*}
[L]=[\sum_{0\le i< j \le n}p_{i,j} e_i\w e_j];
\end{equation*}
then $[L]\in \GR$ if and only if 
\begin{equation*}
L\w L=0;
\end{equation*}  
indeed 
\begin{align*}
[L\w L]&=[(\sum_{0\le i< j \le n}p_{i,j} e_i\w e_j)\w (\sum_{0\le h< k \le n}p_{h,k} e_h\w e_k)]\\
&= [(\sum_{0\le i< j <h<k\le n}2(p_{i,j}p_{h,k}-p_{i,h}p_{j,k}+p_{i,k}p_{j,h})  e_i\w e_j\w e_h\w e_k)]
\end{align*}  
and the vanishing of the coefficients gives the Pl\"ucker relations. More intrinsically, define the \emph{reduced square}
$L^{[2]}\in \bw^4V$ of a bivector $L=\sum_{0\le i< j \le n}p_{i,j} e_i\w e_j\in\bw^2V $ as
\begin{equation*}
L^{[2]}:=\sum_{0\le i< j <h<k \le n}\pf\begin{pmatrix}
0&p_{i,j}&p_{i,h}&p_{i,k}\\
-p_{i,j}&0&p_{j,h}&p_{j,k}\\
-p_{i,h}&-p_{j,h}&0&p_{h,k}\\
-p_{i,k}&-p_{j,k}&-p_{h,k}&0
\end{pmatrix}
 e_i\w e_j\w e_h\w e_k.
\end{equation*}
Clearly we have $L\w L=2 L^{[2]}$ (i.e. this definition does not depend on the chosen basis of $V$), 
and $[L]\in \GR$ if and only if $L^{[2]}=0$. 

Recall that one can extend the definition of quadratic form to define 
a \emph{quadratic map} 
\begin{equation*}
q\colon V_1\to V_2
\end{equation*}
 as a map between $\CC$-vector spaces $V_1$ and $V_2$ such that 
\begin{align*}
 q(a L)&=a^2q(L),& \forall &a\in \CC, \forall L\in V_1
\end{align*}
 and  that 
 \begin{equation*}
 B_q(L,L'):=q(L+L')-q(L)-q(L')
\end{equation*}
  is a bilinear map 
 \begin{equation*}
  B_q\colon V_1\times V_1\to V_2.
\end{equation*}
We then have the following 
\begin{prop}
$\GR\subset\p(\bw^2V)$ is scheme-theoretically the zero locus of the quadratic map associated to the reduced square
\begin{align*}
q^{[2]}\colon  \bw^2V &\to \bw^4V\\
 L&\mapsto L^{[2]}.
\end{align*}
\end{prop} 
In other words, the affine cone of the Grassmannian, $C(\GR)$, is the ``kernel'' (i.e. the 
inverse image of  zero) of $q^{[2]}$:
\begin{equation}\label{eq:seq}
0\to C(\GR)\cong \left(q^{[2]}\right)^{-1}(0)\hookrightarrow \bw^2V \xrightarrow{q^{[2]}} \bw^4V.
\end{equation}
Therefore, giving an element of the vector space $I(\GR)_2$ of the quadratic forms in the ideal of Grassmannian  
$I(\GR)$, 
is equivalent to giving a $4$-form in $\bw^4V^*$:
\begin{coro}\label{cor:iso}
The $\SL(n+1)$-equivariant map 
\begin{align*}
\bw^4V^*&\to \Sym^2\left(\bw^2V^*\right)\\
\eta & \mapsto q_{\eta},
\end{align*}
where $q_{\eta}$ is the quadratic
form on $\bw^2V$ defined by
\begin{align*}
q_{\eta}\colon \bw^2V&\to \CC\\ 
L&\mapsto \eta(L^{[2]}),
\end{align*}
 is an isomorphism onto $I(\GR)_2\subset \Sym^2\left(\bw^2V^*\right)$.
\end{coro}
\begin{proof}
We observe that $q_\eta$ is obtained via composition from $q^{[2]}$: $\eta\mapsto q_\eta=\eta \circ q^{[2]}$; 
therefore the thesis follows from \eqref{eq:seq}:
\begin{equation*}
0\to C(\GR)\to \bw^2V \xrightarrow{q^{[2]}} \bw^4V\xrightarrow{\eta} \CC.
\end{equation*}
\end{proof}

\begin{nota}
  We denote by $Q_\eta$ the quadric defined by $q_\eta$ in $\p(\bw^2V)$.
\end{nota}

The singular locus of the quadric $Q_\eta$ is defined by the kernel of the corresponding bilinear form.  
In terms of the corresponding $4$-form $\eta$ we get 
\begin{align}\label{eq:sing}
\Sing(Q_\eta)&=\{[L] \in \p(\bw^2V)\mid \eta(L\w L')=0, \forall L'\in
\bw^2V\} \\
\nonumber &=\{[L] \in \p(\bw^2V)\mid \rho_\eta(L)=0\},
\end{align}
where 
\begin{align*}
\rho_\eta\colon\bw^2V&\to \bw^2V^*\\ 
L&\mapsto \eta(L\w -)
\end{align*} 
is defined by \emph{contraction}. 
Note that $\rho_\eta$ is the \emph{polarity} associated to the quadric $Q_\eta$.

$Q_\eta\subset \p(\bw^2V)$ is a cone with vertex $\Sing(Q_\eta)$.  
Since $\bw^2V$ has dimension $\binom{n+1}{2}$ the rank of $Q_\eta$ is equal to
\begin{equation}\label{eq:rankq}
\rk(q_\eta)=\binom{n+1}{2}-\dim \Sing(Q_\eta)-1= \binom{n+1}{2}-\dim{\ker}{\rho_\eta}=\rk{\rho_\eta}.
\end{equation}

Let $\omega\in \bw^i V^*$ be an $i$-form on $V$ ($i\le n+1$), 
then $\forall j\le i$ the contraction defines two linear maps. The first one is  
\begin{align*}
f_\omega\colon \bw^j V&\to \bw^{i-j}V^*\\
\alpha&\mapsto \omega(\alpha)
\end{align*}
and the other is its transpose (up to sign):
\begin{align*}
{}^tf_\omega\colon  
\bw^{i-j}V&\to \bw^j V^*\\
\beta&\mapsto \omega(\beta).
\end{align*}

\begin{defi}\label{def:rko}
For an $i$-form $\omega\in \bw^iV^*$ we define its $j$-\emph{rank} as the rank of $f_\omega$ 
(or, which is the same, as the rank of 
its transpose). 
If $j=i-1$ we simply call it \emph{rank}  of $\omega$.
\end{defi}

\begin{exem} We shall be interested mainly in the following cases:  
\begin{enumerate}
\item For a $3$-form $\omega\in \bw^3V^*$  its \emph{rank} is its $2$-rank. 
In other words 
\begin{equation*}
\rk\omega:=\rk(\bw^2 V\to V^*: L\mapsto \omega(L))
\end{equation*}
or, which is the same,  
\begin{equation*}
\rk\omega:=\rk( V\to \bw^2V^*: e\mapsto \omega(e)).
\end{equation*}

\item If $\beta\in \bw^2V^*$, then the  \emph{rank} of  $\beta$ is its $1$-\emph{rank}. 

\item  If $\eta$ is a $4$-form, by \eqref{eq:rankq} its $2$-rank coincides with the rank of the quadric $Q_\eta$.
\end{enumerate}
\end{exem}

\begin{rema}
We note that Definition \ref{def:rko} of rank  is different from the usual one  for tensors, which is the minimum number 
of summands in an expression as sum of totally decomposable tensors. 

In particular, in the case of a $2$-form $\beta\in \bw^2V^*$, 
since $f_\beta$ can be identified with the skew-symmetric matrix associated to $\beta$,  the rank defined in Definition \ref{def:rko}
is twice the usual rank of tensors, i.e. $\beta$ is the sum of $\frac{1}{2}\rk\beta$ totally decomposable tensors.  
\end{rema}

We shall study now the rank of the quadratic form $q_\eta$, i.e. the rank of the linear map
$\rho_{\eta}$,  when $\eta$ is a decomposable $4$-form.  
We need some preparation.

Recall from Section \ref{notation} that when 
$x\in V^*$ is nonzero, then $V_x:=\{x=0\}\subset V$, denotes the subspace of equation $x=0$.  A $3$-form $\omega\in \bw^3V^*$ 
has a decomposition $\omega=\omega_x+\beta_x\w x$,  with  
 $\beta_x\in \bw^2V_x^*$ (see \ref{notation}).

\begin{lemma}\label{eta} Let $\eta\in \bw^4V^*\setminus\{0\}$.
\begin{enumerate}
\item\label{xxx} If $\eta$ is totally decomposable, i.e. $\eta=x\w x'\w x''\w x'''$, where 
$x,x',x'',x'''\in V^*$ are linearly independent, then the $2$-rank of $\eta$ is $\rk{\rho_\eta}=6$.
\item\label{xx} If $\eta=\beta\w x\w x'\neq 0$, and $\beta_{x,x'}$ is the restriction of the $2$-form $\beta$
to $V_{x\w x'}=V_x\cap V_{x'}$, 
then $\rk{\rho_\eta}= 2\rk\beta_{x,x'}+2\leq 2n$.
\item\label{x} If $\eta=\omega\w x$, where $x\in V^*\setminus\{0\}$ and $\omega_x$ is the restriction of $\omega$ to $V_x$,  
then $\rk{\rho_\eta}=2\rk{\omega_x}\leq 2n$.

\end{enumerate}
\end{lemma}
\begin{proof}  Case \eqref{xxx} is immediate, since the image of $\rho_\eta$ is spanned by pairs of factors in $\eta$.  

In Case \eqref{xx}, we may assume $x=x_0, x'=x_1$ and $\beta\in \bw^2 \gen{x_2,\dotsc,x_n}$.  For $L\in \bw^2V$, we write $L=L_{01}+v_0\w e_0+v_1\w e_1+ce_0\w e_1$, 
where $L_{01}\in \bw^2\gen{e_2,\dotsc,e_n}$ and $v_0, v_1\in \gen{e_2,\dotsc,e_n}$. Then  
\begin{equation*}
(\beta\w x_0\w x_1)(L)=\beta(L_{01})x_0\w x_1-\beta(v_0)  x_1+\beta(v_1)x_0 +c\beta,
\end{equation*}
so the formula follows.

In Case \eqref{x} we may assume $x=x_0$ and write $\omega\w x_0=\omega_{x_0}\w x_0$.  For $L\in \bw^2V$, we write $L=L_0+v_0\w e_0$, 
where $L_0\in \bw^2\gen{e_1,\dotsc,e_n}$ and $v_0\in \gen{e_1,\dotsc,e_n}$. Then  
\begin{equation*}
(\omega_{x_0}\w x_0)(L)=\omega_{x_0}(L_0)\w x_0-\omega_{x_0}(v_0),
\end{equation*}
and the formula follows.
\end{proof}

The rank of the quadrics in the ideal of the Grassmannian $G(2,6)$ has been studied by Mukai in \cite{Mu}.  In this case $\bw^4 V^*\cong \bw ^2 V^*$, so the three possible ranks of a $2$-form  give three possible ranks for quadrics in the ideal of $G(2,6)$:
\begin{prop}[{\cite[Proposition 1.4]{Mu}}] 
Let $I_r\subset \p(I(G(2,6))_2)$ be the set of quadrics of rank $r$ in the ideal of $G(2,6)$.  Then 
$\p(I(G(2,6))_2)=I_6\cup I_{10}\cup I_{15}$, and $\dim I_6=8$, $\dim I_{10}=13$, $ \dim I_{15}=14$.  
Moreover, 
\begin{itemize}
\item If $[q]\in I_6$, then $q$ is a Pl\" ucker quadric.
\item If $[q]\in I_{10}$, then $q$ is a linear combination of two Pl\" ucker quadrics.
\item If $[q]\in I_{15}$, then $V(q)$ is smooth.
\end{itemize}
\end{prop}


\subsection{Genericity conditions on $3$-forms}\label{genericity}
In this paper our key objects are defined by ``general'' $3$-forms. In this section, we introduce and discuss 
some of the relevant genericity conditions on $\omega\in \bw^3V^*$.

\begin{defi}[Genericity conditions 1--3]
Let $\omega\in \bw^3V^*$.  We will  consider the following conditions: 
\begin{description}
\item[(GC1)\label{GC1}]  $\omega$ is indecomposable, i.e.  the 
multiplication map 
\begin{align*}
 V^*&\to \bw^4V^*\\ 
x &\mapsto \omega\w x
 \end{align*}
is injective
(the image of this map is the subspace $\omega\w V^*\subset \bw^4V^*$); 

\item[(GC2)\label{GC2}]  $\omega$ has rank $n+1$, i.e.  the linear map
\begin{align}\label{eq:fomegat}
{}^tf_\omega\colon \bw^2V&\to V^*\\ 
L&\mapsto \omega(L)  \nonumber
 \end{align}
is surjective, or, equivalently,   the linear map
 \begin{align}\label{eq:fomega}
f_\omega\colon  V&\to \bw^2V^*\\
v&\mapsto \omega(v) \nonumber
\end{align} 
is injective;

\item[(GC3)\label{GC3}]  for any $v\in V$, $v\neq 0$, $\rk  f_\omega(v)>2$. 
\end{description}
\end{defi}

\begin{rema}  There are dependencies among these conditions:
\begin{enumerate}
\item Condition \ref{GC2} is satisfied, i.e. the rank of $\omega$ is $n+1$, only if $n\ge 4$.  Indeed, when $n=3$, any nonzero $3$-form $\omega$ is totally decomposable
 so it has rank $3$.
 \item Condition \ref{GC1} is satisfied, i.e. $x \mapsto \omega\w x$ is injective, only if $n\ge 5$. Indeed, when $n\leq 4$, any $3$-form is decomposable 
 so the map $x \mapsto \omega\w x$ is not injective, see Example~\ref{ex2}:  with coordinates as in the example, $x_0$ is in the kernel. 
\item 
It is immediate to observe that  condition 
\ref{GC3} is more restrictive than both \ref{GC1} and \ref{GC2}:
\begin{equation*}
\textup{\ref{GC3}}\Longrightarrow\textup{\ref{GC1}}, \  \textup{\ref{GC3}}\Longrightarrow\textup{\ref{GC2}}.
\end{equation*}
\item There are forms that satisfy \ref{GC1} and not \ref{GC2}:
Let $v_0\in V$ be a nonzero vector and assume $\omega(v_0)=0$, i.e. \ref{GC2} is not satisfied.  Then $\omega\in \bw^3 V_0^*$, where $V_0^*=\{v^*\in V^*|v^*(v_0)=0\}$.
 If the multiplication map by $\omega$: $V_0^*\to \bw^4 V_0^*$ is injective, then so is $V^*\to \bw^4 V^*$ and  \ref{GC1} is satisfied.
\item If $n$ is even, there are forms that satisfy \ref{GC2} and not \ref{GC1}:
If  $\omega$ does not satisfy \ref{GC1}, for a suitable choice of coordinates $\omega=x_0\w \beta_0$. When $n$ is even and $\beta_0$ is general, then $\rk \beta_0=n$, hence $\rk \omega=n$ and condition \ref{GC2} holds.
\item If $n$ is odd, \ref{GC2} implies \ref{GC1}:
If  $\omega$ does not satisfy \ref{GC1}, for a suitable choice of coordinates $\omega=x_0\w \beta_0$. When $n$ is odd, $\rk \beta_0\leq n-1$ so  $\rk \omega\leq n-1$ and condition \ref{GC2} does not hold.
\end{enumerate}
\end{rema}

Next proposition ensures that the conditions \ref{GC1}, \ref{GC2}, \ref{GC3} are all satisfied by general $3$-forms for $n$ sufficiently large.
\begin{prop}\label{dimPn}
Let $\omega\in \bw^3V^*$ be general. If $n>3$, then 
conditions \ref{GC1} and  \ref{GC2} hold, and if $n>5$, then $\omega(v)$ has rank at least $4$ for every $v\in V$, i.e. condition \ref{GC3} holds. 
\end{prop}

\begin{proof}
To prove the first claim, it is enough to find an example for any $n$. For instance  the following ones, depending on the congruence class of $n$ modulo $3$, will do the job:
if $n+1\equiv 0\mod 3$ take $\omega= x_0\w x_1\w x_2+ \dotsm +x_{n-2}\w x_{n-1}\w x_{n}$,  if $n+1\equiv 1\mod 3$, $n>3$, take $\omega= x_0\w x_1\w x_2+ \dotsm +x_{n-3}\w x_{n-2}\w x_{n-1} + x_n\w x_0\w x_3$, if $n+1\equiv 2\mod 3$, $n>4$, take $\omega= x_0\w x_1\w x_2+ \dotsm +x_{n-1}\w x_{0}\w x_{3} + x_n\w x_1 \w x_4$.  To prove the second claim, we recall that the $2$-forms of rank $\leq 2$ describe a Grassmannian of dimension $2(n-1)$ in $\p(\bw^2 V^*)$, hence for $n>5$ the codimension is $>n+1$.
\end{proof}

\subsection{Quadrics $Q_{\omega\w x}$ and their linear subspaces.}\label{the quadric}
Given a $3$-form $\omega\in \bw^3V^*$, the $4$-forms $\omega\w x$ for $x\in V^*$ define a linear space of quadrics 
\begin{align*}
Q_{\omega\w x}&\in I(\GG)_2,& x&\in V^*.
\end{align*}
As we shall see, they are intimately related to the variety $X_\omega$ which is our main object of study.

In this section we identify linear subspaces of the quadrics $Q_{\omega\w x}$.
If $\omega\in\bw^3V^*$ and $x, y\in V^*$, we define the following linear subspaces of $\p(\bw^2V)$:
\begin{align}\label{span} 
\Lambda_{\omega}&:=\{[L]\in\p(\bw^2V)\mid\omega(L)=0\},\\
\label{span1} \Lambda_{\omega}^x&:=\{[L]\in\p(\bw^2V)\mid\omega(L)\w x=0\},\\
\label{span2} \Lambda_{\omega}^{xy}&:=\{[L]\in\p(\bw^2V)\mid\omega(L)\w x\w y=0\},
\end{align}
and moreover
\begin{align}
\label{span3}  \Lambda_{\omega_x}&:=\{[L]\in \p(\bw^2 V)\mid x(L)=\omega(L)\w x=0\},\\
\label{span4}  \Lambda_{\omega,x\w y}&:=\{[L]\in \p(\bw^2 V)\mid (x\w y)(L)=\omega(L)\w x\w y=0\}.
\end{align}

Furthermore we denote by  
$P_\omega:=\Lambda_\omega^\perp\subset\p(\bw^2 V)^*$ the subspace  orthogonal to $\Lambda_\omega$. 
Note the inclusions 
\begin{equation*}
\begin{matrix}
\Lambda_{\omega}&\subset&\Lambda_{\omega}^x&\subset&\Lambda_{\omega}^{xy}\\
&&\cup &&\cup\\
&&\Lambda_{\omega_x} &&\Lambda_{\omega,x\w y}
\end{matrix}
\end{equation*}
and the relation 
\begin{equation*}
\Lambda_{\omega_x}=\Lambda_{\omega}^x\cap\gen{ G(2,V_x)},
\end{equation*}
 where $V_x$ is the hyperplane $\{ x=0\}$ in $V$. Let us remark that the notation for $\Lambda_{\omega_x}$ is coherent with the one for $\Lambda_\omega$, because $\omega_x$ denotes the restriction of $\omega$ to the hyperplane $V_x$. Moreover $ \Lambda_{\omega,x\w y}$ is the intersection of $\Lambda_{\omega}^{xy}$ with the Schubert hyperplane generated by the planes in $\GR$ meeting the codimension $2$ subspace $V_{x\w y}:=\{x=y=0\}$.

\begin{rema}\label{rmk:l} The codimension of the spaces \eqref{span}-\eqref{span4} depends on the rank of $\omega$, and on the rank of $\omega_x$ as a form on $V_x$.

\begin{enumerate}
\item
\begin{equation*} 
\Lambda_{\omega}=\p(\ker({}^tf_\omega))\subset \p(\bw^2V),
\end{equation*}
where $\ker({}^tf_\omega)$ is the map defined in \eqref{eq:fomegat}; 
therefore 
${\codim}\Lambda_{\omega}
=\rk \omega$.

\item
If $x\in \im{}^tf_\omega\subset V^*$, 
then ${\codim}\Lambda_{\omega}^x=\rk \omega -1$, i.e. $\Lambda_{\omega}$ is a hyperplane in $\Lambda_{\omega}^x$. 
\item 
If $x\notin \im{}^tf_\omega\subset V^*$, then  ${\codim}\Lambda_{\omega}^x=\rk \omega$, i.e. $\Lambda_{\omega}=\Lambda_{\omega}^x$.  
\item
If $\gen{x,y}\subset \im{}^tf_\omega $, then ${\codim}\Lambda_\omega^{xy}=\rk \omega-2$,
 i.e. $\Lambda_{\omega}^x$ is a hyperplane in $\Lambda_{\omega}^{xy}$. 
\item ${\codim}\Lambda_{\omega_x}=n+\rk \omega_x$.

\item 
If $\gen{x,y}\subset \im{}^tf_\omega $, then ${\codim}\Lambda_{\omega,x\w y}=\rk \omega-1$,
 i.e. $\Lambda_{\omega,x\w y}$ is a hyperplane in $\Lambda_{\omega}^{xy}$. 
\end{enumerate}
\end{rema}

For a general $3$-form $\omega$ we immediately get:

\begin{lemma}\label{supsets} If $\omega$ has rank $n+1$, i.e. it satisfies \ref{GC2}, 
then  $\{\Lambda_{\omega}^x\mid x\in V^*\}$ and $\{\Lambda_{\omega}^{xy}\mid x\w y\in \bw^2V^*\}$ are the sets of codimension $n$ subspaces 
and codimension $n-1$ subspaces of $\p(\bw^2V)$  that contain $\Lambda_\omega$, respectively.
\end{lemma}
Consider now the quadric $Q_{\omega\w x}\subset \p(\bw^2 V)$, of equation $(\omega\w x)(L\w L)=0$. 
Notice, first, that it depends only on the restriction $\omega_x$ of $\omega$ to $V_x$.  
In fact if $\omega=\omega_x+\beta_x\w x$, then $\omega\w x=\omega_x\w x$. 

We first find the singular locus of $Q_{\omega\w x}$.  
 Let $K_{\omega_x}=\ker f_{\omega_x}
\subset V_x$, and set $\Lambda^K_{\omega_x}  =\p( K_{\omega_x}\bw V)\subset \p(\bw^2V)$.

\begin{lemma}\label{singularlocus_quadrics}
The singular locus of $Q_{\omega\w x}$ is the subspace
$\gen{\Lambda_{\omega_x}\cup \Lambda^K_{\omega_x}}\subset \p(\bw^2V)$.
In particular, if $\omega_x$ has rank $m$, then $Q_{\omega\w x}$ has rank $2m$.
\end{lemma}
\begin{proof}
The singular locus  of $Q_{\omega\w x}$ is $\p(\ker\rho_{\omega\w x})$, where 
\begin{align*}
\rho_{\omega\w x}\colon \bw^2 V&\to \bw^2 V^*\\ 
L&\mapsto (\omega\w x)(L),
\end{align*}
so as in the proof of Lemma \ref{eta}, \eqref{x}, we may assume  $\omega=\omega_x\in \bw^3V_x^*$ and let $L=L_x+v_x\w e$, where $L_x\in \bw^2 V_x$,  $v_x\in V_x$ and $x(e)=1$.  Then the singular locus of $Q_{\omega\w x}$ is spanned by classes of $2$-vectors $L$ such that
\begin{equation*}
(\omega\w x)(L)=\omega_x(L_x)\w x -\omega_x(v_x)=0,
\end{equation*}
which implies 
\begin{equation*}
\omega_x(L_x)=\omega_x(v_x)=0,
\end{equation*} 
or equivalently 
\begin{align*}
[L_x]&\in \Lambda_{\omega_x} &\textup{and} && [v_x\w e]&\in \Lambda^K_{\omega_x}.
\end{align*}
The rank of $Q_{\omega\w x}$ equals the rank of  $\rho_{\omega\w x}$, which is $2\rk \omega_x$, by Lemma \ref{eta}\ref{x}.
\end{proof}

\begin{rema} When $f_{\omega_x}:V_x\to \bw^2V_x^*$ is injective, i.e. $\omega_x$ has rank $n$, then 
the singular locus of $Q_{\omega\w x}$ is $ \Lambda_{\omega_x}$ and the rank of $Q_{\omega\w x}$ is $2n$.
\end{rema}

We shall study now the linear spaces in the quadrics $Q_{\omega\w x}$, with special attention to those containing also $\Lambda_\omega$. 

\begin{lemma}\label{span_quadrics}
The quadric $Q_{\omega\w x}$, for $x\in V^*$,   contains the codimension $n$ linear subspaces $\Lambda^x_{\omega}$, 
$\Lambda_{\omega,x\w y}$ for $x\w y\in \bw^2V^*$, and $\p(\bw^2V_x)$. 

In particular, each quadric $Q_{\omega\w x}$, $x\in V^*$, contains $\Lambda_\omega$, and each 
quadric of the pencil generated by $Q_{\omega\w x}$ and $Q_{\omega\w y}$ contains 
the linear subspace
$\Lambda_{\omega,x\w y}$. 
\end{lemma}
\begin{proof}
$[L]\in Q_{\omega\w x}$ if and only if
\begin{equation*}
(\omega\w x)(L\w L)=(\omega (L)\w x-\omega (x(L))(L)=-2\omega(L)(x(L))=0,
\end{equation*}
which is equivalent to 
\begin{equation*}
(\omega(L)\w x)(L)=0.
\end{equation*}
So if $\omega(L)\w x=0$ or $x(L)=0$, then $[L]\in Q_{\omega\w x}$.
Therefore $\Lambda_{\omega\w x}\subset Q_{\omega\w x}$ and $\p(\bw^2V_x)\subset Q_{\omega\w x}$.

Similarly, if $x,y\in V^*$ and $\omega(L)\w x\w y=0$, then $\omega(L)=ax+by$ for some $ax+by\in \gen{ x,y}$.  If furthermore $(x\w y)(L)=0$, then 
\begin{equation*}
(\omega(L)\w (cx+dy))(L)=((ax+by)\w (cx+dy))(L)=0
\end{equation*}
so $[L]\in Q_{\omega\w (cx+dy)}$ for any $cx+dy\in \gen{ x,y}$.
Therefore the linear space $\Lambda_{\omega,x\w y}$ is contained in the pencil of quadrics generated by $Q_{\omega\w x}$ and $Q_{\omega\w y}$.
\end{proof}

Since the quadric $Q_{\omega\w x}=Q_{\omega'\w x}$, whenever $(\omega'-\omega)\w x=0$,  the Lemma \ref{span_quadrics} applies to show that 
\begin{equation*}
\Lambda^x_{\omega'}\subset Q_{\omega\w x}
\end{equation*}
for every $[\omega']$ in the set

 \begin{equation}\label{firstfamily}
 \{[\omega']\in \p(\bw^3V^*)\mid 
 \omega(L)\w \omega'(L)=0\quad \textup{when} \quad x(L)=0\}.
 \end{equation}

Likewise
\begin{equation*}
\Lambda_{\omega',x\w y}\subset Q_{\omega\w x}
\end{equation*}
for every $[\omega']$ and $[y]\neq [x]$ such that 
 \begin{equation}\label{secondfamily}
\{([\omega'], [y])\in \p(\bw^3V^*)\times \p(V_x^*)\mid 
\omega(L)\w \omega'(L)=0\quad \textup{when} \quad (x\w y)(L)=0\}.
\end{equation}

When the restriction $\omega_x\in \bw^3V_x^*$ has rank $n$, then 
the quadric $Q_{\omega\w x}$ has rank $2n$, by Lemma \ref{eta} \eqref{x}.  Hence 
$Q_{\omega\w x}$ contains two  $\binom{n}{2}$-dimensional families of linear spaces of minimal codimension $n$.  

\begin{theorem}\label{rk2n-quadrics} 
Let $\omega\in \bw^3V^*$ and $x\in V^*$ and assume that the restriction $\omega_x$ to $V_x$ has rank $n$. 
Then:
\begin{enumerate}
\item\label{rk2n:1}
The singular locus of the quadric $Q_{\omega\w x}$ is the linear space 
 \begin{equation*}
\Lambda_{\omega_x}=\{[L]\in \p(\bw^2 V)\mid x(L)=\omega_{x}(L)=0\}.
\end{equation*}
\item The two $\binom{n}{2}$-dimensional spinor varieties of $n$-codimensional linear subspaces in the quadric 
$Q_{\omega\w x}\subset \p(\bw^2 V)$ are birationally parametrised by the sets \eqref{firstfamily} and \eqref{secondfamily} respectively.  
\item The linear space 
 \begin{equation*}
\{[L]\in \p(\bw^2 V)\mid x(L)=0\}=\gen{G(2,V_x)},
\end{equation*}
spanned by the subgrassmannian $G(2,V_x)$, is contained in $Q_{\omega\w x}$ and belongs to the spinor variety parametrised by \eqref{firstfamily} 
if $n$ is odd, and to the spinor variety parametrised by \eqref{secondfamily} if $n$ is even.
\end{enumerate}
\end{theorem}

\begin{proof}  It remains, first, to show that the two families of linear subspaces \eqref{firstfamily} and \eqref{secondfamily}  are
$\binom{n}{2}$-dimensional.  The linear spaces in $\Lambda^x_{\omega'}$ in \eqref{firstfamily} depend on a $2$-form $\beta\in \bw^2V^*$ such that $\omega'=\omega_x+\beta\w e$, where $x(e)=1$, so this family is $\binom{n}{2}$-dimensional.

The linear spaces $\Lambda_{\omega',x\w y}$ in \eqref{secondfamily} depend on 
the choice of $y$ and the choice of $\beta'\w y$, such that $\omega'\in \gen{\omega, \beta'\w y}$.  The first choice is $(n-1)$-dimensional, the second is $\binom{n-1}{2}$-dimensional, 
so they sum to  $\binom{n}{2}$.
The second statement of the theorem follows.

For the third statement,  assume $x(L)=0$.  Then clearly
\begin{equation*}
\omega(L)(x(L))=0
\end{equation*}
and so the linear space
\begin{equation*}
\{[L]\mid x(L)=0\}
\end{equation*}
of codimension $n$ is also contained in $Q_{\omega\w x}$.  
The intersection 
\begin{equation*}
\{[L]\mid x(L)=0\}\cap \{[L]\mid \omega(L)\w x=0\}= \{[L]\in \p(\bw^2 V_{x})\mid \omega_{x}(L)=0\}
\end{equation*}
has codimension $n$ in $\{[L]\mid x(L)=0\}$, so this last linear subspace
belongs to the spinor variety parametrised by \eqref{firstfamily}  if $n$ is odd, 
and to the spinor variety parametrised by \eqref{secondfamily}  if $n$ is even.
\end{proof}

\begin{rema}\label{intersection}  We conclude this section observing that two  maximal isotropic spaces $\Lambda_{\omega'}^x$ and $\Lambda_{\omega',x\w y}$ 
of opposite  families on $Q_{\omega\w x}$, obtained from the \emph{same} $3$-form $\omega'$, 
are both contained in the subspace  
\begin{equation*}
\Lambda_{\omega'}^{xy}=\{[L]\in \p(\bw^2V) \mid \omega'(L)\w x\w y=0\} 
\end{equation*}
of codimension $n-1$.   
So they intersect along a subspace having codimension $1$ in both of them. This will come in hand in Section \ref{residual}.
\end{rema}

\section{The congruence}\label{sect:2}

In this section we introduce the congruence $X_\omega$ defined by a $3$-form $\omega$.  
We keep our notation $V$ for a fixed vector space of dimension $n+1$
and $\GG=G(2,V)$.

\subsection{The congruence as a linear section of the Grassmannian}\label{xomega}

Now, our main object of study is the following subset of the Grassmannian $\GR$:
\begin{equation*}
X_\omega:=\{[L]\in \GR\mid \omega(L)=0\}\subset \p(\bw^2V).
\end{equation*}

\subsubsection{The equations of the congruence}
We start by studying $X_\omega$ in coordinates.
Our $3$-form $\omega$ reads:
 \begin{equation}\label{eq:3form}
\omega=\sum_{0\le i< j< k\le n}a_{i,j,k} x_i\w x_j \w x_k\in \bw^3V^*.
\end{equation}
If we write $L$ as
\begin{equation}\label{eq:1}
L=\sum_{0\le a< b \le n}p_{a,b} e_a\w e_b
\end{equation}
where the $p_{a,b}$'s in \eqref{eq:1} satisfy the Pl\"ucker relations, then we require that 
\begin{equation}\label{eq:2}
\omega (\sum_{0\le a<b \le n}p_{a,b}e_a\w e_b)=0.
\end{equation}
More explicitly, we get 
\begin{equation}\label{eq:px}
\omega(L)=\sum_{i,j,k}( (-1)^{i+j-1}(a_{i,j,k}-a_{i,k,j}+a_{k,i,j}))x_kp_{i,j}=0. 
\end{equation} 
Therefore, we deduce that the equations of $X_\omega$ are
\begin{align}\label{eq:important}
\sum_{0\le i<j\le n}( (-1)^{i+j-1}(a_{i,j,k}-a_{i,k,j}+a_{k,i,j}))p_{i,j}&=0, & k&=0,\dotsc,n,
\end{align}
i.e. we have $n+1$ linear equations, together with the Pl\"ucker relations. 
The equations are sometimes more convenient in the form
 \begin{align}\label{eq:importantbis}
\sum_{0\le i<j\le n}( (-1)^{i+j+k-1}(a_{i,j,k}-a_{i,k,j}+a_{k,i,j}))p_{i,j}&=0, & k&=0,\dotsc,n,
\end{align}
\begin{rema}
Note that the equations 
\eqref{eq:important} define the linear span of $X_\omega$, $\gen{X_\omega}$, which was denoted by $\Lambda_\omega$ in \eqref{span}. 
\end{rema}

The linear space generated by  
equations  \eqref{eq:important} has a natural embedding in  $\p(\bw^2 V)^*$ as
the linear subspace 
$P_\omega$,  orthogonal to the linear span of $X_\omega$ in the Pl\" ucker embedding,   
i.e. 
\begin{equation}\label{not:1}
P_\omega:=\Lambda_\omega^\perp
\end{equation}
(see section \ref{the quadric}). 

By Proposition \ref{dimPn}, if $n\geq 4$ and  
the genericity condition \ref{GC2} holds, these equations are linearly independent, therefore $\dim P_\omega=n$, 
while if $n=3$, $\dim P_\omega=2$. 
Let us see the embedding of $P_\omega$ in coordinates: to obtain the parametric equations of $P_\omega$ 
we simply consider the coefficients of the $p_{i,j}$'s 
in $\omega(L)=0$, i.e. in equation \eqref{eq:px}: 
\begin{equation*}
\sum_k( (-1)^{i+j-1}(a_{i,j,k}-a_{i,k,j}+a_{k,i,j}))x_k;
\end{equation*}
therefore, if $q_{i,j}$ are the dual coordinates on $\p(\bw^2 V)^*$, the parametric equations of $P_\omega$
are given by  
\begin{equation}\label{eq:qij}
q_{i,j}=\sum_k( (-1)^{i+j-1}(a_{i,j,k}-a_{i,k,j}+a_{k,i,j}))x_k.
\end{equation}
In other words, if  $\omega$ has rank $n+1$, i.e. 
the linear map $f_\omega: V\mapsto \bw^2 V^*;  v\mapsto \omega(v)$ is injective \ref{GC2},
we have a linear embedding 
\begin{equation}\label{eq:emb}
\p(f_\omega) \colon \p(V)\to \p(\bw^2 V)^*
\end{equation}
defined by \eqref{eq:qij}. 
This map is represented by the following $(n+1)\times (n+1)$ skew-symmetric matrix of linear forms on $\p(V)$: 
\begin{equation}\label{eq:momega}
M_\omega:=\left(\sum_k( (-1)^{i+j-1}(a_{i,j,k}-a_{i,k,j}+a_{k,i,j}))x_k \right)_{\substack{i=0,\dotsc,n \\ j=0,\dotsc,n}}
\end{equation}
using the usual convention, in \eqref{eq:qij}, that $q_{i,j}=-q_{j,i}$. 

Summarising, if $\omega$ has rank $n+1$, then  
\begin{equation}\label{eq:pomega}
\Lambda_\omega^\perp=P_\omega=\im(\p(f_\omega))\cong \p(V).
\end{equation}
Clearly, $P_\omega=\overline{\im(\p(f_\omega))}$ always holds true
(i.e. also if $f_\omega$ is not injective).

\subsubsection{The tangent space of the congruence}

To interpret $X_\omega$ geometrically,  we consider certain linear subspaces of  
$\p(\bw^{3}V^*)$.

Fix a decomposable tensor $L\in \bw^{2}V$ and let $L^*\in \bw^{n-1}V^*$ be its dual.
We shall denote by $\TT_L\subset \p(\bw^3V^*) $ 
the linear span of the union of all the embedded tangent spaces to $G(3,V^*)$
at the points corresponding to the $3$-spaces $\pi^*$ such that   $L^*$ contains $\pi^*$, 
i.e. $\TT_L=\gen{\TT_{\pi^*} G(3,V^*)\mid \pi^*\subset L^*}$.

\begin{lemma} \label{lem} Assume $n\geq 4$, and let $\omega\in \bw^3V^*$. 
Let $L= e\w f\in \bw^2 V$.  We fix a basis of $V$, $e_0,\dotsc,e_n$, with $e_0=e$, $e_1=f$, and the dual basis $x_0, \dotsc, x_n$. Then the following are equivalent:
\begin{enumerate}
\item\label{lem1:1} $[L]\in X_\omega$; 
\item\label{lem1:2} $\omega$ can be uniquely 
written as $\omega=\omega_{01}+\beta_0\w x_0+ \beta_1\w x_1$,
with $\beta_0$, $\beta_1\in \bw ^2\gen{x_2, \dotsc, x_n}$,  $\omega_{01}\in \bw ^3\gen{x_2, \dotsc, x_n}$; 
\item\label{lem1:3} $[\omega]\in \TT_L$.
\end{enumerate}
\end{lemma}

\begin{proof}
 $[L]\in \GR$ has coordinates $p_{0,1}\neq 0$ and $p_{i,j}=0$ if $(i,j)\neq(0,1)$.
Therefore, from \eqref{eq:important}, we deduce that 
$a_{0,1,h}=0$, $\forall h=0,\dotsc, n$, and \eqref{eq:3form} becomes:
\begin{align*}
\omega=&\sum_{\substack{0\le i< j< k\le n\\ (i,j)\neq (0,1)}}a_{i,j,k} x_i\w x_j \w x_k
=  \beta_0\w x_0+ \beta_1\w x_1+ \omega_{01}, 
\end{align*}
where we have set
\begin{align*}
\beta_0 &:= -\sum_{2\le j< k\le n}a_{0,j,k}x_j\w x_k,
&&\beta_1 := -\sum_{2\le j< k\le n}a_{1,j,k}x_j\w x_k,\\
\omega_{01}  &:=\sum_{2\le i< j< k\le n}a_{i,j,k}x_i\w x_j\w x_k.
\end{align*}
Therefore the equivalence of \eqref{lem1:1} and \eqref{lem1:2} is proved. 
For the last equivalence, observe that
$L^*=x_2\w \dotsm \w x_n$, so, up to a coordinate change, a $3-$space $\pi^*\subset L^*$ can be expressed as
\begin{equation*}
[\pi^*]=[x_{n-2}\w x_{n-1}\w x_n]\in G(3,V^*)\subset\p(\bw^3V^*),
\end{equation*}
and the embedded tangent space to  $G(3,V^*)$ 
at $\pi^*$ can be expressed as
\begin{multline*}
\TT_{\pi^*} G(3,V^*) =\gen{ [(\sum_{i=0}^{n-2}a_ix_i)\w x_{n-1}\w x_n], \\
[(\sum_{i=0}^{n-3}b_ix_i-b_{n-1}x_{n-1})\w x_{n-2}\w x_n],[(\sum_{i=0}^{n-3}c_ix_i+c_n x_n)\w x_{n-2}\w x_{n-1}]}
\end{multline*}
from which the last equivalence follows.
\end{proof}

\begin{coro}
Let $L\subset V$ be a $2$-vector subspace, with $n\ge 4$; then 
\begin{equation}\label{c:1}
\dim(\TT_L)=\frac{n+3}{3}\binom{n-1}{2}-1.
\end{equation} 
\end{coro}

\begin{proof}
It follows from the equivalence of  \eqref{lem1:3} and \eqref{lem1:2} of Lemma \ref{lem}: 
\begin{align*}
\dim (\TT_L) &= 2\binom{n-1}{2}+\binom{n-1}{3}-1
=\frac{n+3}{3}\binom{n-1}{2}-1.
\end{align*}
\end{proof}

Recall from \S \ref{notation}, that 
a \emph{linear congruence} is a congruence obtained by proper intersection of $\GR$ with a linear subspace of $\p(\bw^2V)$ of codimension $n-1$, and that 
the \emph{order} of a congruence $\Gamma$ is  the number of lines of
$\Gamma$ passing through a general point of $\p(V)$.

\begin{theorem}\label{thm:1}
Let $\omega\in \bw ^3V^*$, 
with $n=\dim \p(V)\ge 3$, be a general $3$-form. 
Then $X_\omega\subset \GR$ has dimension $n-1$, i.e. it is a congruence. Moreover, $X_\omega$ is contained in a reducible linear 
congruence. 
\end{theorem}

\begin{proof}
We give here an elementary proof of the Theorem which works for $n\geq 6$. For $n\leq 5$, see Examples \ref{ex1}, \ref{ex2}, \ref{ex3} in Section \ref{examples}. A different proof of the Theorem will be given in Section \ref{congr_deg}. 

By case \eqref{lem1:3} of Lemma \ref{lem} we need to find how many spaces of the form $\TT_L$ pass through $\omega$. First of all, 
we observe that two (general) 
spaces $\TT_L$ and $\TT_{L'}$ are in general position; in fact, if 
$L=\gen{e_0, e_1}$ and $L'=\gen{e_2, e_3}$, 
then, by the equivalence of cases \eqref{lem1:2} and \eqref{lem1:3} of Lemma \ref{lem}, 
a $3$-form $[\omega]\in \p(\bw^3V^*)$  
belongs to $\TT_L \cap \TT_{L'}$ if and only if 
we can write 
\begin{multline*}
 \omega=x_0\w x_2\w\phi_{02}+x_0\w x_3\w\phi_{03}+x_1\w x_2\w\phi_{12}+x_1\w x_3 \w\phi_{13}+\\
 +x_0 \w \beta_0+x_1 \w \beta_1+x_2 \w \beta_2 + x_3\w\beta_3+ \omega', 
 \end{multline*}
where, as in Lemma \ref{lem}, 
 $\beta_0,\dotsc, \beta_3 \in \bw ^2{\gen{x_4, \dotsc, x_n}}$ and $\omega'\in \bw ^3{\gen{x_4, \dotsc, x_n}}$.

From this, we infer that 
\begin{align*}
\dim(\TT_L \cap \TT_{L'})+1&=4(n-3)+4\binom{n-3}{2}+\binom{n-3}{3}=\\
&=\frac{(n-3)(n-1)(n+4)}{6},
\end{align*}
from which we deduce 
\begin{align*}
\dim(\gen{\TT_L, \TT_{L'}})+1&=\dim (\TT_L)+ \dim (\TT_{L'})-\dim(\TT_L \cap \TT_{L'})+1\\
&=2\frac{n+3}{3}\binom{n-1}{2}-\frac{(n-3)(n-1)(n+4)}{6}\\
&=\binom{n+1}{3},
\end{align*}
i.e. $\gen{\TT_L, \TT_{L'}}=\p(\bw^3V^*)$,
and $\TT_L$ and $\TT_{L'}$ are in general position.

Then, we observe that there is a family of dimension $\dim \GR=2(n-1)$ of $\TT_L$'s, therefore, recalling \eqref{c:1}, if $\omega$ is general, 
the dimension of the  $\TT_L$'s passing through $\omega$ is equal to  
\begin{align*}
\dim \GR+\dim(\TT_L) -\dim(\p(\bw^3V^*)
&=2(n-1)+\frac{n+3}{3}\binom{n-1}{2}-\binom{n+1}{3}\\
&=n-1.
\end{align*}

By \eqref{eq:important}, it follows that $X_\omega$ is contained in a linear congruence.
\end{proof}

\subsubsection{More genericity conditions}

Theorem \ref{thm:1} naturally gives rise to the following genericity
conditions, which we call 4 and 5, after \S \ref{genericity}.

\begin{defi}\label{GC45}
We shall consider the following conditions on a $3$-form $\omega\in\bw^3V^*$:
\begin{description}
\item[(GC4)\label{GC4}]  $X_\omega$ has the expected dimension $n-1$;
\item[(GC5)\label{GC5}] $X_\omega$ satisfies \ref{GC4} and it is smooth. 
\end{description}
\end{defi} 
\begin{rema} \label{implicationsGC}
The condition \ref{GC4} implies the condition  \ref{GC2} in \ref{genericity}.
\begin{enumerate}
\item When $\omega$ has rank $m$, then the linear span of $X_\omega$ has codimension $m$ in $\p(\bw^2V)$ and the codimension of $X_\omega$ in $\GR$ is at most $m-2$.
\item  For $n\geq 8$, the $3$-form $\omega=x_0\w x_1\w x_2+x_3\w x_4\w x_5+x_6\w x_7\w x_8$, has rank $9$, while $X_\omega$ has codimension $6$, so \ref{GC2} holds, while \ref{GC4} fails for $\omega.$
\end{enumerate}
\end{rema}

\subsubsection{The order of the congruence}

We compute now the order of the congruence $X_\omega$.

\begin{prop}\label{prop:ordine}
If 
$\omega$ is general, $X_\omega\subset \GR$ is a congruence of order zero if $n$ is even and of order one if $n$ is odd.
\end{prop}

\begin{proof}
Let $n\geq 6$. We look at the intersection $\Sigma_P\cap X_\omega$, where, without loss of generality, 
we can suppose that $P=[1,0,\dotsb,0]=[e_0]$. $\Sigma_P$ is given by the lines of type $[L]=[e_0\w f]$, 
with $e_0$ and $f$ linearly independent; so
$\Sigma_P$  is defined, in the Pl\"ucker coordinates, by $p_{i,j}=0$ if $i>0$. 

Therefore, by \eqref{eq:importantbis}, the intersection $\Sigma_P\cap X_\omega$ is defined by the $n$ equations:
\begin{equation}\label{eq:sis}
\sum_{0< j\le n}( (-1)^{j+k-1}(a_{0,j,k}-a_{0,k,j}))p_{0,j}=0
\end{equation}
where $k=1,\dotsc,n$, i.e. we have $n$ linear equations in the $n$ indeterminates $p_{0,1},\dotsc,p_{0,n}$.
Let $A$ be the matrix associated to the homogeneous linear system \eqref{eq:sis}; 
then $A$ is skew-symmetric, and therefore, 
if $\omega$ is general, the system \eqref{eq:sis} has only the zero solution if $n$ is even and has a vector space of solutions of dimension one if $n$ is odd. 

If $n\leq 5$, see the examples in Section \ref{examples}.

\end{proof}

\subsection{The congruence $X_\omega$ as a degeneracy locus}\label{congr_deg}

We recall now some facts about vector bundles on Grassmannians, and we apply them to the study of the congruence $X_\omega$.
We recall  the
universal exact sequence \eqref{seq:uq} on the Grassmannian $\GR=G(2,V)$: 
\begin{equation*}
 0\to \UU     \to     V\otimes \OO_{\GR}\to \QQ\to 0.
 \end{equation*}

Recall that, for the Pl\"ucker embedding, 
$\OO_{\GR}(1) \cong\det(\UU^*) \cong\det(\QQ)$.
Over $\GR$, we have also  
$
H^0(\QQ)=V$, 
$H^0(\UU^*)=V^*.$ 
We refer to \cite{--98W}
  for terminology and basic results on homogeneous bundles.
  
 \begin{theorem}\label{thm:qq1}
The congruence $X_\omega$ is the zero locus of a section of
$\QQ^*(1)$. If $\omega$ is general enough, $X_\omega$ 
is smooth of the expected dimension $n-1$, i.e. 
$X_\omega$ satisfies \ref{GC5}.
\end{theorem}

\begin{proof}
We use  the theorem of Borel-Bott-Weil, in
  particular  \cite[Theorem 4.1.8]{--98W}. 

 Here, this gives a natural isomorphism
\begin{equation*}
H^0(\QQ^*(1))\simeq \bw^3V^*.
\end{equation*}
Therefore a $3$-form $\omega$ on $V$ can be interpreted as a global section of 
 $\QQ^*(1)$ and it defines  a bundle map 
\begin{equation}\label{eq:asterisco}
\OO_{\GR}\xrightarrow{\varphi_\omega} \QQ^*(1).
\end{equation}
The degeneracy locus of $\varphi_\omega$  is therefore  $X_\omega$, which was defined in Section \ref{xomega} as $\{[L]\in \GR \mid \omega(L)=0\}$. By a Bertini type theorem,  for general $\omega$, $X_\omega$  has codimension equal to 
$\rk\QQ^*(1)=n-1$, and it is smooth. Moreover also  $\dim X_\omega=n-1$. 
\end{proof}

From Theorem \ref{thm:qq1} we deduce further global properties of  $X_\omega$.

\begin{theorem}\label{thm:qq2}
If $\omega$ is an alternating $3$-form on $\p(V)$ that satisfies \ref{GC4}, then $X_\omega$ is a Fano variety of index $3$ and
  dimension $n-1$ with Gorenstein singularities, which is smooth if
  $\omega$ is general \ref{GC5}. 
Moreover 
 the sheaf  $\OO_{X_\omega}$  has the Koszul locally free resolution:
\begin{equation}\label{eq:resolution}
0\to\OO_{\GR}(2-n)\to\bw^{n-2}(\QQ(-1))\to \dotsb \to \QQ(-1)\xrightarrow{{}^t\varphi_\omega} \OO_{\GR} \to \OO_{X_\omega} \to 0. 
\end{equation}
\end{theorem}
\begin{proof} We note that the Koszul complex associated with the
  section $\varphi_\omega$ of $\QQ^*(1)$ defining $X_\omega$
  is exact. The resolution follows observing that
  $\det(\QQ^*(1))^*\cong \det(\QQ(-1))\cong \OO_{\GR}(-n+2)$.  Then
  by adjunction,  the canonical sheaf $\omega_{X_\omega}$ of
  $X_\omega$ is locally free of rank $1$ and has the following expression 
\begin{equation*}
\omega_{X_\omega}\cong \omega_{\GR}\otimes\bw^{n-1} (\QQ^*(1))|_{X_\omega}
\cong \OO_{\GR}(-n-1+n-2)|_{X_\omega}
= \OO_{X_\omega}(-3). 
\end{equation*}
\end{proof}

\begin{coro}\label{cor:acm}
In the hypothesis of the preceding theorem, $X_\omega$ is arithmetically Cohen-Macaulay (aCM for short, in what follows) 
in its linear span $\Lambda_\omega$ and arithmetically Gorenstein (aG for short). 
\end{coro}

\begin{proof}
The first statement follows from resolution \eqref{eq:resolution}, because the exterior powers of $\QQ$
are aCM by Bott's theorem (cf.  \cite[Ch. 4]{--98W}). Since  $X_\omega$ is also subcanonical by Theorem \ref{thm:qq2}, then it is aG by \cite[Proposition 4.1.1]{M}.
\end{proof}

\subsubsection{The cohomology class of the congruence}\label{coho}

The cohomology class of $X_\omega$ is given by Porteous formula:
\begin{equation*}
[X_\omega]=c_{n-1}(\QQ^*(1))\cap [\GG].
\end{equation*}  
If 
\begin{equation*}
P^0\subset P^1\subset \dotsb \subset P^{n-1-i}\subset \dotsb \subset P^{n-j}\subset \dotsb\subset P^n=\p(V)
\end{equation*}
is a complete flag, 
the cohomology ring of $\GG$ has basis
$\sigma_{i,j}$,  $i=0,\dotsc,n-1$,  $j=0,\dotsc,i$,
where 
\begin{equation*}
\sigma_{i,j}\cap [\GG] =[\{[L]\in \GG| L\subset P^{n-j}, L\cap P^{n-1-i}\neq\emptyset\}].
\end{equation*}
Then 
\begin{equation*}
c_i(\QQ)=\sigma_i:=\sigma_{i,0},
\end{equation*}
 and therefore $c_i(\QQ^*)=(-1)^i\sigma_i$.
As above, we write formally the Chern polynomial of $\QQ^*$
\begin{equation*}
c_t(\QQ^*)=\prod_{i=1}^{n-1}(1-a_it),
\end{equation*}
where $a_1,\dotsc,a_{n-1}$ are formal symbols. Since 
$\sigma_1\cap [\GG]$ is the class of a hyperplane section of the Pl\"ucker embedding, 
we have also $c_1(\OO_{\GR}(1))=\sigma_1$, and we get
\begin{equation*}
c_t(\QQ^*(1))=(1+(a_1+\sigma_1)t)\dotsm (1+(a_{n-1}+\sigma_1)t),
\end{equation*}
and therefore, applying \cite[Lemma 2.1]{DP}, 
\begin{align*}
c_{n-1}(\QQ^*(1))&=(a_1+\sigma_1)\dotsm(a_{n-1}+\sigma_1)=\sum_{\ell=0}^{n-1}(-1)^{n-1-\ell}\sigma_{n-1-\ell}\sigma_1^\ell\\
&=\sum_{\ell=0}^{n-1}(-1)^{n-1-\ell}\sigma_{n-1-\ell}(\sum_{i=0}^{\left[\frac{\ell}{2}\right]}\left(
\binom{\ell}{i}\cdot\frac{\ell-2i+1}{\ell-i+1}
\right)
\sigma_{\ell-i,i})\\ 
&=:\sum_{\ell=0}^{{\left[\frac{n-1}{2}\right]}}d_\ell(n)\sigma_{n-1-\ell,\ell}.
\end{align*}
Here the integers $d_\ell(n)$ are defined by the last equality, and their collection is called the \emph{multidegree of $X_\omega$}.
By Poincar\'e duality
\begin{equation*}
\int_{[\GG]}\sigma_{n-1-\ell,\ell}\sigma_{n-1-\ell',\ell'}=\delta_{\ell,\ell'}
\end{equation*}
so the multidegree of $X_{\omega}$ is also defined by
\begin{align*}
d_\ell(n)&=\int_{[X_{\omega}]}\sigma_{n-1-\ell,\ell} & 0\leq \ell&\leq n-1.
\end{align*}
\begin{lemma}\label{recursion}  The multidegree $(d_\ell(n))$, $\ell=0,\dotsc,n-1$ satisfies the initial condition
\begin{align*}
d_0(2m)&=0, & d_0(2m-1)&=1, &  m&=2,3,4,\dotsc
\end{align*}
 and the recursion relation
\begin{equation*}
d_\ell(n)=d_{\ell-1}(n-1)+d_{\ell}(n-1)
\end{equation*}
 when $\ell=1,2,\dotsc,\left[\frac{n-1}{2}\right]$.
\end{lemma}
\begin{proof}
The degree $d_0(n)$ is simply the order of the congruence $X_\omega$, so the initial condition follows from Proposition \ref{prop:ordine}.

Next, recall, from Schubert calculus, that  $\sigma_{i,j}\sigma_1=\sigma_{i+1,j}+\sigma_{i,j+1}$ (when $i>j$).
Let $\omega=\omega_x+\beta_x\w x$. We choose a flag such that  $P^{n-\ell}\subset \p(V_x)\subset \p(V)$.  If $\ell>0$, then
$d_\ell(n)=\deg X_\omega\cap Z_\ell(n)$
where 
\begin{equation*}
Z_\ell(n)=\{[L]|L\subset P^{n-\ell}, L\cap P^{\ell}\neq \emptyset\} \subset G(2,V),
\end{equation*}
and has class 
\begin{equation*}
[Z_\ell(n)]=\sigma_{n-1-\ell,\ell}\cap [G(2,V)].
 \end{equation*}
But $P^{n-\ell}\subset \p(V_x)$, so 
$Z_\ell(n)\subset G(2,V_x)$.    The class of this subvariety in $G(2,V_x)$ is 
\begin{equation*}
[Z_\ell(n)]=\sigma_{n-2-\ell,\ell-1}\cap [G(2,V_x)],
  \end{equation*}
while
\begin{equation*}
[X_{\omega}]\cap [G(2,V_x)]=[\{[L]| \omega_x(L)=\beta_x(L)=0\}]= [X_{\omega_x}]\cap [H_{\beta_x}]\cap [G(2,V_x)],
 \end{equation*} 
where clearly $H_{\beta_x}$ is the hyperplane defined by $\{\beta_x(L)=0\}$. 
Computing the degree of $X_\omega\cap Z_\ell(n)$ on $G(2,V)$ and on $G(2,V_x)$ we get
\begin{align*}
d_\ell(n)&=\int_{[X_{\omega_x}]}\sigma_{n-2-\ell,\ell-1}\sigma_1=\int_{[X_{\omega_x}]} (\sigma_{n-1-\ell,\ell-1}+\sigma_{n-2-\ell,\ell})\\
&=d_{\ell-1}(n-1)+d_{\ell}(n-1).
\end{align*}
\end{proof}

For low values of $n$, $3\leq n \leq 9$, we get the following multidegree for $X_\omega$:
\begin{align}\label{multidegree}
(1,0)&, &(0,1)&, &(1,1,1)&, &(0,2,2)&, &(1,2,4,2)&, &(0,3,6,6)&, &(1,3,9,12,6).
\end{align}

The recursion of Lemma \ref{recursion}, and the initial degrees of \eqref{multidegree}, may be displayed in a triangle with initial entries
\begin{align*}
a_{(2k+1,0)}&=1,& a_{(2k,0)}&=0,& k&=0,1,2,\dotsc
\end{align*}
and 
\begin{align*}
a_{(i,j)}&=a_{(i,j-1)}+a_{(i-1,j)}& i&=1,2,\dotsc,\textup{ and } j=1,2,\dotsc,i.
\end{align*}
The multidegree of $X_\omega$ is identified as $(d_\ell(n))=(a_{(n-1-\ell,\ell)}),\ell=0,...,n-1$.

\begin{equation}\label{eq:tridn} (a_{i,j})=
 \begin{array}{cccccccccc} 
      1 & & & & & & & & & \\
      0 & 0 & & & & & & & & \\
      1 & 1 & 1& & & & & & & \\
      0 & 1 & 2&2 & & & & & & \\
      1 & 2 & 4&6 & 6& & & & & \\
      0 & 2 & 6&12 & 18&18 & & & & \\
      1 & 3 & 9& 21& 39& 57& 57& & & \\
      0 & 3 & 12&33 & 72& 129& 186 &186 & & \\
      1 & 4 & 16& 49& 121& 250& 436& 622& 622& \\
      0 & 4 & 20& 69& 190& 440& 876&1498 &2120 &2120 \\
   \end{array}.
   \end{equation}
In fact  
\begin{equation*}
a_{(n-1-\ell,\ell)}=d_\ell(n)=d_{\ell-1}(n-1)+d_{\ell}(n-1)=a_{(n-1-\ell,\ell-1)}+a_{(n-2-\ell,\ell)}.
  \end{equation*}

\subsubsection{The degree of the congruence and the Fine numbers}

The Fine numbers form the sequence $1,0,1,2,6,18,57,186,\dotsc$ see \cite{fine,DS}.

\begin{prop}\label{prop:tria}  
The degree of $X_\omega\subset \GG$ is the $n$-th Fine number  that can be read off 
as the $(n+1)$-st diagonal element 
 \begin{equation*}
\deg X_\omega=a_{(n-1,n-1)}=a_{(n-1,n-2)}=d_{n}(2n-1)
\end{equation*}
 in the above triangle of numbers.
  The multidegree of $X_\omega$ is given by the antidiagonals 
  \begin{equation*}
  (d_{0}(2m),d_{1}(2m),\dotsc,d_{m-1}(2m))=(a_{(2m-1,0)},a_{(2m-2,1)},\dotsc,a_{(m,m-1)})
    \end{equation*}
 when $n=2m$  and
    \begin{equation*}
  (d_{0}(2m-1),d_{1}(2m-1),\dotsc,d_{m-1}(2m-1))=(a_{(2m-2,0)},a_{(2m-3,1)},\dotsc,a_{(m-1,m-1)}).
    \end{equation*}
    when $n=2m-1$.
  \end{prop}

\begin{proof} 
  It remains to show that
   \begin{equation*}
\deg X_\omega=d_{m}(2m-1).
 \end{equation*}
  Let $\dim V=2m$, 
and let $\omega\in \bw^3V^*$ be a general $3$-form.  Assume $V_0^*\subset V^*$ is a general subspace of dimension $m+1$, and choose $x_{m+1},\dotsc,x_{2m-1}\subset V^*$ such that 
   \begin{equation*}
  V^*=V_0^*\oplus\gen{ x_{m+1},\dotsc,x_{2m-1}}.
   \end{equation*}
  Then 
   \begin{equation*}
  \omega=\omega_0+x_{m+1}\w\beta_{m+1}+\dotsm+x_{2m-1}\w\beta_{2m-1}+\alpha
   \end{equation*}
  where $\omega_0\in \bw^3V_0^*$, $\beta_i\in \bw^2V_0^*$, and $\alpha\in \bw^2\gen{ x_{m+1},\dotsc,x_{2m-1}}\w V^*$.
  Let $P^{m}=\{x_{m+1}=\dotsb=x_{2m-1}=0\}\subset \p(V)$ and let
  $Z_{m-1}(2m-1) =\{[L]| L\subset P^m\}$.
   Then
  \begin{align*}
  X_\omega\cap Z_{m-1}(2m-1)&=\{[L]| L\subset P^m, \omega_0(L)=\beta_{m+1}(L)=\dotsb=\beta_{2m-1}(L)=0\}\\
   &= X_{\omega_0}\cap H_{\beta_{m+1}}\cap\dotsm\cap H_{\beta_{2m-1}}.
   \end{align*}
   So 
    \begin{equation*}
   \deg X_{\omega_0}=\deg (X_\omega\cap Z_{m-1}(2m-1))=d_{m}(2m-1),
    \end{equation*}
 and the proposition follows.
  \end{proof}

\section{The fundamental locus}\label{sect:fund}

\begin{defi}\label{fund} The \emph{fundamental locus} of a congruence $X$ 
is the set of points of $\p(V)$ belonging to infinitely many lines of $X$. 
\end{defi}

It is clear that if $X$ has order zero, then its fundamental locus is simply the union of the lines of $X$.
We will see in a minute that the fundamental locus of the congruence $X_\omega$ has a natural structure of scheme 
as a suitable degeneracy locus.


\subsection{A skew symmetric matrix associated with a trilinear form}

The $3$-form $\omega\in \bw^3 V^*$ defines a natural map of vector
bundles on $\p(V)$. In this section we identify this map and its
degeneracy loci on $\p(V)$. 

The first observation is that
there is a natural isomorphism:
\begin{equation*}
H^0(\Omega^2_{\p(V)}(3)) \cong \bw^3 V^*.
\end{equation*}
This is provided again by Borel-Bott-Weil's theorem.
Now, in view of the natural isomorphisms
\begin{align*}
 H^0(\Omega^2_{\p(V)}(3))\subset& H^0(\Omega^1_{\p(V)}\otimes \Omega_{\p(V)}^1(3))\cong \Hom (\Omega^1_{\p(V)}(1)^*, \Omega_{\p(V)}^1(2))\\
\cong &\Hom (\TTT_{\p(V)}(-1), \Omega_{\p(V)}^1(2)),
\end{align*}
the $3$-form $\omega\in  \bw^3 V^*$ 
 determines a bundle map:
\begin{equation}
  \label{fomega}
  \phi_\omega\colon \TTT_{\p(V)}(-1)\to \Omega_{\p(V)}^1(2).
\end{equation}
 Note that $\phi_\omega$ is skew-symmetric, in the sense that 
\begin{equation*}
\phi_\omega^*=-\phi_\omega(1),
\end{equation*}
where $\phi_\omega^*: \TTT_{\p(V)}^1(-2)\to  \Omega_{\p(V)}(1)$
is the dual map of $\phi_\omega$. Indeed, by the above description,  
$H^0(\Omega^2_{\p(V)}(3))$ is just the skew-symmetric part of $\Hom (\TTT_{\p(V)}(-1), \Omega_{\p(V)}^1(2))$, and the map
induced by $\phi_\omega$ on the global sections 
\begin{equation*}
  f_\omega\colon H^0(\TTT_{\p(V)}(-1))\to H^0(\Omega_{\p(V)}^1(2))
\end{equation*}
is the map $f_\omega$ of \eqref{eq:fomega}. 

 We can interpret $\phi_\omega$ in more concrete terms,  in the following way:  from Euler sequence twice  we get the diagram
\begin{equation} \label{diagramma}
\begin{CD}
\TTT_{\p(V)}(-1)@>\phi_\omega>> \Omega_{\p(V)}^1(2)\\
@AAA  @VVV \\
V\otimes \OO_{\p(V)}@>M_\omega>> V^*\otimes \OO_{\p(V)}(1)
\end{CD}
\end{equation}
where $M_\omega$ is obtained by composition, so we can think of $M_\omega$ as 
a $(n+1)\times(n+1)$ skew-symmetric matrix with linear entries on 
$\p(V)$.  In fact, this matrix $M_\omega$ in suitable coordinates is the matrix 
defined in \eqref{eq:momega}. 
\begin{rema}\label{rank1} 
Since $\phi_\omega$ is a skew-symmetric map between two  bundles of rank $n$,  
$M_\omega$ has rank at most $n$ when $n$ is even, and $n-1$ when $n$ is odd.  
We will see in next Section that, 
if $\omega$ is general, then these are the generic ranks of $M_\omega$.
\end{rema}

 The map $\phi_\omega$ will be considered again in Section \ref{proj bundles}, 
where we will describe its kernel and cokernel both in the cases $n$ even and $n$ odd.


\subsection{Degeneracy loci}\label{deg loci}


Let us now study the degeneracy loci of $\phi_\omega$ (or, which is the same, of $M_\omega$). 
Let us denote  by $M_r=\{P\in\p(V)\mid \rk{{\varphi_\omega}_\mid}_P\leq r\}$ the locus of the points where $\phi_\omega$ has rank 
at most $r$.

We endow  $M_r$
with its natural scheme structure, given by the principal Pfaffians,  
$(r+2)\times (r+2)$ if $r$ is even, and $(r+1)\times(r+1)$ if $r$ is odd.  
Notice that in this last case $M_r=M_{r-1}$. Moreover, by Remark \ref{rank1}, $\p(V)=M_n$ if $n$ is even and $\p(V)=M_{n-1}$ if $n$ is odd.

The degeneracy loci of a (twisted) skew-symmetric map of vector bundles is studied in \cite{HT}: 
in particular, in \cite[Theorem 10(b)]{HT} the
cohomology class of each degeneracy locus is computed. In our setting, the class of $M_r$, if $r$ is even, is 
\begin{equation}\label{class}
[M_r]=\det\begin{pmatrix}c_{n-r-1}& c_{n-r}& \cdots&\\
c_{n-r-3}& c_{n-r-2}& &\\
\vdots&& \ddots& \\
&&&c_1
\end{pmatrix}
\end{equation}
where $c_i=c_i(\Omega_{\p(V)}^1(1)\otimes \sqrt{\OO_{\p(V)}(1)})$ and $\sqrt{\OO_{\p(V)}(1)}$ has to be thought, formally, by the 
\emph{squaring principle}, as a 
line bundle such that $\sqrt{\OO_{\p(V)}(1)}\otimes\sqrt{\OO_{\p(V)}(1)}=\OO_{\p(V)}(1)$.

Now, $c_t(\OO_{\p(V)}(1))=1+ht$, where $c_t$ as usual denotes the Chern polynomial and $h$ is the hyperplane class, so, if we put 
$c_t(\sqrt{\OO_{\p(V)}(1)})=1+at$, where $a$ is a formal symbol, we have 
$1+ht=(1+2at)$, which implies that $c_t(\sqrt{\OO_{\p(V)}(1)})=1+\frac{h}{2}t$; in other words,  in expression \eqref{class},  we can write
$c_i=c_i(\Omega_{\p(V)}^1\otimes \OO_{\p(V)}(\frac{3}{2}))$.

We recall, for example by the Euler sequence, that 
\begin{equation*}
c_t(\Omega_{\p(V)}^1)=(1-ht)^{n+1}
\end{equation*}
and therefore 
\begin{align*}
c_i(\Omega_{\p(V)}^1)&=(-1)^i\binom{n+1}{i}h^i, & i=0,\dotsc, n.
\end{align*}
On the other hand, it is easy to see, reasoning as above, that $c_t(\OO_{\p(V)}(\frac{3}{2}))=1+\frac{3}{2}ht$.

Finally, if we write formally $c_t(\Omega_{\p(V)}^1)=\prod_{i=1}^n(1+a_iht)$, recalling the formula for the Chern polynomial of a tensor product,
we obtain
\begin{align*}
c_t(\Omega_{\p(V)}^1\otimes \OO_{\p(V)}(\frac{3}{2})) &=\prod_{i=1}^n(1+(a_i+\frac{3}{2})ht\\
&=\left(\sum_{k=0}^i(-1)^k\frac{3^{i-k}(1+k)}{2^{i-k}(n+1-k)}\binom{i+1}{k+1}\right)\binom{n+1}{i+1} h^i. 
\end{align*}
In particular
\begin{align*}
c_1&=(\frac{n}{2}-1)h,\\
c_2&=\frac{n^2-5n+12}{8}h^2,\\
c_3&=\frac{n^3-9n^2+44n-108}{48}h^3, && \textup{etc.}
\end{align*}

\subsubsection{Even $n$}

In this case 
\begin{equation*}
[M_{n-2}]=c_1(\Omega_{\p(V)}^1\otimes \OO_{\p(V)}(\frac{3}{2}))=(\frac{n}{2}-1)h
\end{equation*}
i. e. $M_{n-2}$ is a hypersurface of degree $\frac{n}{2}-1$. 

By a Bertini type theorem, its singular locus is contained in $M_{n-4}$ and it is equal to $M_{n-4}$ if $\omega$ is general, for which we have 
\begin{equation*}
[M_{n-4}]=\det\begin{pmatrix}c_3& c_4& c_5\\
c_1& c_2& c_3\\
0&1& c_1
\end{pmatrix}=c_1(c_2c_3-c_1c_4+c_5)-c_3^2
\end{equation*}
and in particular it has codimension $6$ in $\p(V)$; therefore, $M_{n-2}$ is smooth up to $\p^4$, 
and we expect that it is singular of dimension $0$ in
 $\p^6$ and of dimension $2$ in $\p^8$. But, making explicit calculations, we obtain
\begin{equation*}
[M_{n-4}]=\frac{(n)(n-6)(n+1)(n+2)(n^2-9n+44)}{2880}h^6
\end{equation*}
and therefore  $M_{n-2}$ is smooth also in $\p^6$. Actually for $n=6$
we get a smooth quadric $5$-fold as degeneracy locus, as we shall see in Example \ref{ex4}.

Moreover $M_{n-6}$ has codimension $15$, so $M_{n-4}$ is smooth up to $n=14$.


\subsubsection{Odd $n$} 

 In this case we have 
\begin{align}\label{degreeF}
[M_{n-3}]=\det\begin{pmatrix}c_2& c_3\\
1& c_1
\end{pmatrix}&=c_1c_2-c_3
=\frac{n^3-6n^2+11n+18}{24}h^3=(\frac{1}{4}\binom{n-1}{3}+1) h^3
\end{align}
i.e. $M_{n-3}$ is a codimension $3$ subvariety of $\p(V)$ of degree $\frac{1}{4}\binom{n-1}{3}+1$.

By a Bertini type theorem, its singular locus is contained in $M_{n-5}$ and it is equal to $M_{n-5}$ if $\omega$ is general, for which we have 
\begin{align*}
[M_{n-5}]=&\det\begin{pmatrix}
c_4& c_5& c_6& c_7\\
c_2& c_3& c_4& c_5\\
1& c_1& c_2& c_3\\
0&0&1& c_1
\end{pmatrix}
=c_2c_3c_5 + 2c_1c_4c_5-c_1c_3c_6-c_1c_2c_7-c_5^2+c_3c_7
\end{align*}
and in particular it has codimension $10$ in $\p(V)$; therefore,
$M_{n-3}$ is smooth up to $\p^9$, and is generically singular in dimension $1$ in
 $\p^{11}$.  As above, we can make explicit calculations, obtaining
\begin{equation*}
[M_{n-5}]=\frac{n(n-1)(n + 1)^2(n+2)(n+3)(n^4-26n^3+311n^2-1966n+5400)}{4838400}h^{10}.
\end{equation*}

\subsection{Equations of the fundamental locus} 
Recall from \eqref{diagramma} that 
the map
\begin{align*}
\p(f_\omega)\colon  \p(V)&\to \p(\bw^2 V)^*\\ 
 P&\mapsto [M_\omega(P)]=\left[\left(\sum_k( (-1)^{i+j-1}(a_{i,j,k}-a_{i,k,j}+a_{k,i,j}))x_k(P) \right)_{\substack{i=0,\dotsc,n \\ j=0,\dotsc,n}}\right],
\end{align*}
of \eqref{eq:emb} is the projectivised map on global sections of the bundle map 
 $\phi_\omega\colon \TTT_{\p(V)}(-1)\to \Omega_{\p(V)}^1(2).$ 
The degeneracy locus $M_{r}$ defined at the beginning of Section \ref{deg loci} may therefore be interpreted as
  \begin{equation*}
M_{r}=\{ P\in \p(V)|\rk M_\omega(P)\leq r\}.
\end{equation*}

\begin{nota}
 Let $F_\omega\subset \p(V)$ denote the locus where the map $\phi_\omega$ drops rank.
As a degeneracy locus $F_\omega$ has a natural scheme structure.  
In fact, it is a scheme structure on the fundamental locus of $X_\omega$ (see Definition \ref{fund}):
\end{nota}

\begin{prop}\label{prop:fonda} 
 Let $\omega\in \bw^3V^*$ be general,
 and let $M_{r}\subset \p(V)$ be the degeneracy locus where $\phi_\omega$ has rank at most $r$.  
Then  $F_\omega=M_{n-2}$ if $n$ is even, and $F_\omega=M_{n-3}$ if $n$ is odd, and is a 
scheme structure on the fundamental locus of $X_\omega$. The lines of $X_\omega$ through a  point 
$P\in M_{r}\setminus M_{r-2}$, $r$ even, form a star in a $\p^{n-r}$.

If $n$ is even, then $F_\omega$ is a  hypersurface of degree ${\frac{n}{2}}-1$, it is smooth if $n\leq 6$ and it is singular in codimension 
$5$ if $n\ge 8$. 
If $n$ is odd, then $F_\omega$ has codimension $3$ and degree $\frac{1}{4}\binom{n-1}{3}+1$, it is smooth if $n\leq 9$ and it is singular in 
codimension $7$ if $n\ge 11$.
\end{prop}

\begin{proof}  Let  $P\in \p(V)$. 
Without loss of generality we may suppose that $P=[1,0,\dotsc,0]=[e_0]$, 
and evaluate the matrix $M_\omega$ at the point $P$:
\begin{equation}\label{eq:mop}
M_\omega(P)=\left( (-1)^{i+j-1}(a_{0,i,j}) \right)_{\substack{i=0,\dotsc,n \\ j=0,\dotsc,n}},
\end{equation}
where we have set $a_{0,j,i}:=-a_{0,i,j}$ if $j>i$. 
Then, if we cut the first row and the first column of $M_\omega(P)$---which are zero---we obtain a submatrix that 
coincides with the matrix $A$ associated to the homogeneous linear system \eqref{eq:sis} 
(proof of Proposition \ref{prop:ordine}), whose zeros define exactly the lines passing through $P$.

On one hand, since $M_\omega$ defines a map on global sections between bundles of rank $n$, 
the rank of $M_\omega(P)$ is at most $n$.  
On the other hand, 
$M_\omega(P)$ is skew symmetric, so it has even rank, so $P\in F_\omega$ if and only if $\rk (M_\omega(P))\leq n-2$ when $n$ is even, 
and $\rk (M_\omega(P))\leq n-3$ when $n$ is odd.
Furthermore, if $M_\omega(P)$ has rank $r$, then there is a linear $\p^{n-r-1}\subset \GR$ parametrising lines of the congruence 
$X_\omega$ 
that pass through $P$, and vice versa.
The dimension, degree formulas for $F_\omega$ and its singular locus
 follow from Section \ref{deg loci}. 
\end{proof}

\begin{rema}
If $n$ is odd, i.e. when the order of $X_\omega$ is one, another scheme structure on $F_\omega$ is described in \cite{DP0}. 
It comes from the interpretation of $F_\omega$ as branch locus of the projection from the incidence correspondence to $\p(V)$. 
With this structure $F_\omega$ is non-reduced because its codimension is $3$, 
while from this point of view its expected codimension would be $2$.
\end{rema}

The fundamental locus $F_\omega$ of the congruence $X_\omega$ has a natural interpretation in terms 
of the secant varieties of the Grassmannian. Consider the natural filtration 
 of $\p(\bw^2 V)$ by the secant varieties of the Grassmannian:
\begin{equation*}
\GR\subset S^1\GR \subset \dotsb\subset S^{r-1}\GR\subset S^{r}\GR= \p(\bw^2 V),
\end{equation*}
where  $r=[\frac{n-1}{2}]$.  In the dual space $\p(\bw^2 V)^*$ there is a dual
 filtration. Write $\GR'=G(n-1,V^*)$. Then the filtration can be interpreted in the form:
\begin{equation*}
\GR'\subset S^1\GR' \subset \dotsb\subset S^r\GR' =\GR^*\subset \p(\bw^2 V^*).
\end{equation*}
The last secant variety $\GR^*$ is the dual of $\GR$ parametrising its tangent hyperplanes. 
If $n$ is odd, the general tangent hyperplane is tangent at one point only (corresponding to a line of $\p(V)$), 
the previous variety of the filtration parametrises hyperplanes which are tangent along the lines of a $3$-space, 
and so on, until the smallest one corresponds to hyperplanes tangent along the lines of a $\p^{n-2}$. 
If $n$ is even, the description is similar, but the general tangent hyperplane is tangent along the lines of a $2$-plane, and so on.

\begin{coro}\label{fundamental} 
Let $\omega\in \bw^3V^*$ be a $3$-form 
such that $f_\omega\colon V\to \bw^2V^*$ is injective (condition \ref{GC2}).
If we identify $\p(V)$ with $P_\omega=\im(\p(f_\omega))\subset \p( \bw^2V^*)$, 
i.e. the space of  linear equations defining the congruence $X_\omega\subset \p(\bw^2V)$,
then 
\begin{align*}
M_{2k}&=P_\omega\cap S^k\GR', &  k&\geq 0.
\end{align*}
In particular, 
the fundamental locus $F_\omega$ is the locus of equations whose rank, as a skew-symmetric matrix, is $<n-1$, 
and, when $\omega$ is general, the singular locus of $F_\omega$ is the locus of equations whose rank is $<n-3$.
\end{coro}
\begin{proof}
When $f_\omega$ is injective, $P\mapsto [M_{\omega}(P)]$ defines the identification of $\p(V)$ with $P_\omega$, 
and so the corollary follows immediately from Proposition \ref{prop:fonda}, noting that $S^k\GR'\setminus S^{k-1}\GR'$ is smooth.
\end{proof} 

Finally, we state a theorem 
illustrating the geometric connection between a congruence of order $1$ and its fundamental locus.

\begin{theorem}\label{thm:pippo}
Let $\omega$ be a  $3$-form that satisfies \ref{GC2} on $\p(V)$, with $\dim(V)=n+1$ even. 
Then $X_\omega$ is the closure of the family of $(\frac{n-1}{2})$-secant lines of $F_\omega$.
\end{theorem}

\begin{proof}
Let $n=2m+1$ and consider as usual the map $f_\omega \colon V\to \bw^2V^*$ of \eqref{eq:fomega}.  
Up to a change of coordinates, a  line $L$ can be written as $L =\gen{e_0,e_1}$. We can interpret the elements of $\bw^2V^*$ as 
in \eqref{eq:momega} as $(n+1)\times(n+1)$ skew-symmetric matrices, and therefore, since we are supposing \ref{GC2},  the image of $L$ 
is a pencil of $(n+1)\times(n+1)$ skew-symmetric matrices of the form $M(s,t)=sf_\omega(e_0)+tf_\omega(e_1)$, and 
each matrix in the pencil has rank at most $n-1$ (see Remark \ref{rank1}). 

Now, if $[L]\in X_\omega$, then the matrices $M(s,t)$ all contain the same $2$-subspace $L$ in their kernel.  
So on the quotient space by $L$, the matrices are $(n-1)\times(n-1)$ of generic rank $n-1$.  They are all skew-symmetric, so in the pencil there are 
$(n-1)/2=m$ matrices  of smaller rank. 
 
The same argument can be reversed.  If there are $m$ matrices $M(s,t)$ of  rank smaller than $m$, the principal Pfaffians of order $m$ of these matrices have $m$ common zeroes. Therefore their GCD is a non-zero homogeneous polynomial of degree $m$, and we conclude that they are all proportional. 
So the matrices $M(s,t)$ must have a common rank $2$ subspace $L$ in their kernel. From Lemma \ref{lem} it follows that $[L]\in X_\omega$.
\end{proof}
\begin{rema} When $F_\omega$ is smooth, the result also follows from  \cite[Theorem 4.6]{P2}.
\end{rema}


\subsection{Examples}\label{examples}


When $n\leq 7$
the natural group action of $\SL(V^*)$ on $\bw^3V^*$ has finitely many orbits.  In particular, there is a unique open orbit, so we list in the examples below, for $3$-forms $\omega$ of this open orbit,  the  congruence $X_\omega$ and fundamental locus $F_\omega$ of $X_\omega$.
We start considering the case $n=3$ in which Lemma \ref{lem} does not apply. 
\begin{exem}\label{ex1}
If $n=3$, $\omega\in \bw^3 V^*$ is totally decomposable, so without loss of generality we may assume 
\begin{equation*}
\omega=x_1\w x_2\w x_3;
\end{equation*}
the equation \ref{eq:px} reduces to 
\begin{equation*}
x_1p_{23}-x_2p_{13}+x_3p_{12}=0.
\end{equation*}
So $X_\omega=\{p_{12}=p_{13}=p_{23}=0\}\subset \GR$ is the $\alpha$-plane of lines passing through the point $[e_0]: \{x_1=x_2=x_3=0\}$,  which is $F_\omega$ in $\p(V)=\p^3$. 
\end{exem}

\begin{exem} \label{ex2}
If $n=4$,
there are two (non-trivial) orbits.  If $\omega\in \bw^3 V^*$ belongs to the open orbit it is the product of a $1$-form and a general $2$-form, so without loss of generality we may assume 
\begin{equation*}
\omega=x_0\w (x_1\w x_2+x_3\w x_4).
\end{equation*}
The equation \ref{eq:px} reduces to 
\begin{equation*}
x_0(p_{12}+p_{34})-x_1p_{02}+x_2p_{01}-x_3p_{04}+x_4p_{03}=0.
\end{equation*}
So $X_\omega=\{p_{12}+p_{34}=p_{01}=p_{02}=p_{03}=p_{04}=0\}\subset \GR$ is a smooth quadric threefold, a smooth hyperplane section of the Grassmannian of lines in $\{x_0=0\}=\p(V_{x_0})\subset \p(V)$. 
In particular, the fundamental locus $F_\omega=\p(V_{x_0})$. 
 As a congruence 
of $\p^4$ $X_\omega$ has order zero and class one, i.e. it is a Schubert variety: $[X_\omega]=\sigma_{2,1}$. 

If instead $[\omega]$ lies in the closed orbit, it is totally decomposable, so we can suppose $\omega= x_1\w x_2\w x_3$, and  we deduce the equations
$p_{12}=p_{13}=p_{23}=0$, i.e. the condition to be incident to the line $\{x_1=x_2=x_3=0\}$. In this case, $\dim(X_\omega)=4$.
\end{exem}

\begin{exem}\label{ex3}
Let $n=5$ and $\omega\in \bw^3 V^*$. In this case there are $4$ orbits.  They are described in \cite{S} (see \cite{AOP} and references therein for 
modern accounts). 
In particular, it is shown that the secant variety of $G(3,6)$ is the whole $\p^{19}$, so for $\omega$ in the open orbit, we may assume that 
\begin{equation*}
\omega= x_0\w x_1 \w x_2+x_3\w x_4 \w x_5
\end{equation*}
which means $a_{0,1,2}=a_{3,4,5}$ and $a_{i,j,k}=0$ for $(i,j,k)\neq  (0,1,2), (3,4,5)$. 

From \eqref{eq:important} we deduce $p_{0,1}=p_{0,2}=p_{1,2}=p_{3,4}=p_{3,5}=p_{4,5}=0$, so  $X_\omega$ is contained in a reducible linear congruence 
and is given by the lines which meet the two planes 
$\alpha=\{x_0=x_1=x_2=0\}$ and $\beta=\{x_3=x_4=x_5=0\}$ in general position, so $X_\omega=\p^2\times \p^2$ and $F_\omega=\alpha\cup \beta$.

Since the Schubert cycle which represents the lines meeting a plane is $\sigma_2$, 
by Pieri's formula we have that in the Chow ring of the Grassmannian, our congruence 
is $[X_\omega]=\sigma_2^2=\sigma_4+\sigma_{3,1}+\sigma_{2,2}$, which confirms our calculations in \eqref{multidegree} that  its multidegree is $(1,1,1)$. 
We remark also that $\rk(M_\omega)=4$ for 
points of $\p^5$ not in the fundamental locus, and $\rk(M_\omega)=2$ for the points in the fundamental locus, see also Proposition \ref{prop:fonda}.

If $\omega$ belongs to the second largest orbit, in which case $[\omega]$ is a point on a projective tangent space to $G(3,6)$, then we may assume that 
\begin{equation*}
\omega=x_0\w x_1 \w x_2+x_2\w x_3 \w x_4 + x_4\w x_5\w x_0.
\end{equation*}
Also in this case $X_\omega$ has dimension $4$, while for $\omega$ in the remaining two orbits, the dimension of $X_\omega$ is $>4$.
\end{exem}

\begin{exem}\label{ex4}
Let $n=6$ and $\omega\in \bw^3 V^*$. In this case there is 
an open orbit and $8$ other (non-trivial) orbits, see \cite{Sch} and\cite{AOP} for explanations and references.  If $\omega$ belongs to the open orbit, we may assume that 
\begin{equation*}
\omega= x_1\w x_2 \w x_3+x_4\w x_5 \w x_6+x_0\w(x_1\w x_4+x_2\w x_5+x_3\w x_6).
\end{equation*}
The stabiliser of $\omega$ is the simple Lie group $\G2_2$.  The
congruence $X_\omega\subset\GR$ is 
the homogeneous variety $\G2_2\subset\p^{13}$, the $5$-dimensional closed orbit of the projectivised adjoint representation of  $\G2_2$,
a
Fano manifold of index $3$.  This congruence has order $0$, and the
fundamental locus is a smooth quadric $5$-fold in $\p(V)$  (\cite{FH}, Ch.22, \cite{KR}, \cite{Mukai}).
We will say more on this example further on, cf. Example \ref{G2}.
\end{exem}

\begin{exem}
  For $n=7$, there is an open orbit and $21$ other non-trivial orbits 
(\cite{Gurevich35}, \cite{Gurevich}, \cite{ozeki}, \cite{Djokovic}, \cite{Holweck}). 
A representative of the open orbit, according to Ozeki,  is 
\begin{equation*}
 \omega=x_0\w x_1 \w x_2+x_0\w x_3\w x_4+x_1\w x_3\w x_5+x_1\w x_6\w x_7+x_2\w x_3\w x_6+ x_2\w x_5\w x_7+x_4\w x_5\w x_6. 
\end{equation*}
\DJ okovi\'c gives a different representative: 
\begin{equation*}
\omega=x_0\w(x_1+x_2)\w x_3+x_1\w x_4\w x_5+x_2\w x_6\w x_7+x_0\w x_4\w x_6+x_3\w x_5\w x_7.
\end{equation*}
The variety $X_\omega\subset\GR(2,V)$ parametrises the trisecant lines
of a general projection in $\p^7$ of the Severi variety
$\p^2\times\p^2\subset \p^8$ (\cite{Iliev-Manivel}). This projection
is $F_\omega$.
The computation $\deg(X_\omega)=57$ is \cite[Proposition 4.6]{Iliev-Manivel}.
\end{exem}

\begin{exem} \label{eg:n=8}
  For $n=8$, there are infinitely many orbits. They are described in \cite{VE}. It is still possible to write explicitly a general $3$-form $\omega$. Indeed, there is  a continuous   family of semi-simple orbits, depending on $4$ parameters, that can be explicitly described. Their union is a Zariski-dense  open subset in $\bw^3 V^*$. To write a $3$-form $\omega$ in this family, we introduce the following notation: 
\begin{align*} 
p_1&= x_0\w x_1 \w x_2+x_3\w x_4\w x_5+x_6\w x_7\w x_8\\
p_2&= x_0\w x_3 \w x_6+x_1\w x_4\w x_7+x_2\w x_5\w x_8\\
p_3&= x_0\w x_4 \w x_8+x_1\w x_5\w x_6+x_2\w x_3\w x_7\\
p_4&= x_0\w x_5 \w x_7+x_1\w x_3\w x_8+x_2\w x_4\w x_6\\
\end{align*}
Then a general $\omega=\lambda_1p_1+\lambda_2p_2+\lambda_3p_3+\lambda_4p_4$, where the coefficients satisfy $\lambda_1\lambda_2\lambda_3\lambda_4\neq 0$ and other explicit open conditions. For more details see \cite{VE}.

The moduli space of
  alternating trilinear forms is related to the moduli space of curves
  of genus $2$, cf.  \cite{GS15,GSW13} and references therein.
  The fundamental locus is a Coble cubic in $\p^8$, whose singular
  locus is an Abelian surface given as Jacobian of a curve of genus
  $2$ with a $(3,3)-$polarisation.
  More precisely, the moduli space of alternating 3-forms obtained as
  GIT quotient $\bw^3 V^* \sslash \GL(V)$ contains a dense subset 
  which is isomorphic to the moduli space of genus 2 curves C with a
  marked Weierstrass point. 

  The topological Euler characteristic of the $7$-fold $X_\omega$ is $0$.
\end{exem}

\begin{exem}
  The case $n=9$ has been studied by Peskine. The congruence $X_\omega$ is formed by the $4$-secant lines of the fundamental locus $F_\omega$, 
that is a smooth variety of dimension $6$ in $\p^9$, also called \emph{Peskine variety}. $F_\omega$ is not quadratically normal. 
In fact it is on the border of Zak's conjectures on $k$-normality (cf. \cite[Conjecture 1]{Zak}).  
$X_\omega$ has been considered also in \cite{DV10} in a construction of
  hyper-K\"ahler fourfolds.
\end{exem}

\bigskip

{\sc Problem.}
  Compute the Hodge numbers of $X_\omega$ for general $\omega$. The
  topological Euler characteristic of $X_\omega$ for $n=2t+2$ for $t=3,4,5,6$
  equals $0$, $-254$, $-8412$, $-284598$, so in these cases the
  derived category of $X_\omega$ cannot admit a full exceptional sequence.
  Is it true that the same thing happens for any $n \ge 8$?

\medskip

\subsection{Incidence varieties and projective bundles}\label{proj bundles}

Here we look more closely at the relation between a congruence and its
fundamental locus. This has been developed also in \cite{Han,P2}.

Consider the projective bundle $\XX=\p(\UU^*)$ over the
Grassmannian $\GR$. 
This bundle can be seen as the universal line
over $\GR$,
i.e. the point-line incidence variety in $\p^n \times \GR$.
Let  $\OO_\XX(\ell)$ be the tautological relatively ample line bundle on $\XX$ and 
\begin{equation*}
\lambda\colon \XX \to \GR
\end{equation*}
be the projection. It is well-known that there is a canonical isomorphism 
\begin{equation*}
\p(\Omega^1_{\p(V)}(2)) \simeq \XX.
\end{equation*}
Let $\OO_{\XX}(h)$ be the relatively ample line bundle on $\p(\Omega^1_{\p(V)}(2))$ and 
\begin{equation*}
\mu\colon  \p(\Omega^1_{\p(V)}(2))\to \p(V)
\end{equation*}
be the natural projection. Then, we have 
\begin{align}
  \label{OO(1)}
\lambda^*(\OO_{\GR}(1)) &\simeq \OO_\XX(h), & \mu^*(\OO_{\p(V)}(1)) &\simeq \OO_\XX(\ell).  
\end{align}
For brevity, we often denote the pull-back of a bundle $\E$ on $\p(V)$
to $\XX$ also by $\E$ omitting the symbol $\mu^*$, and likewise for $\GR$.

Again we see $\omega \in \bw^3 V^*$ as an element of
$H^0(\XX,\QQ^*(1))$ under the isomorphism:
\begin{equation*}
\bw^3 V^*  \simeq H^0(\GR,\QQ^*(1))\simeq H^0(\XX,\QQ^*(1)).
\end{equation*}
We let $I_\omega$ be the zero locus of $\omega$ in this sense, i.e. 
the zero-locus of the pull-back of $\varphi_\omega$ to $\XX$.
Clearly $I_\omega \simeq \p(\UU^*\mid_{X_\omega})$, i.e. $I_\omega$ is the point-line  incidence variety restricted to $X_\omega$.

\subsubsection{Locally free resolution of the fundamental locus}

The $3$-form $\omega$ can be considered as global section of $\Omega^2(3)$
over $\p(V)$, which is to say as the skew-symmetric morphism
$\phi_\omega$ of \eqref{fomega}.
Write $\C_\omega$ for the cokernel sheaf of $\phi_\omega$.
Looking back at \eqref{diagramma} we may write, for even $n$, the exact sequence:
\begin{equation}
  \label{f-even}
  0 \to T_{\p(V)}(-1) \to \Omega^1_{\p(V)}(2) \to \C_\omega \to 0,
\end{equation}
   and for odd $n=2m+1$, the resolution:
\begin{equation}
  \label{f-odd}
   0 \to \OO_{\p(V)}(1-m) \to T_{\p(V)}(-1) \to \Omega^1_{\p(V)}(2) \to
   \II_{F_\omega/\p(V)}(m) \to 0,
 \end{equation}
   where we used $\C_\omega \simeq \II_{F_\omega/\p(V)}(m)$.

   \begin{coro}
     If $n$ is odd and $\omega$ satisfies \ref{GC4}, then the
     fundamental locus $F_\omega$ is a Fano variety and
     $\omega_{F_\omega} \simeq \OO_{F_\omega}(-3)$.
   \end{coro}

   \begin{proof}
     Let $n=2m+1$. Since $\omega$ satisfies \ref{GC4} we have the
     exact sequence \eqref{f-odd} and $F_\omega$ is a subvariety of
     $\p(V)$ of pure codimension $3$. We have:
     \begin{equation*}
     \EExt^2_{\p(V)}(\II_{F_\omega/\p(V)},\OO_{\p(V)}(-n-1)) \simeq \omega_{F_\omega}.
   \end{equation*}
   On the other hand, dualising the self-dual exact sequence
     \eqref{f-odd}, we easily get :
     \begin{equation*}
     \EExt^2_{\p(V)}(\II_{F_\omega/\p(V)}(m),\OO_{\p(V)}) \simeq \OO_{F_\omega}(m-1).
   \end{equation*}
   Therefore:
     \begin{equation*}
     \omega_{F_\omega} \simeq \OO_{F_\omega} (m-1+m-n-1) \simeq \OO_{F_\omega} (-3).
   \end{equation*}
 \end{proof}

\subsubsection{Incidence variety and projectivised cokernel sheaf}

The main feature of the cokernel sheaf $\C_\omega$ is that it
recovers the tautological $\p^1$-bundle over $X_\omega$, the
incidence variety $I_\omega$.

 \begin{prop} \label{blup} 
   There is an isomorphism:
    \begin{equation*}
I_\omega \simeq \p(\C_\omega).
  \end{equation*}
 \end{prop}

 \begin{proof}
   We noticed that $I_\omega$ is the zero-locus in $\XX$ of
   $\varphi_\omega$, so our goal will be to describe
   also $\p(\C_\omega)$ this way. 
   So first of all recall the tautological exact sequences:
   \begin{align*}
     & 0 \to \K \to \Omega^1_{\p(V)}(2 \ell) \to \OO_\XX(h) \to 0, \\
     & 0 \to \OO_\XX(h-\ell) \to \UU^* \to \OO_\XX(\ell) \to 0,
   \end{align*}
   where the vertical tangent bundle $\K$ is defined by the sequence.
   It is clear that these fit into a commutative exact diagram:
    \begin{equation*}
   \xymatrix@-3ex{
     0\ar[r] & \K(h-\ell) \ar[r] \ar[d] & \Omega^1_{\p(V)}(\ell + h) \ar[r] \ar[d]& \OO_\XX(2h-\ell) \ar[d] \ar[r] & 0 \\
     0\ar[r] &  \QQ^*(h) \ar[r] & V^*\otimes \OO_\XX(h) \ar[d] \ar[r] &  \UU^*(h) \ar[d] \ar[r] & 0 \\
     && \OO_\XX(h+\ell) \ar@{=}[r] & \OO_\XX(h+\ell).
   }
    \end{equation*}
   This shows:
   \begin{equation*} 
\QQ^*(h) \simeq \K(h-\ell).
    \end{equation*}
   
   Remark that, by the sequences \eqref{f-even} and
   \eqref{f-odd}, the variety $\p(\C_\omega)$ is cut in $\XX \simeq
   \p(\Omega^1_{\p(V)}(2))$ by $T_{\p(V)}(-1)$ linearly on the
   fibres of $\mu$, i.e. it is the zero-locus of
   a section $\psi_\omega : T_{\p(V)}(-\ell) \to \OO_\XX(h)$.
   Notice that:
     \begin{equation*}
\psi_\omega \in H^0(\XX,\Omega^1_{\p(V)}(h+\ell)) \simeq
    H^0(\p(V),\Omega^1_{\p(V)} \otimes \Omega^1_{\p(V)}(3)).
 \end{equation*}

    Observe that $\psi_\omega$ lies in the skew-symmetric part
    $H^0(\p(V),\Omega^2_{\p(V)}(3))$. In other words, the image of
    $\psi_\omega$ in the summand $H^0(\p(V),S^2 \Omega^1_{\p(V)}(3))$
    is zero. Now:
    \begin{equation*}
    H^0(\p(V),S^2 \Omega^1_{\p(V)}(3)) \simeq H^0(\XX,\OO_\XX(2h-\ell)),
     \end{equation*}
    so $\psi_\omega$ goes to zero under the projection
    $\Omega^1_{\p(V)}(\ell + h) \to  \OO_\XX(2h-\ell)$, i.e. it lies
    in $H^0(\XX,\K(h-\ell)) \simeq H^0(\XX,\QQ^*(h))$. This is
    compatible with the isomorphism $H^0(\p(V),\Omega^2_{\p(V)}(3))
    \simeq \bw^3 V^* \simeq H^0(\GR,\QQ^*(1))$, so that $\psi_\omega$
    agrees with $\varphi_\omega$.
 \end{proof}
Let again $M_{r}\subset \p(V)$ be the locus of points where $\phi_\omega$ has rank at most $r$.  Then Proposition \ref{blup} and the sequences \ref{f-even} and  \ref{f-odd} yield
\begin{coro} When $n$ is odd, the incidence variety $I_\omega$ is the blow-up of $\p(V)$
along $F_\omega$.
When $n$ is even, the restriction of the incidence variety $I_\omega$ is a
$\p^1$-bundle over $F_\omega\setminus M_{n-4}$, and a $\p^{2k-1}$-bundle over $M_{n-2k}\setminus M_{n-2k-2}$, when $k=2,..., (n-2)/2$.\end{coro}

\subsubsection{Linear sections of the fundamental locus}

This framework can be further used to study linear sections of $F_\omega$ and $X_\omega$.
Indeed, let $M^* \subset V^*$ and $N^* \subset \bw^2 V^*$
be linear subspaces and we write
$X_{\omega,N} = \p(N) \cap X_\omega$ and $I_{\omega,N}=
\p(\UU^*|_{X_{\omega,N}})$. Also we write $I_{M,\omega}$ for the fibre
product of $I_\omega$ and $\p(M)$ over $\p(V)$.
Assume until the end of the section that $I_{\omega,N}$ and
$I_{M,\omega}$ have expected dimension.

\begin{lemma}
 The image in $\p(V)$ of $I_{\omega,N}$ is the
 degeneracy locus of a map $N \otimes \OO_{\p(V)} \to \C_\omega$. 
 Likewise, the image in $\GR$ of $I_{M,\omega}$ is the degeneracy
 locus of a map $M \otimes \OO_{X_\omega} \to \UU^*\mid_{X_\omega}$. 
\end{lemma}

\begin{proof}
We treat only $N^* \subset \bw^2 V^*$, the argument for $M^* \subset
V^*$ being analogous.

By Proposition \ref{blup}, we have $\p(\C_\omega) \simeq I_\omega$,
and we know by \eqref{OO(1)} that $\OO_{I_\omega}(h)$ is the pull-back
of $\OO_{\GR}(1)$.
So, in $I_\omega$, cutting
with $\p(N) \hookrightarrow \p(\bw^2 V)$ corresponds to the
vanishing of a morphism $N\otimes \OO_{I_\omega} \to
\OO_{I_\omega}(h)$. 
This vanishing locus is isomorphic to the projectivised
cokernel of the direct image in $\p(V)$ of this map, which again
by \eqref{OO(1)} is a morphism $N\otimes \OO_{\p(V)} \to
\C_\omega$.
So the image in $\p(V)$ of  $I_{\omega,N}$
is the degeneracy locus of this last morphism. Actually for even $n$
this map factors through $N\otimes \OO_{F_\omega} \to \C_\omega$ so
the image of $I_{\omega,N}$ is a degeneracy locus inside $F_\omega$.
\end{proof}

\begin{exem} \label{G2}
  For $n=4$ and generic $\omega$, $F_\omega$ is $\p^3 \subset \p^4$
  and $\C_\omega$ is the null correlation bundle on $\p^3$. On the
  other hand, $X_\omega$ is a smooth quadric threefold. The
  variety $I_\omega$ is the projectivisation of the null
  correlation bundle, which in turn is isomorphic to the projectivised
    spinor bundle over the quadric threefold. In other words, $I_\omega$
  is the complete flag for the Lie group $\Spin(5)$ or
  equivalently of $\Sp(4)$.

  For $n=6$ and generic $\omega$, the variety $I_\omega$ is the complete flag for the
  exceptional group $\G2_2$, $F_\omega$ is a smooth
  $5$-dimensional quadric and $X_\omega$ is the $5$-dimensional homogeneous space
  $\G2_2/P(\alpha_2)$ of Picard number $1$. The sheaf
  $\C_\omega$ is the rank-2 stable $\G2_2$-homogeneous bundle 
  on the quadric $F_\omega$, also called Cayley bundle (cf. \cite{Ott}
  for a description in terms of spinor bundles).

  Taking a general linear section of codimension $2$ of $X_\omega$,
  i.e. a general subspace $\CC^2 = N^* \subset \bw^2 V^*$, one
  gets a smooth Fano threefold $X_{\omega,N}$ of genus 10. The image
  in $\p^6=\p(V)$ of the
  manifold $\p(\UU^*|_{X_{\omega,N}})$ is a
  determinantal cubic hypersurface of $F_\omega$ defined by an
  injective map of the form:
  \begin{equation*}
  \OO_{F_\omega}^2 \to \C_\omega. 
  \end{equation*}
  This is a Fano variety of index $2$, singular along a curve of
  degree $18$ and arithmetic genus $10$, given by the locus where the
  map displayed above vanishes. This curve is the image in $\p^6$
  of the Hilbert scheme of lines contained in the Fano threefold
  $X_{\omega,N}$. If $L$ is any such line, $\UU^*$ splits
  over $L$ as $\OO_{L}\oplus
  \OO_{L}(1)$, so that $\OO_{L}(h)$ contracts $\p(\UU^*|_L)$ to 
  a plane with a marked point,
  which is the point in $\p^6$ that corresponds to the given line $L$.
\end{exem}

\subsubsection{Further degeneracy locus for even $n$}

Let us briefly study the further degeneracy locus of the bundle map $\phi_\omega$ in case $n$ is
even. Set $t=n/2-1$. We know that $F_\omega \subset \p^n$ is a hypersurface of
degree $t$, whose singular locus is, generically,
$F_\omega'=M_{n-4}$, a subvariety of codimension $6$ in $\p^n$. 

Recall that the singular locus of $F_\omega'$ is, generically, 
$M_{n-6}$, a subvariety of codimension $15$ in $\p^n$.
All these varieties are subcanonical.

To write a locally free resolution of $F_\omega'$, assuming that the
codimension is $6$ as expected, we use the
sheafified J\'ozefiak-Pragacz complex, cf. for instance \cite[\S
6.4.6]{--98W}.
We denote by $\Gamma^{a,b}$ the Schur functor associated with the
Young tableau having two columns of sizes $a$ and $b$.
This gives:
\begin{equation}
\begin{split}
  0 &\to \OO_{\p^n}(-3t) \to \wedge^2 T_{\p^n}(-2t-3) \to
  \Gamma^{2t+1,1}T_{\p^n}(-4t-3) \to S^2 T_{\p^n}(-t-3) \oplus\\
  &\oplus S^2 \Omega_{\p^n}(-2t+3) \to \Gamma^{2t+1,1}\Omega_{\p^n}(t+3)
  \to \wedge^2 \Omega_{\p^n}(-t+3) \to \II_{F_\omega'/\p^n} \to 0.
\end{split}
\end{equation}

This resolution is self-dual up to sign and up to twisting by $\OO_{\p^n}(-3t)$. In particular, dualising the
resolution and using $\omega_{F_\omega'} \simeq
\EExt^5_{\p^n}(\II_{F_\omega'/\p^n},\OO_{\p^n}(-n-1))$ we get:

\begin{coro} If $F_\omega'$ has expected codimension $6$ in
  $\p^n$, then its canonical bundle is
\begin{equation*}
\omega_{F_\omega'} \simeq \OO_{F_\omega'}(t-3).
\end{equation*}
\end{coro}

For $n \le 6$, $F_\omega'$ is just empty. 
For $n=8$ i.e. $t=3$, $F_\omega'$ is an Abelian surface, actually the Jacobian of the
genus-2 curve associated with $\omega$, embedded by the triple Riemann
Theta divisor, appearing as singular locus of
the relevant Coble cubic, cf. Example \ref{eg:n=8}.

For $n=10$, i.e. $t=4$ and generic $\omega$, $F_\omega'$ is a smooth canonical
$4$-fold of degree 99. For $t=5, 6, 7, 8, 9$ the degree of $F_\omega'$ is
$364, 1064, 2652, 5871, 11858$.


\section{Hilbert scheme}\label{hilbert}


Define the open dense subset $\K$ of $\bw^3 V^*$ by the
condition that $\omega$ belongs to $\K$ if and only if $X_\omega$ has
dimension $n-1$.
For $\omega \in \K$, we  let $P(t)$ be the Hilbert polynomial of
$X_\omega$. Then we
define $\HH$ to be the union of the components of the Hilbert scheme
$\Hilb_{P(t)}(\p(\bw^2 V))$ that contain at least one point of
the form $[X_\omega]$, with $\omega \in \K$.
Sending the proportionality class $[\omega]$ of $\omega$ to 
$[X_\omega]$ we get a morphism:
\begin{equation*}
\rho\colon \p(\K) \to \HH.
\end{equation*}
Our goal here is to prove the following.

\begin{theorem}\label{Hilbsch}
Let $\dim V=n+1\ge 6$ and assume that $\omega\in \bw^3V^*$ satisfies \eqref{GC4}. Then 
$\HH$ is
  irreducible and smooth at any point
  $[X_\omega]$ corresponding to $\omega \in \K$. Moreover:
  \begin{enumerate}[(i)]
  \item for $n \ge 6$, $\rho$ embeds $\p(\K)$ as an open dense subset
    of $\HH$, so
    $\dim(\HH) = \binom{n+1}{3}-1$;
  \item for $n = 5$, $\rho$ is dominant with rational curves as fibres, so
    $\dim(\HH) = \binom{n+1}{3}-2$.
  \end{enumerate}
\end{theorem}

\begin{proof}
  Let $n\ge 5$. The algebraic map $\rho$ is defined precisely for all
  $\omega$  lying in $\K$ and the fibres of $\rho$ consist of the forms
  $\omega'$ such that $X_\omega = X_{\omega'}$. Recall that $X_\omega$
  is determined by its linear span $\Lambda_\omega$ so that
  $[\omega']$ lies in $\rho^{-1}([X_\omega])$ if and only if $\Lambda_\omega=\Lambda_{\omega'}$.
  On the other hand $\Lambda_\omega$ is formed by the 2-vectors $L \in
  \p(\bw^2 V)$ such that $\omega(L)=0$.
  Therefore $\rho^{-1}([X_\omega])$ consists of the forms $[\omega']$ such that
  $\omega'(L)=0$ for all $L \in \Lambda_\omega$. Since this is a
  linear condition, we see that the fibre $\rho^{-1}([X_\omega])$ is a linear
  section of $\p(\K)$.

  Let now $\omega \in \p(\K)$ and recall that $X_\omega$ is obtained
  as zero-locus of the global section $\varphi_\omega$ of $\QQ^*(1)$ (see equation \eqref{eq:asterisco}).
  Then the normal bundle $\N$ of $X_\omega$ in $\GR$ is $\QQ^*(1)|_{X_\omega}$.
  Taking the tensor product of $\QQ^*(1)$ with the Koszul complex
  \eqref{eq:resolution} we obtain then a resolution over 
  $\GR$:
  \begin{equation*}
  \dotsb \to \bw^p \QQ \otimes \QQ^*(1-p) \to \dotsb \to \QQ
  \otimes \QQ^* \to \QQ^*(1) \to \N \to 0,
  \end{equation*}
  where $p$ ranges from $2$ to $n-1$.

  It is well-known that $H^0(\QQ \otimes \QQ^*) \simeq \CC$
  and $H^0(\QQ^*(1)) \simeq \bw^3 V^*$(see Section \ref{congr_deg}). Also
  it is clear that:
  \begin{align*}
  H^k(\QQ \otimes \QQ^*) &= H^k(\QQ^*(1)) =0, &
  \forall k > 0. 
  \end{align*}

  Taking global sections of the rightmost part of the previous display, since the map $\QQ
  \otimes \QQ^* \to \QQ^*(1)$ is just $\varphi_\omega$, we get a
  linear map:
  \begin{equation*}
  \rho_\omega \colon \bw^3 V^*/\gen{ \omega} \to H^0(\N),
   \end{equation*}
  which is nothing but the differential of $\rho$ at $[\omega]$.

  Since we already mentioned that the fibres of $\rho$ are linear
  spaces, we have to check that $\rho_\omega$ is an
  isomorphism for $n \ge 6$, or a surjection with $1$-dimensional
  kernel for $n=5$. We will also check that, in both cases,
  $H^k(\N)=0$ for $k>0$. This will show that $\HH$ is
  irreducible (as image of $\p(\K)$) and smooth
  over the image of $\rho$. This will imply that, for $n \ge 6$, $\rho^{-1}[X_\omega]$ is a
  single point, whereby proving that $\rho$ is a birational morphism,
  actually an isomorphism over the image of $\rho$. On the other hand, if 
  $n =5$, we will deduce that $\rho$ is a surjective morphism whose
  fibres are of the form $\p^1 \cap \p(\K)$, which is a rational
  (perhaps non compact) curve. Both claims on
  $\dim(\HH)$ clearly follow.

  In turn, in order to prove the required properties of $\rho_\omega$,
  for $n \ge 6$ it suffices to check the following vanishing:
  \begin{align}
    \label{eq:bottvanish}
  H^k(\bw^p \QQ \otimes \QQ^*(1-p))&=0, &
    p=&2,\dotsc,n-2,& \forall k.
  \end{align}
  On the other hand, for $n=5$ we have to verify the same vanishing as
  \eqref{eq:bottvanish} except for $p=3$ and $k=2$, where we have to check:
   \begin{equation*}
  H^2(\bw^3 \QQ \otimes \QQ^*(-2)) \simeq \CC.
   \end{equation*}

  To prove these facts, we rely on Borel-Bott-Weil's theorem. Indeed, $\bw^p \QQ \otimes \QQ^*(1-p)$
  is an extension of two vector bundles associated with irreducible
  representations of $\GL(V)$. More precisely, write $\E_\lambda$ for the
  homogeneous bundle associated with the weight $\lambda$ of $\GL(V)$
  and $\lambda_1,\dotsc,\lambda_n$ for the fundamental weights. Then, 
  the homogeneous bundles given by the irreducible representations appearing in the
  filtration of $\bw^p \QQ \otimes \QQ^*(1-p)$ are
  $\E_{\lambda_{n+1-p}+\lambda_3-p\lambda_2}$ and
  $\E_{\lambda_{n+2-p}+(1-p)\lambda_2}$.
  Now, by Borel-Bott-Weil's theorem, for $n \ge 6$ we have:
   \begin{equation*}
  H^k(\E_{\lambda_{n+1-p}+\lambda_3-p\lambda_2})=
  H^k(\E_{\lambda_{n+2-p}+(1-p)\lambda_2})=0,
   \end{equation*}
  for $p=2,\dotsc,n-2$, and $\forall k$. This implies
  \eqref{eq:bottvanish} for $n \ge 6$.
  On the other hand, for $n=5$, again Bott's theorem implies the same
  vanishing except for $p=3$ and $k=2$, where we get:
  \begin{align*}
  H^2(\E_{2\lambda_3-3\lambda_2})&\simeq \CC, &
  H^2(\E_{\lambda_{4}-2\lambda_2})&=0.
  \end{align*}
  This implies
  \eqref{eq:bottvanish} for $n = 5$ and thus concludes the proof.
\end{proof}

\section{Quadrics containing 
$\gen{X_\omega}$ and congruences in linear subspaces}\label{sect:quadrics} 


As usual we denote by $\Lambda_\omega$ the linear span of $X_{\omega}$. 
We are interested in characterising the quadrics in the ideal of $\GR$ that contain also  $\Lambda_{\omega}$.

First we compute the dimension of this space of quadrics.
\begin{lemma}\label{n+1} Let $\omega\in\bw^3V^*$ be a $3$-form such that $X_\omega$ has dimension $n-1$, i.e satisfying condition \ref{GC4}. Then 
$h^0(\II_{\Lambda_\omega\cup \GR}(2))=n+1.$ 
\end{lemma}
\begin{proof}
The homogeneous ideal of $\Lambda_\omega\cup \GR$ is the intersection of the homogeneous ideals of 
$\Lambda_\omega$ and $\GR$, so $H^0(\II_{\Lambda_\omega\cup \GR}(2))$ is simply the homogeneous piece of degree $2$ of this intersection. 

Recall that $X_\omega$ is the intersection of $\GR$
and $\Lambda_\omega$, so that we have an equality of homogeneous ideals, $I(X_\omega) = I(\GR) +
I(\Lambda_\omega)$. Consider therefore the following exact sequence of ideals:
\begin{equation*}
0 \to I(\Lambda_\omega) \cap I(\GR) \to I(\Lambda_\omega) \oplus I(\GR) \to I(X_\omega) \to 0.
\end{equation*}
Looking at the dimension of the homogeneous pieces of degree $2$ we get:
\begin{equation*}
\dim I(\Lambda_\omega)_2\cap I(\GR)_2 = \dim I(\GR)_2 + \dim
I(\Lambda_\omega)_2 - \dim I(X_\omega)_2.
\end{equation*}
Also, we have an exact sequence:
\begin{equation*}
0 \to \II_{\GR\mid\p(\bw^2 V)} \to \II_{X_\omega \mid \p(\bw^2  V)}\to \II_{X_\omega \mid \GR} \to 0.
\end{equation*}
Twisting by $\OO_{\p (\bw^2 V)}(2)$ and taking global sections, we
easily get:
\begin{equation*}
\dim I(X_\omega)_2 = \dim I(\GR)_2 + h^0(
\II_{X_\omega \mid \GR}(2)),
\end{equation*}
and therefore:
\begin{equation*}
\dim I(\Lambda_\omega)_2\cap I(\GR)_2 =  \dim I(\Lambda_\omega)_2 - h^0(
\II_{X_\omega \mid \GR}(2)).
\end{equation*}
To finish we have to compute the right-hand-side. On one hand, since
$\gen{X_\omega}=\Lambda_\omega$ has codimension $n+1$ in $\p(\bw^2 V)=\p^{\binom{n+1}{2}-1}$, we
easily get $\dim I(\Lambda_\omega)_2=n\binom{n+1}{2}$.
On the other hand, recall that $X_\omega$ is the vanishing locus of the
global section $\varphi_\omega$ of $\QQ^*(1)$, so that, when the dimension of $X_\omega$ is $n-1$, the Koszul
complex \eqref{eq:resolution}
is a resolution of the ideal of $X_\omega$ on $\GR$.
Now we twist this sequence with $\OO_{\GR}(2)$ and take global sections. 
Applying Borel-Bott-Weil's theorem, we get:
\begin{align*}
h^0(
\II_{X_\omega \mid \GR}(2)) & =h^0(
\QQ(1))-h^0(
\bw^2 \QQ) 
=n\binom{n+1}{2}-n-1.  
\end{align*}
Putting together, we obtain the  equality $\dim I(\Lambda_\omega)_2\cap I(\GR)_2 =n+1$.
\end{proof}

We shall find a natural isomorphism 
\begin{equation*}
V^*\to H^0(\p(\bw^2V),\II_{\Lambda_\omega \cup \GR}(2))
\end{equation*}
 parametrising this subspace in the space of quadrics defining $\GR$.

 \begin{nota}
\label{quadric}   
With notation as in Corollary \ref{cor:iso}, we  denote by $\QQ_\omega$ the 
image of the map $V^* \to I(\GR)$ sending $x$ to $q_{\omega\w x}$. These quadratic forms correspond to the quadrics $Q_{\omega\w x}$ introduced and studied in Section \ref{the quadric}.
 \end{nota}

\begin{prop}\label{Qn+1} Assume that $n\geq 5$ 
and that $X_\omega$ has dimension $n-1$ (condition \ref{GC4}).
Then
\begin{equation*}
\QQ_\omega=H^0(\p(\bw^2V),\II_{\Lambda_\omega\cup\GR}(2)).
\end{equation*}
Furthermore, for any quadric $Q_{\omega \w x}\in \QQ_\omega$ of rank $2n$, the linear span $\Lambda_\omega=\gen{ X_{\omega}}$ has codimension one in the maximal isotropic subspace $\Lambda_\omega^x$ in  $Q_{\omega\w x}$. 
\end{prop}
\begin{proof}
The quadrics in $\QQ_\omega$ contain $\Lambda_\omega$ in view of Corollary \ref{span_quadrics}.  Since $X_\omega$ has dimension $n-1$, 
Remarks \ref{implicationsGC} and 2.13 imply that $\omega$ is indecomposable. Therefore $V^*\to \QQ_\omega$ is an isomorphism, so the first part of the proposition follows. The second part follows from Corollary 2.19.
\end{proof}

We shall denote by  $X_{\omega_{x}}$ the congruence in $G(2,V_x)$ defined by the restriction $\omega_x$ of $\omega$ to $V_x$. We recall also from \ref{span3}, the subspace $\Lambda_{\omega_x}$. 

\begin{coro}\label{n:4} Assume that $n\geq 5$  
and that $X_\omega$ has dimension $n-1$ (condition \ref{GC4}).
If $Q_{\omega\w x}\in \QQ_\omega$ has rank $2n$, 
then
\begin{equation*}
\Sing(Q_{\omega\w x})=\Lambda_{\omega_x}=\gen{ X_{\omega_{x}}}.
\end{equation*}
\end{coro}
\begin{proof} It follows immediately from Proposition \ref{Qn+1} and Theorem \ref{rk2n-quadrics}, \eqref{rk2n:1}, because $\gen{ X_{\omega_x}}=\{[L]\mid x(L)=\omega_x(L)=0\}=\{[L]\mid x(L)=\omega(L)\w x=0\}.$
\end{proof}

\begin{rema} If $n=3$, the only quadric $Q_{\omega\w x}$ is the Grassmannian $G(2,V)$, which is smooth, and also $X_{\omega_x}$ is empty. If $n=4$, the rank of $Q_{\omega\w x}$ is $6$ for any choice of $x$, so $\Sing(Q_{\omega\w x})$ is a $\p^3$. On the other hand $X_{\omega_x}=\Lambda_{\omega_x}$ is a 
$\p^2$. So the conclusion of Corollary \ref{n:4} is not true for $n=4$.
\end{rema}

\begin{coro}\label{4-form}   Assume that $n\geq 5$  
and that $X_\omega$ has dimension $n-1$ (condition \ref{GC4}).
If $\omega_x$ has rank at least $\frac{n+2}{2}$ for every $x\in V^*$,
then there are no $4$-forms vanishing on the linear span of $X_\omega$, i.e. on  $\Lambda_\omega$. 
 In particular, there are no quadrics in the ideal of $\GR$ that are singular along $X_\omega$.
\end{coro}
\begin{proof} Any $4$-form $\eta$ defines a quadric $Q_\eta$ in the ideal of $\GR$.  If $\eta$ vanishes on 
$\Lambda_\omega$, then $X_\omega\subset \Sing( Q_{\eta})$.  Now, by  Proposition \ref{Qn+1},  the linear space $\Lambda_\omega$ is in $Q_{\eta}$ only if   the $4$-form $\eta$ is of the form $\omega\w x$  for some $x\in V^*$.  Furthermore, by Lemma \ref{eta}, the quadric  $Q_{\omega\w x}$ has rank equal to $2\rk \omega_x$.  The space $\Lambda_\omega$ has codimension $n+1$, so it is contained in $\Sing( Q_{\omega\w x})$ only if this rank is at most $n+1$, i.e. 
$2\rk \omega_x\leq n+1$.
\end{proof}

Note that the assumption on $\omega_x$ in Corollary \ref{4-form} is satisfied for general $\omega$. Indeed, let $n=2m-1$, and let $V=V_0\oplus V_1$ be a decomposition in $m$-dimensional subspaces. Let $x\in V_1^*$ and $\omega=\omega_0+\beta\w x$, where $\omega_0\in \bw^3V_0^*$ and $\beta\in \bw^2V_1^*$ are generic forms.   Then $\omega_x=\omega_0$ and has rank $m=\frac{n+1}{2}$.  The variety $X_\omega$ is 
 contained in the singular locus of the  quadric $Q_{\omega_0\w x}$ of rank $2m=n+1$.

Similarly, when $n=2m-2$, let $V=V_0\oplus V_1$ be a decomposition in a $(m-1)$-dimensional subspace and a $m$-dimensional subspace. Let $x\in V_1^*$ and $\omega=\omega_0+\beta\w x$, where $\omega_0\in \bw^3V_0^*$ and $\beta\in \bw^2V_1^*$  are generic forms.   Then $\omega_x=\omega_0$ and has rank $m-1=\frac{n}{2}$.  The variety $X_\omega$ is 
contained in the singular locus of the  quadric $Q_{\omega_0\w x}$ of rank $2m-2=n$.

The next Theorem gives a more precise formulation of Corollary \ref{n:4}.

\begin{theorem}\label{2n-quadric} Let $n\geq 5$ and assume that $\omega\in \bw^3V^*$ is a $3$-form such that $X_\omega$ has dimension $n-1$ 
(condition \ref{GC4}). Let 
$x\in V^*$ be a linear form such that   $\omega_{x}$ 
has rank $n$ and $\omega_{x\w y}$, the restriction of $\omega$ to $V_{x\w y}$, has rank $\geq \frac{n+1}{2}$ for every $y\in V_x^*$. 
 Then 
\begin{equation*}
\Sing(Q_{\omega\w x})\cap \GR= X_{\omega_{x}}. 
\end{equation*} 
Conversely, if $n\geq 7$,   then $Q_{\omega\w x}$ is the unique quadric that contains $\GR$ 
 and is singular along $X_{\omega_{x}}$. Furthermore, the quadric $Q_{\omega\w x}$ contains the linear span of $X_{\omega'}$ for any $\omega'$ such that $\omega'\w x=\omega_{x}\w x$, i.e. $\omega'_x=\omega_x$. 
\end{theorem}
\begin{proof} The quadric $Q_{\omega\w x}$ has rank $2n$ because $\omega_x$ has rank $n$ by assumption  (Lemma \ref{eta}).  The first part follows immediately from Corollary  \ref{n:4}.

For the converse, a general $3$-form $\omega_{x}$ on $V_{x}$ extends to $3$-forms  on $V$, which again define a quadric singular on $X_{\omega_{x}}$.  It suffices therefore identify the space of quadrics  in the ideal of $\GR$ singular along $X_{\omega_{x}}$.

Again we use the correspondence between $4$-forms and quadrics in the ideal of $\GR$ (see Corollary \ref{cor:iso}).
Let $Q=Q_\eta$ be a quadric singular along $X_{\omega_{x}}\subset \p(\bw^2V_{x})\subset\p(\bw^2V),$ where $\eta $ is a $4$-form as above.  Write
$\eta=\eta_x+\gamma_x\w x$ for some $4$-form $\eta_x$ and $3$-form $\gamma_x$ on $V_{x}$.  
Then $Q_\eta$ is singular along $\gen{ X_{\omega_{x}}}$ only if $\eta(L)=0$ for any $[L]\in \gen{ X_{\omega_{x}}}$.  But $x(L)=0$ for any $L\in \bw^2V_x$, so in this case 
$\eta (L)=\eta_x(L)+\gamma_x(L)\w x=0$ only if $\eta_x(L)=\gamma_x(L)=0$.  
But if $\eta_x(L)=0$ for any $[L]\in \gen{ X_{\omega_{x}}},$ then $\eta_x=0$ by Corollary \ref{4-form}.
Hence $\eta=\gamma_x\w x$.  Furthermore, by Theorem \ref{Hilbsch} 
the spaces $\{[L]\mid \gamma_x(L)\w x=0\}$ and $\gen{ X_{\omega_x}}$ are equal if and only if $\gamma_x$ is proportional to $\omega_{x}$. 
Therefore $Q=Q_\eta=Q_{\omega\w x}=Q_{\gamma_x\w x}$ for any $\gamma_x$ whose restriction to $V_{x}$ is proportional to $\omega_{x}$.
In particular, 
$Q$ depends only on $X_{\omega_{x}}$, and contains $X_\omega$ in a maximal dimensional linear subspace.
\end{proof}

We recall  that Theorem \ref{rk2n-quadrics} describes the 
two families of maximal dimensional subspaces in the quadric $Q_{\omega\w x}$ of rank $2n$.  One family is in bijection with all $3$-forms 
whose restriction to $V_{x}$ coincides with $\omega_{x}$. The congruence $X_\omega$ is contained in a unique maximal dimensional linear subspace $\Lambda_\omega^x$ of this family in the quadric $Q_{\omega\w x}$.   
We shall study now  the union and the intersection of $X_{\omega}$ and  $X_{\omega_x}$.

\begin{theorem}\label{union}
Let $\omega\in \bw^3 V^*$ be a $3$-form satisfying 
condition \ref{GC2}. Let $x\in V^*$ be general, let  $V_{x}$ be the hyperplane $\{x=0\}$. 
Let $e\in V$ be a fixed vector such that $x(e)=1$.  
Then let $\omega=\omega_x+\beta_x\w x$ be the unique decomposition as in  \eqref{decomp}. 
Moreover,  for any $L\in \bw^2 V$ let $L=L_x+e\w v_x$ be the unique expression with $L_x\in \bw^2 V_x$  and $v_x\in V_x$. 
Finally let $Q=Q_{\omega\w x}$.  
Then
\begin{enumerate}
\item\label{u:1} $\Lambda_{\omega}^x\cap \GR=X_{\omega}\cup X_{\omega_{x}}$. 
\item\label{u:2} If $\Lambda_{\omega'}^x$ is a maximal isotropic space in $Q$ of the same family as $\Lambda_{\omega}^x$, for some   $3$-form $\omega'$ such that  $\omega'_x=\omega_x$,  then  $\Lambda_{\omega'}^x\cap \GR=X_{\omega'}\cup X_{\omega_{x}}$. 
\item\label{u:3} $X_\omega\cap X_{\omega_x}=X_\omega\cap G(2,V_x)$ is the hyperplane section of $X_{\omega_x}$ of equation $\beta_x(L_x)=0$. 
\end{enumerate}
\end{theorem}

\begin{proof}
The inclusion $\supset$ in \eqref{u:1} is clear, because $X_\omega\subset\Lambda_\omega\subset \Lambda_{\omega}^x$, 
and $X_{\omega_x}\subset \Lambda_{\omega}^x$  is contained in the singular locus of $Q$. 
Conversely, 
we have to prove that, if $L\in \bw^2 V$ satisfies the conditions $\omega(L)\w x=L\w L=0$, 
then either $\omega(L)=0$ or $x(L)=\omega_x(L)=0$.  

Note that  $\omega(L)\w x=(-\beta_{x}(v_x)+\omega_{x}(L_{x}))\w x=0$  implies $\omega_{x}(L_{x})=\beta_{x}(v_x)$. Moreover $x(L)=0$ if and only if $v_x=0$, so in this case $[L]\in G(2,V_{x})$ and $\omega_{x}(L_{x})=\omega_{x}(L)=0$, and we conclude that $L\in X_{\omega_{x}}$.

If instead $v_x\neq 0$, we fix a basis of $V$ with $e_0=e$, $e_1=v_x$, and $x=x_0$, and we  write $\omega$ and $L$ as follows: 
$\omega=\omega_{01}+x_0\w \beta_0+x_1\w \beta_1+x_0\w x_1\w \alpha$, $L=L_{01}+e_0\w e_1+e_1\w v_1$, with 
$\omega_{01}\in  \bw^3 \langle x_2,\dotsc,x_n\rangle$, $\beta_0, \beta_1\in  \bw^2  \langle x_2,\dotsc,x_n\rangle$, $\alpha\in  \langle x_2,\dotsc,x_n\rangle$, $L_{01}\in  \bw^2\langle e_2,\dotsc,e_n\rangle$, $v_1\in \langle e_2,\dotsc,e_n\rangle$.
Now one computes easily that $\omega(L)\w x_0=0$ is equivalent to  $\omega_{01}(L_{01})-\beta_1(v_1)-\alpha+\beta_1(L_{01})x_1=0$, and $\omega(L)=0$ is equivalent to  $\omega(L)\w x_0+x_0\beta_0(L_{01})=0$. But $L\w L=0$ is equivalent to $L_{01}\w L_{01}-2e_1\w L_{01}\w v_1+2e_0\w e_1 \w L_{01}=0$, which implies $L_{01}=0$. This concludes the proof of  \eqref{u:1}.  

The proof of \eqref{u:2} is similar. 
To prove \eqref{u:3}, note that if $L\in \bw^2V_{x}$ then $v_x=0$, so the condition $\omega(L)\w x=0$ is equivalent to $\omega_x(L_x)=0$, whereas the condition $\omega(L)=0$ is equivalent  to $\omega_x(L_x)+\beta_x(L_x)x=0$. Therefore $X_\omega\cap X_{\omega_{x}}$ is equal to the intersection of $X_{\omega_{x}}$ with the hyperplane of equation $\beta_x(L_x)=0$. 

\end{proof}

\section{The residual congruences and their fundamental loci}\label{residual}


As noted in Theorem \ref{thm:1}, the variety $X_\omega$ is contained in a reducible linear congruence, i.e. in a proper section of the Grassmannian by a linear space  of codimension $n-1$.  
\begin{defi}\label{def:res}
Let $\Gamma$ be a linear subspace of $\p(V)$ of codimension $n-1$, containing $X_\omega$ and such that the intersection $Z=\GR\cap \Gamma$ is proper. Then the congruence $Y\subset \GR$ whose ideal $I_Y$ satisfies $[I_Z\colon I_{X_{\omega}}]\simeq I_Y$ and  $[I_Z\colon I_Y]\simeq I_{X_{\omega}}$, is called the\emph{ residual congruence  to $X_\omega$ in $\GR\cap \Gamma$}.  
We also say that $X_\omega$ and $Y$ are \emph{linked in $\GR\cap \Gamma$}.
\end{defi}

In this case, set-theoretically, we have $\GR\cap\Gamma=X_\omega\cup Y$.  
In Proposition \ref{Yred} we shall give a more precise description of $Y$ in terms of $\omega$ and the choice of linear space $\Gamma$.   

We first describe some general facts about the residual congruence.
\begin{prop}\label{Y is CM} Let $Y$ be the residual congruence to $X_\omega$ in $\GR\cap \Gamma$.
\begin{enumerate} 
\item\label{y:1} For general $\Gamma$, $Y$ is an irreducible congruence; 
\item\label{y:2} $Y$ has order $1$ if $n$ is even and order $0$ if $n$ is odd; 
\item\label{y:3} $Y$ is locally Cohen--Macaulay.
\end{enumerate}
\end{prop}
\begin{proof} \eqref{y:1} follows from Bertini theorem. For \eqref{y:2} it is enough to recall Proposition \ref{prop:ordine} and to note that the linear congruence $X_\omega\cup Y$ has order $1$.
To prove \eqref{y:3} we recall that the Grassmannian $\GR\subset \p(\bw^2V)$ 
is arithmetically Gorenstein (aG for short in what follows): in fact, it is subcanonical (the canonical sheaf is $\omega_{\GR}\cong \OO_{\GR}(-n-1)$)
and aCM
(by Bott's theorem, cf. \cite[Theorem 4.1.8]{--98W}). 
Therefore, the linear congruence $X_\omega\cup Y$
  is---by adjunction---arithmetically Gorenstein as well (see for example \cite[Theorem 1.3.3]{M}).
But then  $ X_\omega$ and $Y$  
are geometrically (directly) G-linked (see \cite[Definition 5.1.1]{M}), and the results of \cite[Chapter 5]{M} can be 
applied. In particular, since $X_\omega$ is smooth, $Y$ is locally Cohen-Macaulay (\cite[Corollary 5.2.13]{M}).
\end{proof}

  The linkage will allow us to resolve the ideal of $Y$ in $\GR$.

\subsection{Resolving the ideal of the residual congruence $Y$}

Let us recall the Koszul resolution  \eqref{eq:resolution} of
$X_\omega$ in $\GR$ and write down the Koszul resolution of
$X_\omega \cup Y$ in $\GR$:
\begin{equation*}
0 \to \OO_{\GR}(1-n) \to \dotsb \to \OO_{\GR}(-p)^{\binom{n-1}{p}} \to \dotsb\to \OO_{\GR} \to \OO_{X_\omega \cup Y} \to 0.
\end{equation*}
The surjection $\OO_{X_\omega \cup Y} \to \OO_{X_\omega}$ lifts to a morphism
of complexes that reads:
\begin{equation}\label{Schubert hyperplane}
\xymatrix@C-3ex@R-2ex{
&&&&  & \II_{X_\omega / X_\omega \cup Y} \ar[d] \\
\OO_{\GR}(1-n) \ar[d] \ar[r] & \dotsb \ar[r] &
\OO_{\GR}(-p)^{\binom{n-1}{p}} \ar[d] \ar[r] & \dotsb \ar[r]
\ar[r] & \OO_{\GR} \ar[r] \ar@{=}[d]& \OO_{X_\omega \cup Y} \ar[d] \\
\OO_{\GR}(2-n) \ar[r] & \dotsb \ar[r] & \bw^p\QQ(-p) \ar[r] \ar[r] & \dotsb \ar[r]& \OO_{\GR} \ar[r] & \OO_{X_\omega} 
}\end{equation}
This yields a resolution of $\II_{X_\omega / X_\omega \cup Y}$ of the form:
\begin{equation*}
0\to \OO_\GG(1-n)\to\dotsb \to \OO_{\GR}(-p)^{\binom{n-1}{p}} \oplus \bw^{p+1} \QQ(-p-1) \to \dotsb \to \QQ(-1) \to
\II_{X_\omega / X_\omega \cup Y} \to 0,
\end{equation*} 
where $0\le p \le n-1$. 
By linkage, taking duals and taking tensor product with
$\OO_{\GR}(1-n)$, we get a resolution of $\OO_{Y}$ and
hence of $\II_{Y/\GR}$. Using
the duality $\bw^p \QQ^* \simeq \bw^{n-1-p}\QQ(-1)$, this
takes the form:
\begin{equation}\label{linked}
0\to \QQ^*(2-n)\to\dotsb \to \OO_{\GR}(-p)^{\binom{n-1}{p}} \oplus \bw^{p-1} \QQ(-p) \to
\dotsb \rightarrow \OO_{\GR}(-1)^n \rightarrow 
\II_{Y/\GR} \to 0,
\end{equation}
with $1\le p \le n-1$.

\begin{rema}\label{Schubhyp}
The map $\OO_{\GR}(1-n) \to \OO_{\GR}(2-n)$ in diagram \eqref{Schubert hyperplane} defines a Schubert hyperplane containing $Y$.  In fact, by Lemma \ref{supsets},
if $\Gamma\supset \Lambda_\omega$ is 
any codimension $n-1$ linear space of $\p(\bw^2V)$, then $\Gamma=\Lambda_\omega^{xy}$ (cf. the definition of $\Lambda_\omega^{xy}$ 
in \eqref{span2}) for a $2$-form $x\w y\in \bw^2V^*$. 
 In particular, if  $Y$ is residual to $X_\omega$ in a linear congruence, then 
\begin{equation*}
X_\omega\cup Y=\GR\cap\Lambda_\omega^{xy}
\end{equation*}
for some $x\w y\in \bw^2V^*$. The Schubert hyperplane defined by $x\w y$ intersects $\Lambda_\omega^{xy}$ in $\Lambda_{\omega,x\w y}$.
We shall show in Proposition \ref{spanY} that $Y=\Lambda_{\omega,x\w y}\cap \GR.$
\end{rema}
\begin{nota}\label{resid}
We denote by $Y_{\omega,x\w y}$ the residual congruence to $X_\omega$ in $\GR\cap \Lambda_\omega^{xy}$, whenever the latter intersection is proper. 
It will be denoted simply by $Y_{x\w y}$ when $\omega$ is understood, or $Y$ if also $x,y$ are understood. 
\end{nota}

The residual congruences to $X_\omega$ appear naturally from quadrics of $\QQ_\omega$. 
Let $n\geq 5$ and assume that $\omega$  satisfies \ref{GC4}. Let  us consider the quadric $Q=Q_{\omega\w x}$  for some $x\in V^*$.
It contains $\GR$ and its rank is $2n$, so $Q$ has two $\binom{n}{2}$-dimensional families 
of maximal dimensional subspaces, i.e. subspaces of codimension $n$ in $\p({\bw^2V})$.  These two families were described in   Theorem \ref{rk2n-quadrics}.
The linear space $\Lambda_\omega^x$ belongs to one of these families, and in view of Proposition \ref{Qn+1}, 
the linear span $\gen{X_\omega}=\Lambda_\omega$ is a hyperplane in $\Lambda_\omega^x$. 
A subspace of the other family that intersects $\Lambda_\omega^x$ in codimension one is of the form 
$\Lambda_{\omega,x\w y}$ for some $y\in V^*$. 
The union $\Lambda_\omega^x\cup \Lambda_{\omega,x\w y}$ spans the codimension $n-1$ linear subspace $\Lambda_\omega^{xy}$. 
The restriction of $Q_{\omega\w x}$ to $ \Lambda_\omega^{xy}$ decomposes as
\begin{equation*}
Q_{\omega\w x}\cap \Lambda_\omega^{xy}=\Lambda_\omega^x\cup \Lambda_{\omega,x\w y},
\end{equation*}
and the subspace $\Lambda_{\omega,x\w y}\subset \Lambda_\omega^{xy}$ is the intersection of $\Lambda_\omega^{xy}$ with the Schubert hyperplane $\{x\w y=0\}\subset\p(\bw^2V)$.

Therefore, when $x,y\in V^*$ are general linear forms, the linear congruence $\Lambda_\omega^{xy}\cap   \GR$ decomposes as
\begin{equation*}
\Lambda_\omega^{xy}\cap   \GR=(\Lambda_\omega^x\cap\GR)\cup(\Lambda_{\omega,x\w y}\cap \GR).
\end{equation*}
By Theorem \ref{union} \eqref{u:1}, $\Lambda_\omega^x\cap\GR=X_\omega\cup X_{\omega_x}$ and, by Theorem \ref{2n-quadric}, 
$\Sing(Q_{\omega\w x})\cap\GR=X_{\omega_x}$, so
  $X_\omega\cup Y_{\omega,x\w y}=\Lambda_\omega^{xy}\cap   \GR$ implies the following lemma:  
\begin{prop}\label{spanY} 
The residual congruence $Y_{\omega, x\w y}$ is a non-proper linear congruence.  More precisely
\begin{equation*}
Y_{\omega,x\w y}=\Lambda_{\omega,x\w y}\cap\GR,
\end{equation*}
with $\codim \Lambda_{\omega, x\w y}=n$. 
In particular, the congruence $Y_{\omega,x\w y}$ is contained in the Schubert hyperplane formed by the lines meeting $\{x=y=0\}$.
Furthermore, 
\begin{equation*}
X_\omega\cap Y_{\omega,x\w y}
\end{equation*}
is the intersection of $X_\omega$ with this Schubert hyperplane and 
\begin{equation*}
 X_\omega\cap Y_{\omega,x\w y}\cong H_{Y_{\omega,x\w y}}-X_{\omega_x}
\end{equation*}
as Weil divisors on $Y_{\omega,x\w y}$.
  \end{prop}
  \begin{proof} Recalling Lemma \ref{span_quadrics}, 
only the last part remains to be proven.
 But $\Lambda_\omega^x$ is a hyperplane in $\Lambda_\omega^{xy}$, so $\Lambda^x_\omega\cap Y_{\omega,x\w y}=(X_\omega\cap Y_{\omega,x\w y})\cup X_{\omega_x}$ is a hyperplane section of $Y_{\omega,x\w y}$.
 \end{proof}

\subsection{The multidegree}

The multidegree of a general congruence $Y=Y_{\omega, x\w y}$,  residual  to $X_\omega$, can be computed for every $n$ using 
\eqref{multidegree} and the multidegree of a linear congruence. We recall that in \cite[Corollary 2.3]{DP}
the multidegree $(e_0(n),\dotsc,e_{[\frac{n-1}{2}]}(n))$ of a linear congruence $B\subset \GG$, where
\begin{align*}
e_\ell(n)&:=\int_{[B]}\sigma_{n-1-\ell,\ell} & 0\leq \ell&\leq \left[\frac{n-1}{2}\right],
\end{align*}
is given by the following closed formula
\begin{equation*}
\multdeg(B)=\left(1,n-2,\dotsc,\left(\binom{n-2}{i}-\binom{n-2}{i-2}\right),\dotsc,
\left(\binom{n-2}{\nu}-\binom{n-2}{\nu-2}\right)\right),
\end{equation*}
where  $\nu:=[\frac{n-1}{2}]$. 
Similarly (loc. cit.)
\begin{equation*}
\deg(B)=\sum_{j=0}^{\nu}\left(\binom{n-2}{j}-\binom{n-2}{j-2}\right)^2=
\frac{1}{n-1}\binom{2n-2}{n}. 
\end{equation*}
Anyway, in order to obtain $\multdeg(Y)$, it is useful to  organise these degrees 
as we did for those of $X_\omega$ in Section \ref{coho}. In fact, 
the proof of Lemma \ref{recursion} applied to $e_i(n):=\left(\binom{n-2}{i}-\binom{n-2}{i-2}\right)$ gives 
the following lemma: 
\begin{lemma} The multidegree $(e_i(n))$, $i=0,\dotsc,n-1$, satisfies the initial condition
\begin{align*}
e_0(n)&=1, &  n&=2,3,4,\dotsc
\end{align*}
 and the recursion relation
\begin{equation*}
e_i(n)=e_{i-1}(n-1)+e_{i}(n-1)
\end{equation*}
 when $i=1,2,\dotsc,[\frac{n-1}{2}]$.
\end{lemma}
Then, as we did with the multidegree $(d_i(n))$ of $X_\omega$, we may display the multidegree $(e_i(n))$ of $B$ 
in a triangle with initial entries
\begin{align*}
b_{n,0}&=1,& n&=0,1,2,\dotsc
\end{align*}
and 
\begin{align*}
b_{(i,j)}&=b_{(i,j-1)}+b_{(i-1,j)}& i&=1,2,\dotsc,\textup{ and } j=1,2,\dotsc,i.
\end{align*}
The multidegree of $B$ is identified as $(e_i(n))=(b_{(n-1-i,i)})$, $i=0,\dotsc,n-1$, 
\begin{equation}\label{eq:trien}(b_{i,j})=
 \begin{array}{cccccccccc} 
      1 & & & & & & & & & \\
      1 & 1 & & & & & & & & \\
      1 & 2 & 2& & & & & & & \\
      1 & 3 & 5&5 & & & & & & \\
      1 & 4 & 9&14 & 14& & & & & \\
      1 & 5 & 14&28 & 42& 42 & & & & \\
      1 & 6 & 20& 48& 90& 132& 132& & & \\
      1 & 7 & 27&75 & 165& 297& 429 &429 & & \\
      1 & 8 & 35& 110& 275& 572& 1001& 1430& 1430& \\
      1 & 9 & 44& 154& 429& 1001& 2002& 3432 &4862 &4862 \\
   \end{array}
   \end{equation}
for which we have, with the same proof as in Proposition \ref{prop:tria}: 
\begin{prop}\label{prop:trib}  
The degree of a linear congruence $B\subset G(2,n+1)$ is given 
by the $(n+1)$-st diagonal element 
 \begin{equation*}
\deg B=b_{(n-1,n-1)}=b_{(n-1,n-2)}=e_{n}(2n-1)
\end{equation*}
 in the above triangle of numbers.
  The multidegree of $B$ is given by the antidiagonals 
  \begin{equation*}
\multdeg(B)=(e_{0}(n),e_{1}(n),\dotsc,e_{[\frac{n-1}{2}]}(n))=(b_{(n-1,0)},b_{(n-2,1)},\dotsc,b_{(n-[\frac{n-1}{2}],[\frac{n-1}{2}])}).
    \end{equation*}
  \end{prop}
  
From this, we can infer the multidegree of $Y$; in fact, it is enough to ``subtract'' triangles \eqref{eq:tridn} and \eqref{eq:trien}
obtaining the following triangle 
\begin{equation}\label{eq:tricn} (c_{i,j})=
 \begin{array}{cccccccccc} 
      0 & & & & & & & & & \\
      1 & 1 & & & & & & & & \\
      0 & 1 & 1& & & & & & & \\
      1 & 2 & 3&3 & & & & & & \\
      0 & 2 & 5&8 & 8& & & & & \\ 
      1 & 3 & 8&16 & 24& 24 & & & & \\ 
      0 & 3 & 11& 27& 51& 75& 75& & & \\ 
      1 & 4 & 15& 42 & 93& 168& 243 & 243& & \\ 
      0 & 4 & 19& 61 &154  & 322 &565  & 808& 808& \\ 
      1 & 5 &24 &85 &239 & 561& 1126 & 1934 &2742 &  2742\\ 
   \end{array},
   \end{equation}
and if we call $c_{i,j}$ the entries of the triangle \eqref{eq:tricn}, we have $c_{i,j}=b_{i,j}-a_{i,j}$.
Let 
\begin{align*}
f_\ell(n)&:=\int_{[Y]}\sigma_{n-1-\ell,\ell} & 0\leq \ell&\leq \left[\frac{n-1}{2}\right],
\end{align*}
then $(f_0(n),\dotsc,f_{[\frac{n-1}{2}]}(n))$
is the multidegree of $Y$.   From 
Propositions \ref{prop:tria} and \ref{prop:trib}, we obtain  the following proposition: 

\begin{prop}\label{prop:tric}  
The degree of $Y=Y_{\omega,x\w y}\subset \GG$ is the  $(n+1)$-st diagonal element 
 \begin{equation*}
\deg Y=c_{(n-1,n-1)}=c_{(n-1,n-2)}=f_{n}(2n-1)
\end{equation*}
 in the above triangle of numbers.
  The multidegree of $Y$ is given by the antidiagonals 
  \begin{equation*}
  (f_{0}(2m),f_{1}(2m),\dotsc,f_{m-1}(2m))=(c_{(2m-1,0)},c_{(2m-2,1)},\dotsc,c_{(m,m-1)}),
    \end{equation*}
 when $n=2m$,  and
    \begin{equation*}
  (f_{0}(2m-1),f_{1}(2m-1),\dotsc,f_{m-1}(2m-1))=(c_{(2m-2,0)},c_{(2m-3,1)},\dotsc,c_{(m-1,m-1)}),
    \end{equation*}
    when $n=2m-1$.
  \end{prop}
For $3\leq n \leq 9$, we get the following multidegree for $Y$:
\begin{align*}
(0,1)&, &(1,1)&, &(0,2,1)&, &(1,2,3)&, &(0,3,5,3)&, &(1,3,8,8)&, &(0,4,11,16,8).
\end{align*}

\subsection{Divisors and singularities on the residual congruences to $X_\omega$.}\label{alternative}
Recall that $Y_{\omega,x\w y}=\Lambda_{\omega,x\w y}\cap \GR$ and thus, by Lemma \ref{supsets}, $Y_{\omega,x\w y}$ is contained 
in the pencil of quadrics generated by $Q_{\omega\w x}$ and $Q_{\omega\w y}$.  
Furthermore,  $Y_{\omega,x\w y}=Y_{\omega',x\w y'}$ only if $\gen{ x,y}=\gen{ x,y'}$,  $\omega_x=\omega'_x$ and $\omega_{y}=\omega'_{y}$.

\begin{prop}
There is an $(n-1)$-dimensional family of congruences $X_{\omega'}$ that are linked to $Y_{\omega, x\w y}$ in a linear congruence. 
They are defined by $3$-forms $\omega'=\omega+\alpha'\w x\w y$, for some $\alpha'\in V^*$.
\end{prop}

To further analyse $Y_{\omega, x\w y}$ we consider the pencil of hyperplanes generated by $\p(V_x)$ and $\p(V_y)$:
\begin{align*}
\p(V_{[a:b]})&=\{ax+by=0\},& [a:b]&\in \p^1
\end{align*}
and denote by 
$\omega_{[a:b]}$ the restriction of $\omega$ to $V_{[a:b]}$.

\begin{prop}\label{Yred}
Let $\omega\in \bw^3V^*$ be a $3$-form satisfying condition \ref{GC4}.
Let  $x,y\in V^*$ be linearly independent linear forms, and let $Y_{\omega, x\w y}$ be  the residual congruence to $X_\omega$ in $\GR\cap \Lambda_\omega^{xy}$.
Then: 
\begin{enumerate}
\item\label{r:1} $Y_{\omega, x\w y}$ is aCM in its linear span;
\item\label{r:2} $X_{\omega_x}$ is a Weil divisor in $Y_{\omega, x\w y}$, and $X_{\omega_x}=Y_{\omega, x\w y}\cap G(2, V_x)$; 
\item\label{r:3} $Y_{\omega, x\w y}=\bigcup_{[a:b]\in \p^1} X_{\omega_{[a:b]}},$
where 
\begin{equation*} 
X_{\omega_{[a:b]}}=\{[L]\in G(2,V_{[a:b]})\subset\p(\bw^2V_{[a:b]})\mid\omega_{[a:b]}(L)=0\}\subset G(2,V_{[a:b]})
\end{equation*}
varies in the pencil of divisors on $Y_{\omega,x\w y}$ generated by $X_{\omega_x}$ and $X_{\omega_y}$.
\end{enumerate}
\end{prop}
\begin{proof}
The first statement follows by Gorenstein liaison, see \cite[Remark 5.3.2]{M}, since $Y_{\omega, x\w y}$ is linked to $X_\omega$, which is aCM, see 
Corollary \ref{cor:acm}, and a linear congruence is aG.

For the second statement note that $\Lambda_{\omega_x}\subset\Lambda_{\omega, x\w y}$, so
\begin{equation*}
X_{\omega_x}=\Lambda_{\omega_x}\cap G(2, V_x) \subset \Lambda_{\omega, x\w y}\cap \GR =Y_{\omega, x\w y}.
\end{equation*}
But $\dim X_{\omega_x}=\dim Y_{\omega, x\w y}-1$, so \eqref{r:2} follows. 
Similarly, 
\eqref{r:2} implies that 
$X_{\omega_{[a:b]}}$ is a Weil divisor in $Y_{\omega, x\w y}$ for any $[a:b]\in \p^1$.
Finally, since any $[L]\in Y_{\omega, x\w y}$ meets the codimension $2$ linear subspace $\p(V_{x\w y})$, it lies in $\p(\bw^2V_{[a:b]})$ for some $[a:b]\in \p^1$, i.e. $[L]\in X_{[a:b]}$. Therefore 
\begin{equation*}
Y_{\omega, x\w y}=\cup_{[a:b]\in \p^1} X_{\omega_{[a:b]}}.
\end{equation*}
\end{proof}

Notice that in view of Proposition \ref{Yred}, $X_{\omega_{[a:b]}}=Y_{\omega,x\w y} \cap G(2,V_{[a:b]})$, while $X_{\omega_{[a:b]}}\cap G(2,V_{[a':b']})$ has codimension at least $2$ in $X_{\omega_{[a:b]}}$ whenever $[a':b']\neq [a:b]$.  Since every $X_{\omega_{[a:b]}}$ is a divisor in $Y_{\omega,x\w y}$, this will allow us to conclude that 
$Y_{\omega,x\w y}$ must be singular.

\begin{lemma}\label{codim4} 
For any general $[a':b']\neq [a:b]$ the intersection $X_{\omega_{[a:b]}}\cap X_{\omega_{[a':b']}}$ has codimension $4$ in $Y_{x\w y}$.
\end{lemma}

\begin{proof} Without loss of generality  we can take $X_{\omega_{[a:b]}}=X_{\omega_x}$ and $X_{\omega_{[a':b']}}=X_{\omega_y}$, and we choose coordinates so that $x=x_0$, $y=x_1$. With the usual conventions, we can write 
\begin{align*}
\omega&=\omega_{01}+\gamma_0\w x_0 +\gamma_1\w x_1+\alpha_{01}\w x_0\w x_1,\\
L&=L_{01}+w_0\w e_0 + w_1\w e_1+ce_0\w e_1.
\end{align*}
 With reference to Theorem \ref{union} and the notations used in its proof,  we get 
\begin{align*}
\omega_0&=\omega_{01}+\gamma_1\w x_1,  &  \beta_0&=\gamma_0-\alpha_{01}\w x_1,\\
\omega_1&=\omega_{01}+\gamma_0\w x_0,  &  \beta_1&=\gamma_1+\alpha_{01}\w x_0,\\
L_0&=L_{01}+w_1\w e_1, &  v_0&=w_0-ce_1,\\
L_1&=L_{01}+w_0\w e_0, &  v_1&=w_1+ce_0.
\end{align*}

Consider the intersection
\begin{equation*}
X_{\omega_x}\cap X_{\omega_y}=(X_{\omega_x}\cap G(2, V_{x\w y})\cap( X_{\omega_y}\cap G(2,V_{x\w y})\subset X_{\omega_{xy}}.
\end{equation*}

 We apply Theorem \ref{union} \eqref{u:3} 
 to $X_{\omega_x}$ and $X_{\omega_y}$, and get that $X_{\omega_x}\cap G(2, V_{x\w y})=X_{\omega_{01}}\cap H_0$, where $H_0$ is the hyperplane  $\{\gamma_0(w_0)=0\}\subset\p(\bw^2V_x)$.  So $X_{\omega_x}\cap G(2, V_{x\w y})$ has codimension two in $X_{\omega_x}$. 
 
 Similarly $X_{\omega_y}\cap G(2, V_{x\w y})=X_{\omega_{01}}\cap H_1$, where $H_1$ is the hyperplane $\{\gamma_1(w_1)=0\}\subset\p(\bw^2V_y)$. So $X_{\omega_y}\cap G(2, V_{x\w y})$ has codimension  two in $X_{\omega_y}$.  
 
 But the hyperplanes $H_0$, $H_1$ are  distinct if $\omega$ is general, so  $X_{\omega_x}\cap X_{\omega_y}=X_{\omega_{xy}}\cap H_0 \cap H_1$ has codimension three in $X_{\omega_x}$ and hence codimension four in $Y_{x\w y}$.
\end{proof}

The hyperplanes $H_0, H_1$ are equal if, in the expression \ref{eq:3form} of $\omega$, we have $a_{0,i,j}=a_{1,i,j}$ for any $1<i<j$.  But then $\omega(e_0-e_1)=0$, so $\omega$ has rank at most $n$.

\begin{prop} \label{singular} 
Let $\omega\in \bw^3V^*$ be a $3$-form satisfying condition \ref{GC4}.   
 Let  $x,y\in V^*$ be general linearly independent linear forms, and let $Y_{\omega, x\w y}$ be  the residual congruence  to $X_\omega$ in $\GR\cap \Lambda_\omega^{xy}$.
The singular locus of the congruence $Y_{\omega, x\w y}$ is 
\begin{align*}
\Sing(Y_{\omega, x\w y})&=\bigcap_{[a:b]\in \p^1} X_{\omega_{[a:b]}}=\{[L]\in G(2,V_{x\w y})\mid\omega_{x\w y}(L)=\omega_x(L)=\omega_{y}(L)=0\}\\
&=X_\omega\cap G(2,V_{x\w y})
=X_{\omega_{x\w y}}\cap H_x\cap H_y,
\end{align*}
where $\omega_{x\w y}$ is the restriction of $\omega$ to $\Pi= \{x=y=0\}$, and $H_x$ and $H_y$ are the hyperplanes defined---using notation as in Theorem \ref{union}---by $\beta_x(L_x)$ and  $\beta_y(L_y)$, respectively. 

In particular the codimension of the singular locus  of $Y_{\omega, x\w y}$ is $4$.
\end{prop}
\begin{proof} By Bertini-Kleiman \cite{Kl} applied to $Q_{\omega\w x}\setminus \Sing(Q_{\omega\w x})$, $Y_{\omega, x\w y}$ is smooth outside $Y_{\omega,x\w y}\cap \Sing(Q_{\omega\w x})=Y_{\omega,x\w y}\cap\Lambda_{\omega_x}=X_{\omega_{x}}$ (Proposition \ref{Yred} \eqref{r:2}).

Since the same is true for any linear form $ax+by$, $Y_{\omega, x\w y}$ is smooth outside $\cap_{[a:b]\in \p^1} X_{\omega_{[a:b]}}$.  But if $y\in \cap_{[a:b]\in \p^1} X_{\omega_{[a:b]}}$, $y$ cannot be a smooth point on $Y$.   If it were, the  Weil divisors $X_{\omega_{[a:b]}}$ would be Cartier in a neighbourhood of $y$, so the intersection of any two of them would have codimension two, contradicting Lemma \ref{codim4}.

For the third equality, 
a computation in  coordinates shows that $X_\omega\cap G(2,V_{x\w y})=X_{\omega_x}\cap X_{\omega_y}$. For the fourth we notice that the second equality is equivalent to 
\begin{equation*}
\Sing (Y_{\omega, x\w y})=X_{\omega_x}\cap X_{\omega_y}\cap X_{\omega_{x\w y}};
\end{equation*}
so we conclude by Theorem \ref{union}, 
\eqref{u:3}, 
\begin{align*}
X_{\omega_x}\cap X_{\omega_{x\w y}}&=X_{\omega_{x\w y}}\cap H_x, & X_{\omega_y}\cap X_{\omega_{x\w y}}&=X_{\omega_{x\w y}}\cap H_y.
\end{align*}

\end{proof}

\begin{coro} Assume $\omega\in \bw^3V^*$ satisfies condition \ref{GC4} and let $x,y\in V^*$ be general forms.   
If $n\leq 4$, then $Y_{\omega, x\w y}$ is smooth. 

If $n\geq 5$,
 $Y_{\omega, x\w y}$ is not even factorial in any point $y\in
 \bigcap_{[a:b]\in \p^1} X_{\omega_{[a:b]}}$. In particular, a general congruence
 linked to $X_\omega$ in a linear congruence is  non-singular in
 codimension $3$, but it is not Gorenstein. 
\end{coro}



The canonical divisor of $Y_{\omega, x\w y}$ can be computed.

\begin{prop} Let $n\le 4$ and let $Y$ be a general residual congruence to
  $X_\omega$ in a linear congruence. Then
\begin{equation*}
K_Y \cong X_{\omega_{x}}-3H_Y,
\end{equation*}
as Cartier divisors.
\end{prop}
\begin{proof}

The union $X_{\omega}\cup Y$ is a linear congruence with canonical sheaf 
\begin{equation*}
\omega_{X_{\omega}\cup Y}\cong {\OO}_{X_{\omega}\cup Y}(-2).
\end{equation*}
$Y$ is smooth so, by adjunction and Proposition \ref{spanY}, 
\begin{equation*}
K_{Y}\cong -2H_{Y}-X_{\omega}\cap Y \cong-2H_{Y}-(H_{Y}-X_{\omega_{x}})\cong-3H_{Y}+X_{\omega_{x}}.
\end{equation*}
\end{proof}

\subsection{Fundamental locus of the residual congruence}

 Let $Y=Y_{x\wedge y}$ be a residual congruence to $X_\omega$, defined by some general pencil of hyperplanes
 \begin{align*}
\p(V_{[a:b]})&=\{ax+by=0\},& [a:b]&\in \p^1, 
\end{align*}
  of $\p(V)$, with base locus $\Pi=\{x=y=0\}$. Let us denote by $G$ its fundamental locus.
 Our goal here is to describe  $G$. We distinguish two cases
 according to whether $n$ is even or odd.

 \begin{theorem}\label{Gomega}
   Assume $\omega$ satisfies \ref{GC4} and let $Y$ be the residual congruence to $X_\omega$, associated
   with a general pencil of 
   hyperplanes of $\p(V)$, with base locus $\Pi$. 
Then the fundamental locus $G$ of
   $Y$ is
   \begin{enumerate}[i)]
   \item a hypersurface of degree $m$ in $\p(V)$ containing $\Pi$ and $F_\omega$, when $n=2m+1$;
   \item the union of $\Pi$ and of a subvariety $G_0$ of codimension $3$ and degree
     $(2 m^{3}-3 m^{2}-5 m+12)/6=2\binom{m+1}{3}-\binom{m+1}{2}+2$, contained in the degree $m$ hypersurface $F_\omega$, when $n=2m$.
   \end{enumerate}
 \end{theorem}

The proof occupies the rest of this section.
To start with, let us consider the reducible linear congruence
$R = X_\omega \cup Y$ and the universal line
$\p(\UU^*|_{R})$ over
$R$. Following the classical description of 
degeneracy loci of webs of twisted $2$-forms, we write $\p(\UU^*|_{R})$ 
as a  reducible Palatini scroll, i. e. degeneracy locus of a morphism
(cf. for instance
\cite{Ott2, FF}):
\begin{equation}
  \label{pala}
\OO_{\p(V)}^{n-1} \to \Omega_{\p(V)}(2).  
\end{equation}

Let $\FF$ be the cokernel of this map. We have $\p(\FF) \cong \p(\UU^*|_{R})$.
Also, the map \eqref{pala} corresponds to the choice of $n-1$ independent hyperplanes
containing $X_\omega$. So it factors through the map $\phi_\omega\colon
T_{\p(V)}(-1) \to \Omega_{\p(V)}(2)$ defining $F_\omega$.

For the remaining part of the proof, we distinguish two cases as in the
statement according to the parity of $n$.

\subsubsection{If $n=2m+1$}

Recall the resolution \eqref{f-odd} of $F_\omega$. Since the map in
\eqref{pala} factors through $\phi_\omega$, we get an exact
commutative diagram:
\begin{equation} \label{diagram-defining-T}
\xymatrix@-2.9ex{
& 0 \ar[d]  & 0 \ar[d]  & \\
& \OO_{\p(V)}^{n-1} \ar[d] \ar@{=}[r] & \OO_{\p(V)}^{n-1} \ar[d] \\
0 \ar[r] & \im(\phi_\omega) \ar[d] \ar[r] & \Omega_{\p(V)}(2) \ar[r]
\ar[d] & \II_{F_\omega /  \p(V)}(m) \ar[r]  \ar@{=}[d]& 0\\
0  \ar[r]  & \TTT \ar[r] \ar[d]& \FF \ar[r] \ar[d]& \II_{F_\omega /  \p(V)}(m)\ar[r] & 0 \\
& 0  & 0   & 
}
\end{equation}
The coherent sheaf $\TTT$, defined by the diagram, is thus the
torsion part of $\FF$. Note that the exact sequence on the bottom
line, after projectivisation, accounts for the exact sequence 
\begin{equation*}
0 \to \II_{X_\omega / R} \to \OO_{R} \to \OO_{X_\omega}
\to 0
\end{equation*}
extracted from \eqref{Schubert hyperplane}, once taken the universal
lines above $R$ and $X_\omega$.
Then the fundamental locus  $G$ is the support of the sheaf
$\TTT$, which is a hypersurface of degree $m$, by a straightforward
Chern class computation.

To check that $G$ contains $\Pi$, we observe that the cokernel of the induced map $\OO_{\p(V)}^{n-1} \to
T_{\p(V)}(-1)$ is clearly $\II_{\Pi/\p(V)}(1)$. So 
 the leftmost part of
the resolution \eqref{f-odd}, combined with the left column of the
above diagram, yields another commutative exact diagram:
\begin{equation*}
\xymatrix@-2.9ex{
& & 0 \ar[d]  & 0 \ar[d]  & \\
& & \OO_{\p(V)}^{n-1} \ar[d] \ar@{=}[r] & \OO_{\p(V)}^{n-1} \ar[d] \\
0 \ar[r] & \OO_{\p(V)}(1-m)   \ar@{=}[d] \ar[r] & T_{\p(V)}(-1)\ar[r] \ar[d] & \im(\phi_\omega) \ar[d] \ar[r]  & 0\\
0 \ar[r] &  \OO_{\p(V)}(1-m) \ar[r] & \II_{\Pi/\p(V)}(1) \ar[r] \ar[d]& \TTT \ar[r]\ar[d]& 0 \\
&& 0  & 0   & 
}
\end{equation*}
The bottom row shows that $\Pi$ lies in the support of $\TTT$,
i.e. $G$ contains $\Pi$.

To show that $G$ contains $F_\omega$, we look back at
Diagram \eqref{diagram-defining-T}. First note that, dualising the
middle column, we easily get $\EExt^i_X(\FF,\OO_X)=0$ for all $i >
1$. On the other hand, $\TTT$ is  supported on the hypersurface $G$ of degree $m$.
Therefore, 
\begin{equation*}
\EExt^1_X(\TTT,\OO_X) \simeq \HH om_G(\TTT,\OO_G(m))
\end{equation*}
is also supported on $G$. 
Also, since $F_\omega$ is a Gorenstein subvariety of codimension $3$
in $\p(V)$ with $\omega_{F_\omega} \simeq \OO_{F_\omega}(-3)$
we have that
\begin{equation*}
\EExt^2_X(\II_{F_\omega/\p(V)}(m),\OO_X)\simeq \OO_{F_\omega}(m-1).
\end{equation*}
Therefore,
taking duals of the bottom row of Diagram \eqref{diagram-defining-T}
we end up with a surjection:
\begin{equation*}
\HH om_G(\TTT,\OO_G(m)) \twoheadrightarrow \OO_{F_\omega}(m-1).
\end{equation*}
In particular, $F_\omega$ is contained in the
support of $\HH om_G(\TTT,\OO_G(m))$, i.e. in $G$.
\subsubsection{If $n=2m$}

We consider the resolution \eqref{linked} of
$\II_{Y/\GG}$ and twist with $\OO_\GG(1)$.
Recalling the setup of \S \ref{proj bundles},
 we lift this resolution to the universal lines
over $R$ and $Y$ by
pulling back via $\lambda$. Finally, we take direct
image in $\p(V)$ via $\mu$. 
The
universal line over $Y$ is the projectivisation of
$\mu_*(\lambda^*(\OO_{Y}(1))$, cf.  \cite{P2}.

Recall that $\mu_*(\lambda^*(\OO_\GG(1))) \cong \Omega_{\p(V)}(2)$
and that $\mu_*(\lambda^*(\OO_\GG)) \cong \OO_{\p(V)}$. 
The remaining terms of the resolution of
$\mu_*(\lambda^*(\OO_{Y_\omega}(1))$ are
computed via Bott's theorem, which provides a long exact sequence:
\begin{equation}
  \label{longg}
  0 \to \OO_{\p(V)}(1-m) \to \OO_{\p(V)}^{n} \xrightarrow{g} \Omega_{\p(V)}(2) \to \II_{G/\p(V)}(m) \to 0,  
\end{equation}
where $\mu_*(\lambda^*(\OO_{Y}(1)) \cong \II_{G/\p(V)}(m)$.

The middle map $g$ in the above sequence is just the result of applying $\mu_*
\circ \lambda^*$ to the map $\OO_\GG^n \to \OO_\GG^n(1)$ expressing the generators of
$\II_{Y/\GG}(1)$.
By construction of the resolution of \eqref{linked}, the map
\eqref{pala} then fits into the long exact sequence above to give the
exact commutative diagram
\begin{equation*}
\xymatrix@-2.9ex{
& 0 \ar[d]  & 0 \ar[d]  & \\
& \OO_{\p(V)}^{n-1} \ar[d] \ar@{=}[r] & \OO_{\p(V)}^{n-1} \ar[d] \\
0 \ar[r] & \im(g) \ar[d] \ar[r] & \Omega_{\p(V)}(2) \ar[r] \ar[d] & \II_{G /  \p(V)}(m) \ar[r]  \ar@{=}[d]& 0\\
0  \ar[r]  & \TTT \ar[r] \ar[d]& \FF \ar[r] \ar[d]& \II_{G /  \p(V)}(m)\ar[r] & 0 \\
& 0  & 0   & 
}
\end{equation*}
Again, $\TTT$ is the torsion part of $\FF$, but this time its support
is just $F_\omega$. Actually from the leftmost part of \eqref{longg} we
can see that $\TTT \cong \OO_{F_\omega}$. Indeed, we have the diagram
\begin{equation*}
\xymatrix@-2.9ex{
& & 0 \ar[d]  & 0 \ar[d]  & \\
& & \OO_{\p(V)}^{n-1} \ar[d] \ar@{=}[r] & \OO_{\p(V)}^{n-1} \ar[d] \\
0 \ar[r] & \OO_{\p(V)}(1-m)   \ar@{=}[d] \ar[r] & \OO_{\p(V)}^n\ar[r] \ar[d] & \im(\phi_\omega) \ar[d] \ar[r]  & 0\\
0 \ar[r] &  \OO_{\p(V)}(1-m) \ar[r] & \OO_{\p(V)} \ar[r] \ar[d]& \TTT \ar[r]\ar[d]& 0 \\
&& 0  & 0   & 
}
\end{equation*}
and the form of degree $m-1$ appearing in the bottom row must define $F_\omega$.
Furthermore, recalling that the map $\OO_{\p(V)}^{n-1} \to \Omega_{\p(V)}(2)$ factors
through $\phi_\omega$, i. e. through $T_{\p(V)}(-1)$, we get an exact sequence
\begin{equation*}
0 \to \II_{\Pi/\p(V)}(1) \to \FF \to \C_\omega \to 0,
\end{equation*}
where $\II_{\Pi/\p(V)}(1)$ again appears as cokernel of
$\OO_{\p(V)}^{n-1} \to T_{\p(V)}(-1)$.

Also, the torsion part $\OO_{F_\omega}$ of $\FF$
goes to zero under
composition to $\II_{G/\p(V)}(m)$, so we finally get an exact
commutative diagram:
\begin{equation*}
\xymatrix@-2.9ex{
&& 0 \ar[d] & 0 \ar[d] & \\
&& \II_{\Pi/\p(V)}(1)  \ar[d] \ar@{=}[r] & \II_{\Pi/\p(V)}(1) \ar[d] \\ 
0 \ar[r] & \OO_{F_\omega} \ar@{=}[d] \ar[r] & \FF \ar[r] \ar[d] & \II_{G/\p(V)}(m) \ar[d] \ar[r] & 0 \\
0\ar[r] & \OO_{F_\omega} \ar[r] & \C_\omega \ar[r] \ar[d] & \II_{\bar  G/F_\omega}(m) \ar[r] \ar[d] & 0\\
& & 0 & 0
}
\end{equation*}

Here $\bar G$ is defined by the bottom row of the diagram as the zero
locus of a global section of $\C_\omega$ in $F_\omega$. The subvariety
$\bar G$ of $F_\omega$ has thus codimension $2$ in $F_\omega$ (and thus codimension
$3$ in $\p(V))$ by a standard argument relying on the vanishing
$H^0(F_\omega,\C_\omega(-1))=0$ and on the fact that
$\Pic(F_\omega)$ is generated by $\OO_{F_\omega}(1)$.
A direct Chern class computation shows that $\deg(\bar G)=(m+1) (2 m^{2}-5 m+6)/6$.

Our goal is to describe the component $G_0$, the closure of $G
\setminus \Pi$. 
 To this end, let us use the rightmost column of the diagram to describe $\bar
G$ in more detail. The inclusion $\II_{\Pi/\p(V)}(1) \subset
\II_{\Pi/\p(V)}(m)$ factors through $\II_{G/\p(V)}(m)
\subset \II_{\Pi/\p(V)}(m)$. 

On the other hand, if 
we write 
\begin{equation}\label{eq:g1}
G_1=\Pi \cap F_\omega
\end{equation}
and note that $G_1$ is
a codimension $2$ subvariety of 
$F_\omega$ of degree $m-1$, then there is an obvious exact sequence:
\begin{equation*}
0 \to \II_{\Pi/\p(V)}(1) \to \II_{\Pi/\p(V)}(m) \to \II_{G_1/F_\omega}(m) \to 0.
\end{equation*}

We are now in position to compute the
degree of $G_0$. Indeed, we have a last exact commutative diagram:
\begin{equation*}
\xymatrix@-2.9ex{
& 0 \ar[d] & 0 \ar[d] & 0 \ar[d] \\
0 \ar[r] & \II_{\Pi/\p(V)}(1)\ar@{=}[d] \ar[r] & \II_{G / \p(V)}(m)\ar[d]\ar[r]&\II_{\bar G/F_\omega}(m)\ar[r] \ar[d]& 0 \\
0 \ar[r] & \II_{\Pi/\p(V)}(1) \ar[r] & \II_{\Pi / \p(V)}(m)\ar[r]\ar[d]&\II_{G_1/F_\omega}(m)\ar[r] \ar[d]& 0 \\
& & \II_{\Pi/G}(m)\ar[d] \ar^{\cong}[r] & \II_{G_1/G}(m)\ar[d]\\
& & 0 & 0
}
\end{equation*}

Since $\Pi$ and $G_0$ are irreducible components of $G$, the ideal $\II_{\Pi/G}$ is
supported at $G_0$ and is torsion-free of rank $1$ over $G_0$. The
isomorphism in the lower-right corner induced by the diagram expresses
$\bar G$ as union of two irreducible components $G_0$ and
$G_1$, of codimension $2$ in $F_\omega$.
So the degree of $G_0$ is computed
by $\deg(G_0)=\deg(\bar G)-\deg(G_1)$.
Plugging in the formula for $\deg(\bar G)$ and $\deg(G_1)=m-1$, we get
the desired expression of $\deg(G_0)$.

\begin{flushright}$\Box$
\end{flushright}

From this proof we in particular get

\begin{coro} The component $G_0$ of the fundamental locus of $Y$ is obtained as residual in the zero-locus
     $\bar G$ of a global section of $\C_\omega$ over $F_\omega$ with respect to 
      $G_1=\Pi \cap F_\omega$.
\end{coro}

Finally, we can give a geometric description of the congruence $Y$ as multisecant lines to its fundamental locus, when $n=2m$, 
as we did in Theorem \ref{thm:pippo}.

\begin{theorem}\label{pippo2}
 Let $\dim(V)=n+1=2m+1$ and assume $\omega\in \bw^3V^*$ satisfies \ref{GC4} and let $Y$ be the residual congruence to $X_\omega$ associated
   with a general pencil of hyperplanes of $\p(V)$, with base locus $\Pi$. 
Then $Y$ is the closure of the family of $(\frac{n-2}{2})$-secant lines of $G_0$ that also meet $\Pi$, where $G_0\cup \Pi$ is the fundamental locus of $Y$. 
\end{theorem}

\begin{proof}
By Proposition \ref{Yred}, \eqref{r:3},  
\begin{equation*}
Y=\bigcup_{[a:b]\in \p^1} X_{\omega_{[a:b]}},
\end{equation*}
where $X_{\omega_{[a:b]}}\subset G(2,V_{[a:b]})$ is defined by $\omega_{[a:b]}(L)=0$.  But  $X_{\omega_{[a:b]}}$  
is a congruence defined by a $3$-form on the 
$n$-dimensional vector space $V_{[a:b]}$ where $n=2m$ is even, so Theorem \ref{thm:pippo} applies.   We deduce that the lines of 
$X_{\omega_{[a:b]}}$ are $(\frac{n-2}{2})$-secants to its fundamental locus: call it $F_{[a:b]}$. 

Moreover, by Proposition \ref{spanY}, 
the lines of $Y$ are contained in the Schubert hyperplane of the lines meeting $\Pi$.  Since $\Pi$ is a hyperplane 
in $\p(V_{[a:b]})$, all the lines of  $X_{\omega_{[a:b]}}$ meet $\Pi$ and  any line that meets $\Pi$ is contained in a $\p(V_{[a:b]})$.  We infer that 
\begin{align*}
\left(\bigcup_{[a:b]\in \p^1}F_{[a:b]}\right)\setminus \Pi&= G_0\setminus \Pi, &   \bigcup_{[a:b]\in \p^1}F_{[a:b]}&\subset G_0
\end{align*}
and that any  $(\frac{n-2}{2})$-secant line to $G_0$ that meets $\Pi$ belongs to one of the  $X_{\omega_{[a:b]}}$ and 
is a line of $Y$. 
\end{proof}

\begin{rema}
In the proof of the preceding theorem, we have proven that---outside $\Pi$---$G_0$ coincides with $\bigcup_{[a:b]\in \p^1}F_{[a:b]}$, 
where, as in the proof, $F_{[a:b]}$ denotes the fundamental locus of the congruence  $X_{\omega_{[a:b]}}$; in other words, 
\begin{equation}\label{eq:go}
G_0=\overline{\bigcup_{[a:b]\in \p^1}F_{[a:b]}}.
\end{equation}
\end{rema}

We can also analyse what happens to $G_0$ in $\Pi$: 

\begin{prop}\label{prop:fxy}
With notations as above, if $n=2m$, then 
\begin{equation*}
G_0\cap \Pi=G_1\cap F_{\omega_{x\w y}},
\end{equation*}
where  $F_{\omega_{x\w y}}$ is the fundamental locus of the congruence $X_{\omega_{x\w y}}$. 

In particular, $G_0\cap \Pi$ is an improper intersection,
it has codimension $2$ in $\Pi$ and degree $(m-1)(m-2)$,  and is the complete intersection of  
$G_1=F_\omega\cap \Pi$ of degree $m-1$ and $F_{\omega_{x\w y}}$ of degree $m-2$. 
\end{prop}

\begin{proof}
Let $P\in G_1\cap F_{\omega_{x\w y}}$; as usual, up to a change of coordinates, we can suppose that $P=[1,0,\dotsc,0]=[e_0]$, we fix a basis 
$(e_0,\dotsc,e_n)$ and dual basis $(x_0,\dotsc,x_n)$ such that $x=x_{n-1}$ and 
$y=x_n$, i.e. $\Pi^\perp=\gen{x_0,\dotsc,x_{n-2}}$ and we can write uniquely 
\begin{equation*}
\omega =\omega_{x\w y}+\beta_x\w x+\beta_y\w y+ z\w x\w y, 
\end{equation*}
where $\omega_{x\w y}\in \bw^3 \gen{x_0,\dotsc,x_{n-2}}$, $\beta_x,\beta_y\in  \bw^2 \gen{x_0,\dotsc,x_{n-2}}$, 
and $z\in \gen{x_0,\dotsc,x_{n-2}}$. 

Since---see Equation \eqref{eq:g1}---$P\in G_1=F\cap \Pi$,
 we have that the skew-symmetric matrix $M_\omega$ at the point $P$, defined as in \eqref{eq:mop},
\begin{equation*}
M_\omega(P)=\left( (-1)^{i+j-1}(a_{0,i,j}) \right)_{\substack{i=0,\dotsc,n \\ j=0,\dotsc,n}},
\end{equation*}
has rank (at most) $n-2$. The matrix $M_{\omega_{x\w y}}$ at the point $P$ is 
\begin{equation*}
M_{\omega_{x\w y}}(P)=\left( (-1)^{i+j-1}(a_{0,i,j}) \right)_{\substack{i=0,\dotsc,n-2 \\ j=0,\dotsc,n-2}},
\end{equation*}
i.e. it is the submatrix of $M_\omega(P)$ obtained removing  the last $2$ rows and $2$ columns, and, since $P\in F_{\omega_{x\w y}}$, 
it has rank (at most) $n-4$.

In other words, $P\in G_1\cap F_{\omega_{x\w y}}$ if and only if $\rk M_\omega(P)\le n-2$ and $\rk M_{\omega_{x\w y}}(P)\le n-4$, with 
 $\rk M_\omega(P)=\rk M_{\omega_{x\w y}}(P)+2=n-2$ (if $P$ is general). 

It is obvious, from Proposition \ref{prop:fonda}, that $\deg G_1\cap F_{\omega_{x\w y}}=(m-1)(m-2)$ and that 
$G_1\cap F_{\omega_{x\w y}}$ has codimension $2$ in $\Pi$.

On the other hand, if $P\in G_0\cap \Pi$, by Proposition \ref{Yred}, \eqref{r:3}, 
it belongs to infinitely many lines of the congruence $X_{\omega_{[a:b]}}$ 
for some $[a:b]\in \p^1$. Up to a change of coordinates, we can suppose that $[a:b]=[1:0]$ and  $P=[1,0,\dotsc,0]=[e_0]$, so
 by Equation \eqref{eq:go} and Theorem \ref{pippo2} 
\begin{equation*}
M_{\omega_{x}}(P)=\left( (-1)^{i+j-1}(a_{0,i,j}) \right)_{\substack{i=0,\dotsc,n-1 \\ j=0,\dotsc,n-1}},
\end{equation*}
i.e. the submatrix of $M_\omega(P)$ without the last row and column, has rank (at most) $n-4$. If follows that $M_\omega(P)$
has rank at most $n-2$ and that $M_{\omega_{x\w y}}(P)$ has at most the rank of $M_{\omega_{x}}(P)$, which is at most $n-4$ therefore 
$P\in G_1\cap F_{\omega_{x\w y}}$. Therefore, $G_0\cap \Pi\subset G_1\cap F_{\omega_{x\w y}}$, and they have the same dimension.

To finish the proof, it is sufficient to see that 
$\deg G_0\cap \Pi=\deg G_0-\deg  F_{[a:b]}=2\binom{m+1}{3}-\binom{m+1}{2}+2-(\frac{1}{4}\binom{2m-2}{3}+1)=(m-1)(m-2)$, by Theorem 
\ref{Gomega} and Proposition \ref{prop:fonda}. 
\end{proof}

\subsection{Examples}
Let $\omega$ be a general $3$-form, and $Y$ a general residual congruence to $X_\omega$.

\subsubsection{n=3} $Y$ is  the congruence of lines of a $2$-plane $G$,  the plane spanned by the point $F_\omega$ and the line $\Pi$.  The linear congruence $X_\omega\cup Y$ is obtained intersecting $\GR$ with a $\p^3$ tangent to $\GR$ along a line, which is $X_\omega\cap Y$.

\subsubsection{n=4} We recall (see Example \ref{ex2}) that $X_\omega$ is a linear complex of lines in a hyperplane 
$F_\omega\subset \p^4$.  A general $Y$ is obtained by choosing a general $2$-plane $\Pi\subset \p^4$, intersecting $F_\omega$ is a line $G_1$. 
Then $G_0$ is a line skew to $G_1$, and $Y$, which is smooth of degree $3$, is the family of lines meeting $\Pi$ and $G_0$, i.e. $Y=\p^2\times \p^1$.

 $X_\omega\cap Y$ is a Schubert hyperplane section of a linear complex of lines in $\p^3$, so it is the linear congruence of the lines meeting two skew lines: $G_0$ and $G_1=\Pi\cap F_\omega$.

\subsubsection {n=5}
As explained in Example \ref{ex3}, $X_\omega$ is the congruence of the lines meeting two skew planes $\alpha$, $\beta$. 
In this case $\Pi$ is a $3$-space contained in $\p(V)$. If we identify $\p(V)$ with $P_\omega\subset \p(\bw^2V^*)$, it is generated by its intersection with 
$\Sing(\GG^*)\simeq \GG$, which is the union of two $2$-planes, 
each one representing the incidence condition to a plane in $\p^5$. 
The choice of $\Pi$ identifies a point $A\in\alpha$ and similarly $B\in\beta$, and the line $\ell$ joining $A$ and $B$. 
The residual congruence $Y$ 
is a fourfold of degree $8$ in $\p^{9}$, of multidegree $(0,2,1)$, representing a family of lines contained in a quadric $G$.
$G$ contains the two planes $\alpha, \beta$, $\Pi$ and all the fundamental loci $F_{[a:b]}$ of the pencil of congruences $X_{\omega_{[a:b]}}$, which are also $3$-spaces. Therefore $G$ is a quadric cone of rank $4$ with vertex a line. The spaces $F_{[a:b]}$ belong to one family, and $\Pi$ belongs to the other family.
The lines of $Y$ form a subfamily of dimension $4$ of  the lines contained in one of the two  families of $3$-spaces contained in $G$.  The singular locus of $Y$ is the point representing the line $\ell$, which is also the vertex of $G$. The intersection $X_\omega\cap Y$ represents the lines meeting $\alpha$, $\beta$ and $\ell$, so it  is a tangent hyperplane section of $\p^2\times \p^2$, and it results to be the blow up of $\p^3$ at the two points $A,B$. It is embedded in $\p^7$ with the linear system of quadrics.

\subsubsection {n=6}
 $X_\omega$ is the $G_2$-variety, representing a family of lines in a smooth quadric $F_\omega$. Fixed $\Pi$, $Y$ is the congruence of lines that are secant to a rational normal scroll $G_0$ of degree $4$ and dimension $3$ in $\p^6$ and intersect the $4$-space $\Pi$. $G_0$ is smooth because it contains pairs of skew planes, the fundamental loci $F_{[a:b]}$ of the congruences $X_{\omega_{[a:b]}}$.
 The singular locus of $Y$ is the intersection $X_{\omega_{x\w y}}\cap H_0\cap H_1$, where $H_0, H_1$ are hyperplanes. But $X_{\omega_{x\w y}}$ is a complex of lines in a $\p^3\subset \Pi$, i.e. a $3$-dimensional quadric. We deduce that $\Sing Y$ is a conic, the lines of a quadric surface contained in $\Pi$.
 Now consider $G_0\cap \Pi=G_0\cap \Pi\cap F_\omega=G_0\cap G_1$. To understand it, consider first $G_0\cap \p(V_{[a:b]})$: this is a quartic surface containing the two planes whose union is $F_{[a:b]}$, hence $(G_0\cap \p(V_{[a:b]})\setminus F_{[a:b]}$ is a quadric contained in $\Pi$. Moreover $G_0\cap \Pi=G_0\cap \p(V_x)\cap \p(V_y)\subset \Pi \cap F_\omega=G_1$. Since the missing quadric in $G_0\cap \p(V_{[a:b]})$ is $G_1\cap F_{\omega_{x\w y}}$, it is the union of the lines of $\Sing Y$ .  
 
 Thus  $G_0\cap\Pi$ is a quadric surface (cf. Proposition \ref{prop:fxy}). 

\subsubsection {n=7}
The lines of $Y$ are contained in a cubic hypersurface, containing $\Pi$ and a pencil of projections of $\p^2\times \p^2$. 
The singular locus of $Y$ is $(\p^2\times \p^2)\cap H_x\cap H_y$, a del Pezzo surface of degree $6$ in $\p^6$. 

\subsubsection {n=8}
For $n=8$ $Y$ is formed by the trisecant lines of a $5$-dimensional variety $G_0$ of degree 12 that meets $\Pi=\p^6$ in a $4$-fold of degree 6.  This $4$-fold is a complete intersection of the cubic hypersurface $F_\omega\cap \Pi$ and the quadric hypersurface $F_{\omega_{x\w y}}\subset \Pi$. 

\subsubsection {n=9}
$Y$ represents a family of dimension $8$ of lines in a quartic hypersurface containing a $\p^7$  
and a pencil of Peskine varieties. $\Sing Y$ is formed by the trisecant lines to a del Pezzo surface of degree $6$ in $\p^6$.


\section{The tables} \label{tables}


In  tables \ref{t:1} and \ref{t:2} we collect some  geometrical properties of the congruences we have studied. The notations are as usual: 
\begin{itemize}
\item $\omega$ is a general $3$-form in $n+1$ variables;
\item $X_\omega\subset \GR$ is the congruence of lines where $\omega$ vanishes;
\item $F_\omega$ is the fundamental  locus of $X_\omega$;
\item $Y$ is the residual congruence of $X_\omega$ in a general linear section of $\GR$   of codimension $n-1$; 
\item $G$ is the fundamental  locus of $Y$. 
\end{itemize}

 In the last column of table \ref{t:1} we write the dimension of $\bw^3V/\GL(n+1)$, the number of moduli of our construction.

\begin{table}[htbp]
\begin{center}
\textup{
\begin{tabular}{| p{0.5cm} | p{2.9cm} | p{2cm} | p{1cm} | p{3.2cm} | p{2cm}|
}
\hline
$n$ & $X_\omega$ &  $\multdeg X_\omega$ & $\deg X_\omega$ &  $F_\omega$ &  moduli
\\[6pt]
 \hline
 \hline
$3$ & $\p^2$ & $(1,0)$ & $1$ & $\{*\}$ &0 
        \\[6pt]
 \hline
$4$ & $\GG(1,3)\cap\p^4$& $(0,1)$ & 2& $\p^3$ & 0
\\[6pt]
\hline
$5$ & $\p^2\times \p^2$ & $(1,1,1)$ & $6$ &  $\p^2\cup\p^2$ & 0
       \\[6pt]
\hline
$6$ & $G_2$ & $(0,2,2)$ & $18$ &  smooth quadric of $\p^6$ &0
       \\[6pt]
\hline
$7$ & trisecant lines of $F_\omega$ & $(1,2,4,2)$ & 57 &   general projection of $\p^2\times\p^2$ of degree $6$ &0
       \\[6pt]
\hline
$8$ & $7$-dimensional family of lines in $F_\omega$ & $(0,3,6,6)$ & 186 & Coble cubic hypersurface in $\p^8$ singular along an Abelian surface& 3
       \\[6pt]
\hline
$9$ & four-secant lines of $F_\omega$ & $(1,3,9,12,6)$ & 622 & Peskine variety of degree $15$ &20
\\[6pt]
\hline
\end{tabular}
}
\end{center}
\caption{}\label{t:1}
\end{table}

\begin{table}[htbp]
\begin{center}
\textup{
\begin{tabular}{| p{0.5cm} | 
p{2.5cm} | 
p{2.0cm} | p{1.0cm} | p{3cm} | p{2.5cm} |
}
\hline
$n$ & 
$Y$ & 
 $\multdeg Y$  & $\deg Y$ & $G$\newline ($=G_0\cup \Pi$, $n$ even) &$\Pi\cap G_0$, $n$ even
\\[6pt]
 \hline
 \hline
$3$ & $\p^2$  & 
$(0,1)$ & $1$ &   $\p^2$  &
        \\[6pt]
 \hline
$4$ & 
$\p^1\times \p^2$  & 
$(1,1)$ & 3 &  line $G_0$, plane $\Pi$  & $\emptyset$
       \\[6pt]
\hline
$5$ & 
$4$-fold of degree $8$ in $\p^9$, with a singular point & 
$(0,2,1)$ &8 & quadric of rank $4$ &
       \\[6pt]
\hline
$6$ & 
secant lines of rational normal $3$-fold scroll meeting a $\p^4$ & 
$(1,2,3)$ & 24 & rational normal scroll $G_0$ of dim $3$ and degree $4$,   $\Pi=\p^4$& quadric surface
       \\[6pt]
\hline
$7$ & $6$-fold of degree $75$ singular along a del Pezzo surface of degree $6$
  & 
$(0,3,5,3)$ & 75 & cubic hypersurface containing a $\p^5$ and $\pi(\p^2\times \p^2)$&
       \\[6pt]
\hline
$8$ & $3$-secant lines to $G_0$ that meet a $\p^6$
  &
$(1,3,8,8)$ & 243 & $5$-fold $G_0$ of degree $12$, $\Pi=\p^6$ & complete intersection (2,3) fourfold of degree $6$ in $\p^6$
       \\[6pt]
\hline
$9$ & $8$-fold of degree $808$ singular along a $4$-fold of degree $57$
  &
$(0,4,11,16,8)$  & 808 & quartic hypersurface containing a $\p^7$ and a Peskine variety&
\\[6pt]
\hline
\end{tabular}
}
\end{center}
\caption{}\label{t:2}
\end{table}

\providecommand{\bysame}{\leavevmode\hbox to3em{\hrulefill}\thinspace}
\providecommand{\MR}{\relax\ifhmode\unskip\space\fi MR }
\providecommand{\MRhref}[2]{%
\href{http://www.ams.org/mathscinet-getitem?mr=#1}{#2}
}
\providecommand{\href}[2]{#2}

\end{document}